\documentclass[11pt]{article}

\usepackage{amssymb}
\usepackage{graphicx}
\usepackage{amsmath}
\usepackage{verbatim}
\usepackage{setspace}
\usepackage{textpos}
\usepackage{changepage}
\usepackage{url}

\usepackage[title]{appendix}

\usepackage{hyperref}
\usepackage{color}

\usepackage{amsthm}

\usepackage[round]{natbib}

\newtheorem{assumption}{Assumption}

\newtheorem{proposition}{Proposition}

\newtheorem{lemma}{Lemma}

\newtheorem{corollary}{Corollary}
\newtheorem{example}{Example}

\newcommand{\pstar}{\mathbb{P}_{\theta^*}}
\newcommand{\estar}{\mathbb{E}_{\theta^*}}
\newcommand{\ve}{\varepsilon}

    \DeclareMathOperator*{\argmax}{\arg\!\max}

\begin{document}

\bibliographystyle{asa}

{
\singlespacing
\title{\textbf{Asymptotic Properties of the Maximum Likelihood Estimator
in Regime Switching Econometric Models\thanks{This research was supported by the Natural Science and Engineering Research Council of Canada and JSPS KAKENHI Grant Number JP17K03653.}}
\author{Hiroyuki Kasahara\\
Vancouver School of Economics\\
University of British Columbia\\
hkasahar@mail.ubc.ca \and Katsumi Shimotsu\\
Faculty of Economics \\
University of Tokyo\\
shimotsu@e.u-tokyo.ac.jp
}}
%\date{August 2016} 
\maketitle
}
\begin{abstract}
Markov regime switching models have been widely used in numerous empirical applications in economics and finance. However, the asymptotic distribution of the maximum likelihood estimator (MLE) has not been proven for some empirically popular Markov regime switching models. In particular, the asymptotic distribution of the MLE has been unknown for models in which some elements of the transition probability matrix have the value of zero, as is commonly assumed in empirical applications with models with more than two regimes. This also includes models in which the regime-specific density depends on both the current and the lagged regimes such as the seminal model of \citet{hamilton89em} and switching ARCH model of \citet{hamiltonsusmel94joe}. This paper shows the asymptotic normality of the MLE and consistency of the asymptotic covariance matrix estimate of these models.
\end{abstract}

Keywords: asymptotic distribution; autoregressive conditional heteroscedasticity; maximum likelihood estimator; Markov regime switching

JEL classification numbers: C12, C13, C22

\section{Introduction}
Since the seminal contribution of \citet{hamilton89em}, Markov regime switching models have become a popular framework for applied empirical work because they can capture the important features of time series such as structural changes, nonlinearity, high persistence, fat tails, leptokurtosis, and asymmetric dependence \citep[e.g.,][]{evanswachtel93jme, hamiltonsusmel94joe, gray96jfe, simszha06aer, inoueokimoto08jjie, angbekaert02rfs, okimoto08jfqa, daisingletonyang07rfs}. Surveys of applications of Markov regime switching models in economics and finance are provided by, for example, \citet{hamilton08palgrave, hamilton16hdbk} and \citet{angtimmermann12annual}.

Consider the Markov regime switching model defined by a discrete-time stochastic process $\{Y_k,X_k\}$ written as
\begin{equation}\label{MS-model}
Y_k = f_{\theta}(Y_{k-1},\ldots,Y_{k-s},X_k;\ve_k),
\end{equation}
where $\{\ve_k\}$ is an independent and identically distributed sequence of random variables, $\{Y_k\}$ is an inhomogeneous $s$-order Markov chain on a state space $\mathcal{Y}$ conditional on $X_k$ such that the conditional distribution of $Y_k$ only depends on $X_k$ and the lagged $Y$'s, $X_k$ is a first-order Markov process in a state space $\mathcal{X}$, and $f_\theta$ is a family of functions indexed by a finite-dimensional parameter $\theta\in\Theta$. In (\ref{MS-model}), the Markov chain $\{X_k\}$ is not observable.

Surprisingly, the asymptotic distribution of the maximum likelihood estimator (MLE) of the Markov regime switching model (\ref{MS-model}) has not been fully established in the existing literature. \citet{brr98as} and \citet{jp99as} derive the asymptotic normality of the MLE of hidden Markov models in which the conditional distribution of $Y_k$ depends on $X_k$ but not on the lagged $Y$'s. For hidden Markov models and Markov regime switching models with a finite state space, the consistency of the MLE has been proven by \citet{leroux92spa}, \citet{francq98stat}, and \citet{krishnamurthy98jtsa}.

In an influential paper, \citet{dmr04as} [DMR hereafter] establish the consistency and asymptotic normality of the MLE in autoregressive Markov regime switching models (\ref{MS-model}) with a nonfinite hidden state space $\mathcal{X}$ under two assumptions. First, DMR assume that the conditional distribution of $Y_k$ does not depend on the lagged $X_k$'s. Specifically, on page 2259, DMR assume that
\begin{quote}
for each $n \geq 1$ and given $\{Y_k\}_{k=n-s}^{n-1}$ and $X_n$, $Y_n$ is conditionally independent of $\{Y_k\}_{k=-s+1}^{n-s-1}$
and $\{X_k\}_{k=0}^{n-1}$. 
\end{quote}
Second, DMR assume in their Assumption A1(a) that the transition density of $X_k$ is bounded away from 0.

These two assumptions together rule out regime switching models in which some elements of the transition probability matrix take the value of zero. However, empirical researchers often assume that some elements of the transition probability matrix are identically equal to zero when they estimate regime switching models with more than two regimes. For example, \citet{kimmorleypiger05jae} estimate a three-regime model of U.S. GDP growth in which some elements of the transition probability matrix are restricted to be zero because ``the three regimes corresponding to expansion, recession and recovery always occur in that order'' (see also \citet{boldin96snde}). Similarly, \citet{dahlquistgray00jie} estimate a three-regime model of short-term interest rates for France and Italy while restricting some elements of the transition probability matrix to be zero and \citet{davidveronesi2013jpe} estimate a six-regime model of inflation, earnings growth, consumption growth, S\&P 500 P/E ratios, the three-month Treasury bill rate, and one- and five-year Treasury bond yields in which the transition probability matrix is parameterized with five parameters with some elements restricted to zero. Assumption A1(a) of DMR does not hold in these papers.
 
The two assumptions imposed by DMR also rule out models in which the conditional density $Y_k$ depends on both the current and the lagged regimes. Suppose that we specify $X_k$ in (\ref{MS-model}) as
\begin{equation}\label{MS-model-X}
X_k=(\widetilde X_{k},\widetilde X_{k-1},\ldots,\widetilde X_{k-p+1}),
\end{equation}
where $p \geq 2$, and $\widetilde X_{k}$ follows a first-order Markov process and is called the \textit{regime}. Then, the transition density of $X_k$ inevitably has zeros. For example, when $p=2$ and $X_k=(\tilde X_k,\tilde X_{k-1})$, we have $\Pr\left(X_{k+1}=(i',j')|X_{k}=(i,j)\right)=0$ when $j'\neq i$. Consequently, the asymptotic distribution of the MLE has not been proven for some popular Markov regime switching models including the seminal model of \citet{hamilton89em} and switching ARCH (SWARCH) model of \citet{hamiltonsusmel94joe}. 
\begin{example}[\citet{hamilton89em}]\label{example_1} 
Consider the following model:
\begin{equation}\label{hamilton}
Y_k = \mu_{\tilde X_k} + u_k\quad\text{with}\quad u_k = \sum_{\ell=1}^{p-1} \gamma_\ell u_{k-\ell} + \sigma \varepsilon_k\quad\text{for\ $p\geq 2$},
\end{equation}
where $\varepsilon_k\sim$ i.i.d. $N(0,1)$ and $\tilde X_k$ follows a Markov chain on $\tilde{\mathcal{X}}=\{1,2,\ldots,M\}$ with $\Pr(\tilde X_k=j|\tilde X_{k-1}=i)=p_{ij}$, where $M$ represents the number of regimes.
% Let \\ $\theta=(\mu_1',\ldots,\mu_M',\gamma',\{p_{ij}\}_{1 \leq i \leq M, 1 \leq j \leq M-1})'$ with $\gamma=(\gamma_1,\ldots,\gamma_{p-1},\sigma)'$. Then, the conditional density of $Y_k$ given $(Y_{k-1},\ldots,Y_{k-p})'$ and $X_k=(\tilde X_{k},\ldots,\tilde X_{k-p+1})'$ is 
%\begin{align*}
%g_{\theta}(Y_k|Y_{k-1},\ldots,Y_{k-p+1}, X_k)=\frac{1}{\sigma}\phi\left(\frac{(Y_k-\mu_{\tilde X_k} )-\sum_{\ell=1}^{p-1} \gamma_{\ell}(Y_{k-\ell}-\mu_{\tilde X_{k-\ell}} ) }{\sigma}\right),
%\end{align*}
%where $\phi(z) = (1/\sqrt{2\pi})\exp(-z^2/2)$.
\citet{hamilton89em} estimates model (\ref{hamilton}) with $M=2$ and $p=5$ by using data on U.S. real GNP growth.
\end{example} 
\citet{mcconnelperez-quiros00aer} and \citet{camachoperez-quiros07sdne} estimate an augmented model (\ref{hamilton}) that allows the standard deviation parameter $\sigma$ in (\ref{hamilton}) to be regime-dependent.

\begin{example}[SWARCH model of \citet{hamiltonsusmel94joe}]\label{example_2} 
Consider the following model:
\begin{align*} 
Y_k &= \mu + \gamma_y Y_{k-1} + \sigma_{\tilde X_k} h_k \varepsilon_{k} \quad\text{with}\quad
h_k^2 = \gamma_0+\sum_{\ell=1}^{p-1}\gamma_{\ell} \left(h_{k-\ell} \varepsilon_{k-\ell}\right)^2\quad\text{for\ $p\geq 2$},
\end{align*} 
where $\varepsilon_k\sim$ i.i.d. $N(0,1)$ or Student $t$ with $v$ degrees of freedom, and $\tilde X_k$ follows a Markov chain on $\tilde{\mathcal{X}}=\{1,2,\ldots,M\}$ with $\Pr(\tilde X_k=j|\tilde X_{k-1}=i)=p_{ij}$. 
%For a Gaussian SWARCH model, let $\theta=(\theta_1',\ldots,\theta_M',\gamma',\{p_{ij}\}_{ 1 \leq i \leq M, 1 \leq j \leq M-1})'$ with $\theta_j= \sigma_j^2$ and $\gamma=(\mu,\gamma_y,\gamma_0,\gamma_1,\ldots,\gamma_{p-1})'$. Then, the conditional density of $Y_k$ given $(Y_{k-1},\ldots,Y_{k-p})'$ and $X_k=(\tilde X_{k},\ldots,\tilde X_{k-p+1})'$ is
%\begin{align*}
%g_{\theta}(Y_k|Y_{k-1},\ldots,Y_{k-p},X_k)
% =\frac{1}{\sigma_{\tilde X_k}h_k}\phi\left( \frac{Y_k - \mu - \gamma_y Y_{k-1} }{\sigma_{\tilde X_k}h_k}\right),
%\end{align*}
%where $h_k^2 = \gamma_0+\sum_{\ell=1}^{p-1}\gamma_{\ell}\left(\frac{Y_{k-\ell} - \mu - \gamma_y Y_{k-\ell-1} }{\sigma_{\tilde X_{k-\ell}}} \right)^2$.
\end{example} 
\begin{example}[Bounce-back effect model of \citet{kimmorleypiger05jae}]\label{example_3} 
Consider the following model:
\[ 
\Delta Y_k = \mu_{\tilde X_k} + \lambda \sum_{j=1}^{p-1} \tilde X_{k-j} + \varepsilon_k \quad\text{with}\quad
\varepsilon_{k} \sim \text{i.i.d}\ N(0,1),
\]
where $\tilde X_k$ follows a Markov chain on $\tilde{\mathcal{X}}=\{1,2,\ldots,M\}$ with $\Pr(\tilde X_k=j|\tilde X_{k-1}=i)=p_{ij}$. \citet{kimmorleypiger05jae} use this model with $p=6$ to capture the post-recession ``bounce-back'' effect in U.S. quarterly GDP.
\end{example} 
In examples \ref{example_1}--\ref{example_3}, the transition probability of $X_k=(\tilde X_k,\ldots,\tilde X_{k-p+1})'$ has zeros when $p\geq 2$. Therefore, Assumption A1(a) of DMR is violated. As discussed on pages 2257--2258 of DMR, Assumption A1(a) is crucial for their Corollary 1 (page 2262) that establishes the \textit{deterministic} geometrically decaying bound on the mixing rate of the conditional chain, $X|Y$. As DMR recognize on page 2258, this deterministic nature of the bound is vital to their proof of the asymptotic normality of the MLE.
 
This paper shows the consistency and asymptotic normality of the MLE of the Markov regime switching model in which some elements of the transition probability matrix of $X_k$ are zero. To the best of our knowledge, there exists no rigorous proof in the literature of the asymptotic normality of the MLE of these regime switching models, even though empirical researchers often assume that some elements of the transition probability matrix are zero and the models of \citet{hamilton89em} and \citet{hamiltonsusmel94joe} are popular in applied work. This is an important gap in the literature to be filled because empirical researchers regularly make inferences based on the presumed asymptotic normality \citep[e.g.,][]{goodwin93jbes, garciaperron96restat, hamiltonlin96jae, fong97fejm, ramchandsasmel98jef, maheumccurdy2000jbes, mcconnelperez-quiros00aer, edwardssusmel01jde, camachoperez-quiros07sdne}. This paper therefore provides the theoretical basis for the statistical inferences associated with these models.

To derive the asymptotic normality of the MLE, we first establish a bound on the mixing rate of the conditional chain, $X|Y$, in Lemma \ref{corollary_1}. Our bound is written as a product of random variables, where all but finitely many of them are strictly less than 1. Consequently, the mixing rate of the conditional chain is geometrically decaying almost surely. We then use this mixing rate to show that the sequence of the conditional scores and conditional Hessians given the $m$ past periods converge to the conditional score and conditional Hessian given the ``infinite past'' as $m \to \infty$. Given these results, we show the asymptotic normality of the MLE under standard regularity assumptions by applying a martingale central limit theorem to the score function (Proposition \ref{prop_score}) as well as by proving a uniform law of large numbers for the observed Fisher information (Proposition \ref{prop_hessian}). These results extend those in DMR to an empirically important class of models where the transition density has the value of zero. Another feature of the present study is that we introduce an additional weakly exogenous regressor, $W_k$.
 
We also relax the assumption in DMR on the regime-specific density. DMR assume that the regime-specific density is uniformly bounded with respect to $X_k$ and $\theta$, whereas we only assume the existence of the first moment of the supremum of the logarithm of the regime-specific density with respect to $X_k$ and $\theta$. Unbounded densities are used in other analyses \citep[e.g.,][]{zindewalsh08et} and empirical studies. 
\begin{example}[Markov regime-switching conditional duration (MS-CD) model]\label{example_4} 
Consider the following model:
\begin{equation}\label{MSCD}
Y_k = (\mu_{X_k} + \beta Y_{k-1}) \varepsilon_k,
\end{equation}
where $\varepsilon_k$ follows the standardized Weibull distribution with $E(\varepsilon_k)=1$, $\mu_{X_k}$ is a regime-dependent parameter, and $X_k\in \{1,2,\ldots,M\}$ is a regime in period $k$. The density of $Y_k$ conditional on $Y_{k-1}$ and $X_k$ is given by $g_{\theta} (y_k|Y_{k-1},X_k) = \frac{\gamma}{\lambda_k(Y_{k-1},X_k)} \left(\frac{y_k}{\lambda_k(Y_{k-1},X_k)}\right)^{\gamma-1} \exp\left\{- \left(\frac{y_k}{\lambda_k(Y_{k-1},X_k)} \right)^{\gamma}\right\}$, where $\lambda_k(Y_{k-1},X_k) = \frac{\mu_{X_k} + \beta Y_{k-1}}{\Gamma(1+1/\gamma)}$. In this model, the regime-specific density is unbounded when $\gamma<1$.
\end{example} 
Our simulations based on the model (\ref{MSCD}) show that the asymptotic distribution provides a good approximation of the finite sample behavior of the MLE even when the regime-specific density is unbounded.
 
In regime switching models, testing for the number of regimes (number of elements in $\mathcal{X}$) has been an unsolved problem because the standard asymptotic analysis of the likelihood ratio test statistic (LRTS) breaks down. In testing the null hypothesis of no regime switching, the asymptotic behavior of the LRTS has been investigated by \citet{hansen92jae} and \citet{garcia98ier}; \citet{carrasco14em} propose an information matrix-type test and \citet{chowhite07em} derive the asymptotic distribution of the quasi-LRTS. Recently, \citet{quzhuo17wp} derive the asymptotic distribution of the LRTS of testing the null hypothesis of no regime switching under some restrictions on the transition probabilities of regimes and \citet{rabah12wp} compares the finite sample performance of the bootstrapped LRTS with the test of \citet{carrasco14em}. \citet{kasaharashimotsu18markov} derive the asymptotic distribution of the LRTS for testing the null hypothesis of $M$ regimes against the alternative hypothesis of $M + 1$ regimes for any $M\geq 1$ and show the asymptotic validity of the parametric bootstrap.

The remainder of this paper is organized as follows. Section 2 introduces the notation, model, and assumptions. Section 3 derives the bound on the mixing rate of the conditional chain, $X|Y$. Section 4 derives the consistency of the MLE, and the asymptotic normality of the MLE is shown in Section 5. Section 6 reports the simulation results. Appendix A collects the proofs and Appendix B collects the auxiliary results.

\section{Model and assumptions}\label{sec:model}

Our notation largely follows the notation in DMR. Let $:=$ denote ``equals by definition.'' For a $k\times 1$ vector $x = (x_1,\ldots,x_k)'$ and a matrix $B$, define $|x| := \sqrt{x'x}$ and $|B| := \sqrt{\lambda_{\max}(B'B)}$, where $\lambda_{\max}(B'B)$ denotes the largest eigenvalue of $B'B$. For a $k\times 1$ vector $a=(a_1,\ldots,a_k)'$ and a function $f(a)$, let $\nabla_{a}^2 f(a) := \nabla_{aa'}f(a)$. For two probability measures $\mu_1$ and $\mu_2$, the total variation distance between $\mu_1$ and $\mu_2$ is defined as $\|\mu_1-\mu_2\|_{TV}:=\sup_{A}|\mu_1(A)-\mu_2(A)|$. $\|\cdot\|_{TV}$ satisfies $\sup_{f(x):0 \leq f(x) \leq 1}|\int f(x) \mu_1(dx)- \int f(x) \mu_2(dx)| = \|\mu_1-\mu_2\|_{TV}$ and $\sup_{f(x):\max_x|f(x)| \leq 1}|\int f(x) \mu_1(dx)- \int f(x) \mu_2(dx)| = 2 \|\mu_1-\mu_2\|_{TV}$ for any two probability measures $\mu_1$ and $\mu_2$ (e.g., \citet[][Proposition 4.5]{levin09book}). Let $\mathbb{I}\{A\}$ denote an indicator function that takes the value of one when $A$ is true and zero otherwise. For a metric space $\mathcal{A}$, let $\mathcal{A}^k$ denote the $k$-fold product space of $\mathcal{A}$. $\mathcal{C}$ denotes a generic finite positive constant whose value may change from one expression to another. Let $a \vee b :=\max\{a,b\}$ and $a \wedge b :=\min\{a,b\}$. Let $\lfloor x \rfloor$ denote the largest integer less than or equal to $x$, and define $(x)_+ := \max\{x,0\}$. For any $\{x_i\}$, we define $\sum_{i=a}^b x_i:=0$ and $\prod_{i=a}^b x_i:=1$ when $b<a$. ``i.o.'' stands for ``infinitely often.'' All the limits below are taken as $n \to \infty$ unless stated otherwise.
 
We consider the Markov regime switching process defined by a discrete-time stochastic process $\{(X_k,Y_k,W_k)\}$, where $(X_k,Y_k,W_k)$ takes the values in a set $\mathcal{X}\times \mathcal{Y}\times \mathcal{W}$ with the associated Borel $\sigma$-field $\mathcal{B}(\mathcal{X}\times\mathcal{Y}\times\mathcal{W})$. We use $p_\theta(\cdot)$ to denote densities with respect to the probability measure on $\mathcal{B}(\mathcal{X}\times\mathcal{Y}\times\mathcal{W})^{\otimes \mathbb{Z}}$. For a stochastic process $\{U_k\}$ and $a<b$, define ${\bf U}_{a}^b:=(U_a,U_{a+1},\ldots,U_b)$. Denote $\overline{\bf Y}_{k-1}:=(Y_{k-1},\ldots,Y_{k-s})$ for a fixed integer $s$ and $\overline {\bf Y}^b_{a}:=(\overline {\bf Y}_{a},\overline {\bf Y}_{a+1},\ldots,\overline {\bf Y}_{b})$. Define $Z_k:=(X_k,\overline{\bf Y}_{k})$. Let $Q_\theta(x,A):=\mathbb{P}_{\theta}(X_k \in A|X_{k-1}=x)$ denote the transition kernel of $\{X_k\}_{k=0}^\infty$. Let $Q_\theta^r(x,A):=\mathbb{P}_{\theta}(X_k \in A|X_{k-r}=x)$ denote the $r$-step transition kernel of $\{X_k\}_{k=0}^\infty$.

We now introduce our assumptions, which mainly follow the assumptions in DMR.

\begin{assumption} \label{assn_a1}
(a) The parameter $\theta$ belongs to $\Theta$, a compact subset of $\mathbb{R}^{q}$, and the true parameter value $\theta^*$ lies in the interior of $\Theta$. (b) $\{X_k\}_{k=0}^\infty$ is a Markov chain that lies in a compact set $\mathcal{X} \subset \mathbb{R}^{d_x}$. (c) For all $\theta \in \Theta$, $Q_\theta(x,\cdot)$  and $Q_\theta^r(x,\cdot)$ have densities $q_\theta(x,\cdot)$ and $q_\theta^r(x,\cdot)$, respectively, with respect to a \emph{finite} dominating measure $\mu$ on $\mathcal{B}(\mathcal{X})$ such that $\mu(\mathcal{X})=1$, and $\sigma_+^0:=\sup_{\theta \in \Theta} \sup_{x,x' \in \mathcal{X}} q_\theta(x,x') < \infty$. (d) There exists a finite $p \geq 1$ such that $0<\sigma_-:=\inf_{\theta \in \Theta} \inf_{x_{k-p},x_k \in \mathcal{X}} q_\theta^p(x_{k-p},x_k)$ and $\sigma_+:=\sup_{\theta \in \Theta} \sup_{x_{k-p},x_k \in \mathcal{X}} q_\theta^p(x_{k-p},x_k) < \infty$. (e) $\{(Y_k,W_k)\}_{k=-s+1}^\infty$ takes the values in a set $\mathcal{Y}\times \mathcal{W} \subset \mathbb{R}^{d_y} \times \mathbb{R}^{d_w}$. 
\end{assumption}

\begin{assumption} \label{assn_a2}
(a) For each $k\geq 1$, $X_k$ is conditionally independent of $({\bf X}_{0}^{k-2},\overline {\bf Y}_{0}^{k-1},{\bf W}_{0}^\infty)$ given $X_{k-1}$. (b) For each $k \geq 1$, $Y_k$ is conditionally independent of $({\bf Y}^{k-s-1}_{-s+1}, {\bf X}_{0}^{k-1},{\bf W}_{0}^{k-1},{\bf W}_{k+1}^\infty)$ given $(\overline {\bf Y}_{k-1}, X_k, W_k)$, and the model of the conditional distribution of $Y_k$ has a density $g_\theta(y_k|\overline{\bf Y}_{k-1}, X_k, W_k)$ with respect to a $\sigma$-finite measure $\nu$ on $\mathcal{B}(\mathcal{Y})$. (c) ${\bf W}_{1}^{\infty}$ is conditionally independent of $(\overline {\bf Y}_{0}, X_{0})$ given $W_0$. (d) $\{(Z_k,W_k)\}_{k=0}^{\infty}$ is a strictly stationary ergodic process. 
\end{assumption} 
\begin{assumption}\label{assn_a3}
For all $y'\in \mathcal{Y}$, $\overline{\bf y}\in \mathcal{Y}^s$, and $w \in \mathcal{W}$, $0<\inf_{\theta\in\Theta}\inf_{x\in\mathcal{X}} g_{\theta}(y'|\overline{\bf y},x,w)$ and $\sup_{\theta \in \Theta}\sup_{x\in\mathcal{X}} g_{\theta}(y'|\overline{\bf y},x,w)<\infty$.
\end{assumption}
Assumption \ref{assn_a1}(c) is also assumed on page 2258 of DMR. This assumption excludes the case where $\mathcal{X} = \mathbb{R}$ and $\mu$ is the Lebesgue measure but allows for continuously distributed $X_k$ with finite support. If multiple $p$'s satisfy Assumption \ref{assn_a1}(d), we define $p$ as the minimum of such $p$'s. Assumption \ref{assn_a1}(d) implies that the state space $\mathcal{X}$ of the Markov chain $\{X_k\}$ is $\nu_{p}$-small for some nontrivial measure $\nu_p$ on $\mathcal{B}(\mathcal{X})$. Therefore, for all $\theta \in \Theta$, the chain $\{X_k\}$ is aperiodic and has a unique invariant distribution and is uniformly ergodic \citep[][Theorem 16.0.2]{meyntweedie09book}. Assumptions \ref{assn_a2}(a)(b) imply that $Z_k$ is conditionally independent of $({\bf Z}_0^{k-2},{\bf W}_{0}^{k-1},{\bf W}_{k+1}^{\infty})$ given $(Z_{k-1},W_k)$; hence, $\{Z_k\}_{k=0}^{\infty}$ is a Markov chain on $\mathcal{Z}:=\mathcal{X}\times\mathcal{Y}^s$ given $\{W_k\}_{k=0}^{\infty}$. Under Assumptions \ref{assn_a2}(a)--(c), the conditional density of ${\bf Z}_0^n$ given ${\bf W}_0^n$ is written as $p_\theta({\bf Z}_0^n|{\bf W}_0^n)= p_{\theta}(Z_0|W_0) \prod_{k=1}^n p_\theta(Z_k|Z_{k-1},W_k)$. Because $\{(Z_k,W_k)\}_{k=0}^\infty$ is stationary, we extend $\{(Z_k,W_k)\}_{k=0}^\infty$ to a stationary process $\{(Z_k,W_k)\}_{k=-\infty}^\infty$ with doubly infinite time. We denote the probability and associated expectation of $\{(Z_k,W_k)\}_{k=\infty}^\infty$ under stationarity by $\mathbb{P}_\theta$ and $\mathbb{E}_\theta$, respectively.\footnote{DMR use $\overline{\mathbb{P}}_\theta$ and $\overline{\mathbb{E}}_\theta$ to denote the probability and expectation under stationarity because their Section 7 deals with the case when $Z_0$ is drawn from an arbitrary distribution. Because we assume $\{(Z_k,W_k)\}_{k=\infty}^\infty$ is stationary throughout this paper, we use notations such as $\mathbb{P}_\theta$ and $\mathbb{E}_\theta$ without an overline for simplicity.} Assumption \ref{assn_a3} is stronger than Assumption A1(b) in DMR, which assumes only $0<\inf_{\theta\in\Theta}\int_{x\in\mathcal{X}} g_{\theta}(y'|\overline{\bf y},x)\mu(dx)$ and $\sup_{\theta \in \Theta}\int_{x\in\mathcal{X}} g_{\theta}(y'|\overline{\bf y},x) \mu(dx)<\infty$. When $\mathcal{X}$ is finite, Assumption \ref{assn_a3} becomes identical to Assumption A3 of \citet{francq98stat}, who prove the consistency of the MLE when $\mathcal{X}$ is finite. It appears that assuming a lower bound on $g_\theta$ similar to Assumption \ref{assn_a3} is necessary to derive the asymptotics of the MLE when $\inf_\theta \inf_{x,x'} q_\theta(x,x')=0$. When $p=1$, we could weaken Assumption \ref{assn_a3} to Assumption A1(b) in DMR, but we retain Assumption \ref{assn_a3} to simplify the exposition and proof.

DMR assume $p=1$ in Assumption \ref{assn_a1}(d), meaning that the transition density $q_\theta(x,x')$ of the state variable $X_k$ is uniformly bounded from below. DMR show that this lower bound on $q_\theta(x,x')$ translates into a deterministic lower bound on the conditional transition density of $X_k$ given the observations of $\{(\overline{\bf Y}_k, W_k)\}_{k=0}^n$. Owing to this deterministic lower bound, the chain $\{X_k\}$ given $\{(\overline{\bf Y}_k, W_k)\}_{k=0}^n$ is geometrically mixing, and, consequently, the derivatives of the log-densities are also geometrically mixing and follow the law of large numbers and central limit theorem.

When $p \geq 2$, this lower bound is no longer deterministic and depends on the $Y_k$'s. For example, suppose that $\mathcal{X} = \{-1,0,1\}$, which correspond to ``recession,'' ``normal,'' and ``expansion'' periods, respectively, $\mathbb{P}(X_k=1|X_{k-1}=-1)=0$, and $Y_k|(X_k,Y_{k-1}) \sim N(0.6 Y_{k-1} + X_k, 1)$. Then, observing a negative value of $Y_{k-1}$ implies that the likely value of $X_{k-1}$ is $-1$, which in turn implies that the event $\{X_k=1\}$ is unlikely. As $Y_{k-1}$ approaches negative infinity, $\mathbb{P}(X_k=1|Y_{k},Y_{k-1})$ approaches zero and no lower bound on the transition density of $X_{k-1}$ exists given $(Y_k,Y_{k-1})$.

We overcome the zero lower bound on $q_\theta(x,x')$ by noting that, in many econometric models, only extreme values of $Y_{k-1}$ provide a strong signal on the value of $X_{k-1}$. Because $Y_{k-1}$ takes such extreme values with a small probability, the transition probability of the chain $\{X_k\}$ given $\{(\overline{\bf Y}_k, W_k)\}_{k=0}^n$ is bounded from below by a stochastic lower bound whose value is close to zero only with a small probability. As a result, the chain $\{X_k\}$ given $\{(\overline{\bf Y}_k, W_k)\}_{k=0}^n$ is geometrically mixing with a probability close to one.

Following DMR, we analyze the conditional log-likelihood function given $\overline{\bf Y}_{0}$, ${\bf W}_{0}^n$, and $X_0 = x_0$ rather than the stationary log-likelihood function given $\overline{\bf Y}_{0}$ and ${\bf W}_{0}^n$ because, as explained in DMR (pages 2263--2264), the conditional initial density $p_\theta(X_0|\overline{\bf Y}_0^{k-1})$ cannot be easily computed in practice. The conditional density function of ${\bf Y}_{1}^n$ is
\begin{equation} \label{cond_density}
p_{\theta}({\bf Y}_1^n| \overline{\bf Y}_{0},{\bf W}_{0}^n,x_0) = \int \prod_{k=1}^n p_{\theta}(Y_k,x_k|\overline{\bf Y}_{k-1},x_{k-1},W_k) \mu^{\otimes n}(d {\bf x}_1^n),
\end{equation}
where $p_{\theta} (y_k,x_k| \overline{\bf{y}}_{k-1},x_{k-1},w_k) = q_{\theta}(x_{k-1},x_k)g_{\theta}(y_k|\overline{\bf{y}}_{k-1}, x_k, w_k)$. Assumptions \ref{assn_a2}(a)--(c) imply that for $k\geq 1$, $W_k$ is conditionally independent of ${\bf Z}_{0}^{k-1}$ given ${\bf W}_{0}^{k-1}$ because $p(W_k|{\bf Z}_0^{k-1},{\bf W}_0^{k-1}) = p({\bf W}_0^k,{\bf Z}_0^{k-1}) / p({\bf W}_0^{k-1},{\bf Z}_0^{k-1})$ and for $j=k, k-1$, $p({\bf W}_0^j,{\bf Z}_0^{k-1}) = p(Z_0,{\bf W}_0^{j}) \prod_{t=1}^{k-1} p(Z_t|Z_{t-1},W_t) = p({\bf W}_1^{j}|W_0)p(Z_0|W_0)\prod_{t=1}^{k-1} p(Z_t|Z_{t-1},W_t)$. Therefore, for $1 \leq k \leq n$, we have
\begin{align} 
p_\theta ({\bf Y}_1^k| \overline{\bf Y}_{0},{\bf W}_{0}^n,x_0) & = p_\theta ({\bf Y}_1^k| \overline{\bf Y}_{0},{\bf W}_{0}^k,x_0), \label{cond_w_1} \\
p_\theta ({\bf Y}_1^k| \overline{\bf Y}_{0},{\bf W}_{0}^n) & = p_\theta ({\bf Y}_1^k| \overline{\bf Y}_{0},{\bf W}_{0}^k). \label{cond_w_2}
\end{align}
In view of (\ref{cond_w_1}) and (\ref{cond_w_2}), we can write the conditional and stationary log-likelihood functions as
\begin{equation} \label{l_n_x0}
\begin{aligned}
l_n(\theta,x_0) & := \log p_\theta({\bf Y}_1^n| \overline{\bf Y}_{0},{\bf W}_{0}^n,x_0) =\sum_{k=1}^n \log p_\theta (Y_k| \overline{\bf Y}_{0}^{k-1},{\bf W}_{0}^{k},x_0),\\
l_n(\theta) & := \log p_\theta({\bf Y}_1^n| \overline{\bf Y}_{0},{\bf W}_{0}^n) =\sum_{k=1}^n \log p_\theta (Y_k| \overline{\bf Y}_{0}^{k-1},{\bf W}_{0}^k). 
\end{aligned}
\end{equation}
Many applications use the log-likelihood function in which the conditional density $p_\theta({\bf Y}_1^n| \overline{\bf Y}_{0},{\bf W}_{0}^n,x_0)$ is integrated with respect to $x_0$ over a probability measure $\xi$ on $\mathcal{B}(\mathcal{X})$, where $\xi$ can be fixed or treated as an additional parameter. We also analyze the resulting objective function: 
\begin{equation} \label{l_n_xi}
l_n(\theta,\xi) := \log \left(\int p_\theta({\bf Y}_1^n| \overline{\bf Y}_{0},{\bf W}_{0}^n,x_0)\xi(dx_0)\right).
\end{equation}

\section{Uniform forgetting of the conditional hidden Markov chain}

In this section, we establish a mixing rate of the conditional hidden Markov chain, which is the  chain $\{X_k\}_{k=-m}^{n}$ given $(\overline{\bf Y}_{-m}^n, {\bf W}_{-m}^n)$. The bounds on this mixing rate are instrumental in deriving the asymptotic properties of the MLE.  The following lemma bounds the distance between the distributions of $X_k$ given $(\overline{\bf Y}_{-m}^n, {\bf W}_{-m}^n)$ when starting from two different initial distributions $\mu_1(\cdot)$ and $\mu_2(\cdot)$ of $X_{-m}$. In other words, this lemma provides the rate at which the conditional hidden Markov chain forgets its past. This lemma generalizes  Corollary 1 of DMR, which shows that the conditional hidden Markov chain forgets its past at a deterministic exponential rate when $p=1$. As DMR note on page 2258, their deterministic rate holds only when $p=1$. 
\begin{lemma}\label{corollary_1}
Assume Assumptions \ref{assn_a1}--\ref{assn_a3}. Let $m, n \in \mathbb{Z}$ with $-m \leq n$ and $\theta \in \Theta$. Then, for all $-m \leq k \leq n$, all probability measures $\mu_1$ and $\mu_2$ on $\mathcal{B}(\mathcal{X})$, and all $(\overline{\bf y}_{-m}^n,{\bf w}_{-m}^n)$,
\begin{align*}
& \left\| \int_{\mathcal{X}} \mathbb{P}_\theta \left(X_k \in \cdot \middle| X_{-m}=x, \overline{\bf y}_{-m}^n, {\bf w}_{-m}^n \right) \mu_1(dx) - \int_{\mathcal{X}} \mathbb{P}_\theta \left(X_k \in \cdot \middle| X_{-m}=x, \overline{\bf y}_{-m}^n, {\bf w}_{-m}^n \right) \mu_2(dx) \right\|_{TV} \\
&\leq \prod_{i=1}^{\lfloor (k+m)/p \rfloor} \left( 1-\omega(\overline{\bf y}_{-m+pi-p}^{-m+pi-1}, {\bf w}_{-m+pi-p}^{-m+pi-1}) \right), 
\end{align*}
where $\omega(\overline{\bf y}_{k-p}^{k-1},{\bf w}_{k-p}^{k-1}) := \sigma_-/\sigma_+$ when $p=1$, and, when $p \geq 2$,\footnote{Strictly speaking, $w_{k-p}$ in $\omega(\overline{\bf y}_{k-p}^{k-1},{\bf w}_{k-p}^{k-1})$ is superfluous because $\omega(\overline{\bf y}_{k-p}^{k-1},{\bf w}_{k-p}^{k-1})$ does not depend on $w_{k-p}$. We retain $w_{k-p}$ for notational simplicity.}
\begin{equation} \label{omega_defn}
\omega(\overline{\bf y}_{k-p}^{k-1},{\bf w}_{k-p}^{k-1}) := \frac{\sigma_-}{\sigma_+}   \left( \frac{  \inf_\theta \inf_{{\bf x}_{k-p+1}^{k-1}} \prod_{i=k-p+1}^{k-1} g_\theta(y_i|\overline{\bf y}_{i-1},x_i,w_i)}{  \sup_\theta \sup_{{\bf x}_{k-p+1}^{k-1}} \prod_{i=k-p+1}^{k-1} g_\theta(y_i|\overline{\bf y}_{i-1},x_i,w_i)}\right)^2.
\end{equation}
\end{lemma}
The convergence rate of the conditional hidden Markov chain depends on the minorization coefficient $\omega(\overline{\bf Y}_{k-p}^{k-1}, {\bf W}_{k-p}^{k-1})$. If this coefficient is bounded away from 0, the chain forgets its past exponentially fast. When $p\geq 2$, this coefficient is not  necessarily  bounded away from 0 because $\inf_{\overline{\bf y}_{k-p}^{k-1}, {\bf w}_{k-p}^{k-1}}\omega(\overline{\bf y}_{k-p}^{k-1}, {\bf w}_{k-p}^{k-1})$  can be possibly 0. However, $\omega(\overline{\bf Y}_{k-p}^{k-1}, {\bf W}_{k-p}^{k-1})$ becomes close to zero only when ${\bf Y}_{k-p+1}^{k-1}$ takes an unlikely value because the denominator of $\omega(\overline{\bf Y}_{k-p}^{k-1}, {\bf W}_{k-p}^{k-1})$ is finite and the numerator of $\omega(\overline{\bf Y}_{k-p}^{k-1}, {\bf W}_{k-p}^{k-1})$ is a product of the conditional density $g_\theta(y|\overline{\bf y},x,w)$. As a result, $\omega(\overline{\bf Y}_{k-p}^{k-1}, {\bf W}_{k-p}^{k-1})$ is bounded away from 0 with a probability close to 1. In the following sections, we use this fact to establish the consistency and asymptotic normality of the MLE.

\section{Consistency of the MLE}

Define the conditional MLE of $\theta^*$ given $\overline{\bf Y}_{0}$, ${\bf W}_{0}^n$, and $X_0 = x_0$ as
\[
\hat \theta_{x_0}:= \argmax_{\theta \in \Theta} l_n(\theta,x_0),
\]
with $l_n(\theta,x_0)$ defined in (\ref{l_n_x0}). In this section, we prove the consistency of the conditional MLE. We introduce additional assumptions required for proving consistency.
\begin{assumption}\label{assn_gbound}
(a)  $\mathbb{E}_{\theta^*} |\log b_+(\overline{\bf Y}_{0}^1,W_1)|< \infty$, where \\$b_+(\overline{\bf Y}_{k-1}^k,W_k):= \sup_{\theta \in \Theta}\sup_{x_k \in \mathcal{X}} g_{\theta}(Y_k| \overline{\bf Y}_{k-1},x_{k},W_k)$. (b) $\mathbb{E}_{\theta^*} |\log b_-(\overline{\bf Y}_{0}^1,W_1)|< \infty$, where \\ $b_-(\overline{\bf Y}_{k-1}^k,W_k):= \inf_{\theta \in \Theta}\inf_{x_k \in \mathcal{X}} g_{\theta}(Y_k| \overline{\bf Y}_{k-1},x_{k},W_k)$. 
\end{assumption}
\begin{assumption}\label{assn_gbound2}
There exist constants $\alpha>0$, $C_1,C_2 \in (0,\infty)$, and $\beta>1$ such that, for any $r > 0$, 
\[
\mathbb{P}_{\theta^*}\left( b_+(\overline{\bf Y}_{0}^1,W_1)/  b_-(\overline{\bf Y}_{0}^1,W_1) \geq  C_1 e^{ \alpha r} \right) \leq C_2 r^{-\beta}.
\]
\end{assumption}
Assumption \ref{assn_gbound}(a)  relaxes Assumption (A3) of DMR, who assume that\\ $\sup_{\theta \in \Theta}\sup_{y_1,\overline{\bf y}_0,x,w} g_{\theta}(y_1|\overline{\bf y}_0,x,w)<\infty$ and hence the density is uniformly bounded. Assumption \ref{assn_gbound}(b) is stronger than Assumption (A3) of DMR, who assume $\mathbb{E}_{\theta^*} |\log (\inf_{\theta \in \Theta} \int g_{\theta}(Y_1| \overline{\bf Y}_{0},x) \mu(dx)|< \infty$. Assumption \ref{assn_gbound} implies that $\mathbb{E}_{\theta^*} \sup_{\theta \in \Theta}\sup_{x \in \mathcal{X}}|\log ( g_{\theta}(Y_1| \overline{\bf Y}_{0},x,W_1))|< \infty$, which is similar to the moment condition used in the standard maximum likelihood estimation, but the infimum is taken over $x$ in addition to $\theta$. Assumption \ref{assn_gbound2} restricts the probability that\\ $\sup_{\theta \in \Theta}\sup_{x_k \in \mathcal{X}}g_\theta(Y_k|\overline{\bf Y}_{k-1},x_k,W_k) /\inf_{\theta \in \Theta}\inf_{x_k \in \mathcal{X}}g_\theta(Y_k|\overline{\bf Y}_{k-1},x_k,W_k)$ takes an extremely  large  value. Assumption \ref{assn_gbound2} is not restrictive because the right hand side of the inequality inside $\mathbb{P}_{\theta^*}(\cdot)$ is exponential in $r$ and the bound $C_2r^{-\beta}$ is a polynomial in $r$. An easily verifiable sufficient condition for Assumption \ref{assn_gbound2} is $\mathbb{E}_{\theta^*} |\log (b_+(\overline{\bf Y}_{0}^1,W_1)/  b_-(\overline{\bf Y}_{0}^1,W_1) )|^{1+\delta}< \infty$ for some $\delta>0$. This is because $\mathbb{P}_{\theta^*} ( b_+(\overline{\bf Y}_{0}^1,W_1) / b_-(\overline{\bf Y}_{0}^1,W_1)  \geq e^{ \alpha r}) = \mathbb{P}_{\theta^*} (  \log( b_+(\overline{\bf Y}_{0}^1,W_1)/ b_-(\overline{\bf Y}_{0}^1,W_1) )  \geq \alpha r)$\\ $ \leq ( \mathbb{E}_{\theta^*} |\log ( b_+(\overline{\bf Y}_{0}^1,W_1)/ b_-(\overline{\bf Y}_{0}^1,W_1) )|^{1+\delta} )/(\alpha r)^{1+\delta} \leq C_2 r^{-(1+\delta)}$, where the first inequality follows from  Markov's inequality. Examples \ref{example_1}--\ref{example_4} satisfy Assumptions \ref{assn_gbound} and \ref{assn_gbound2}.

In the following lemma, we show that the difference between the conditional log-likelihood function $l_n(\theta,x_0)$ and the stationary log-likelihood function $l_n(\theta)$ is $o(n)$ $\mathbb{P}_{\theta^*}$-a.s. 
\begin{lemma}\label{lemma_lnx0}
Assume Assumptions \ref{assn_a1}--\ref{assn_gbound2}. Then, 
\[
n^{-1} \sup_{x_0 \in \mathcal{X}}\sup_{\theta \in \Theta}|l_n(\theta,x_0) - l_n(\theta)| \to 0 \quad \mathbb{P}_{\theta^*}\text{-}a.s.
\]
\end{lemma}
When $p=1$, Lemma 2 of DMR shows that $\sup_{\theta \in \Theta}|l_n(\theta,x_0) - l_n(\theta)|$ is bounded by a deterministic constant. When $p \geq 2$, Lemma 2 of DMR is no longer applicable because $|l_n(\theta,x_0) - l_n(\theta)|$ depends on the products of $1-\omega(\overline{\bf Y}_{pi-p}^{pi-1}, {\bf W}_{pi-p}^{pi-1})$'s for $i=1,\ldots,\lfloor n/p \rfloor$. A key observation is that $\{\omega(\overline{\bf Y}_{pi-p}^{pi-1}, {\bf W}_{pi-p}^{pi-1})\}_{i \geq 1}$ is stationary and ergodic and that $\epsilon:=\mathbb{P}_{\theta^*}(\omega(\overline{\bf Y}_{pi-p}^{pi-1}, {\bf W}_{pi-p}^{pi-1}) \leq \delta)$ is small when $\delta>0$ is sufficiently small. Because the strong law of large numbers implies that $(\lfloor n/p \rfloor)^{-1} \sum_{i=1}^{\lfloor n/p \rfloor} \mathbb{I}\{ \omega(\overline{\bf Y}_{pi-p}^{pi-1}, {\bf W}_{pi-p}^{pi-1}) > \delta\}$ converges to $1-\epsilon$ $\mathbb{P}_{\theta^*}$-a.s., $1-\omega(\overline{\bf Y}_{pi-p}^{pi-1}, {\bf W}_{pi-p}^{pi-1})  \leq  1-\delta$ holds for a large fraction of the $\omega(\overline{\bf Y}_{pi-p}^{pi-1}, {\bf W}_{pi-p}^{pi-1})$'s. Consequently, we can establish a $\mathbb{P}_{\theta^*}$-a.s.\ bound on $n^{-1}|l_n(\theta,x_0) - l_n(\theta)|$.

We proceed to show that, for all $\theta \in \Theta$, $p_\theta(Y_k|\overline{\bf Y}^{k-1}_{-m},{\bf W}_{-m}^k)$ converges to $p_\theta(Y_k|\overline{\bf Y}^{k-1}_{-\infty},{\bf W}_{-\infty}^k)$ $\mathbb{P}_{\theta^*}$-a.s.\ as $m \to \infty$ and that we can approximate $n^{-1}l_n(\theta)$ by $n^{-1}\sum_{k=1}^n \log p_\theta(Y_k|\overline{\bf Y}^{k-1}_{-\infty},{\bf W}_{-\infty}^k)$, which is the sample average of the stationary ergodic random variables. For $x \in \mathcal{X}$ and $m \geq 0$, define
\begin{align*}
\Delta_{k,m,x}(\theta) &:= \log p_\theta(Y_k|\overline{\bf Y}^{k-1}_{-m},{\bf W}_{-m}^k,X_{-m}=x), \\
\Delta_{k,m}(\theta) &:= \log p_\theta(Y_k|\overline{\bf Y}^{k-1}_{-m},{\bf W}_{-m}^k) \\
& = \log \int p_\theta(Y_k|\overline{\bf Y}^{k-1}_{-m},{\bf W}_{-m}^k,X_{-m}=x_{-m}) \mathbb{P}_\theta(dx_{-m}|\overline{\bf Y}^{k-1}_{-m},{\bf W}_{-m}^k),
\end{align*}
so that $l_n(\theta)=\sum_{k=1}^n \Delta_{k,0}(\theta)$. The following proposition corresponds to Lemma 3 of DMR. This proposition shows that, for any $k \geq 0$, the sequences $\{\Delta_{k,m}(\theta)\}_{m \geq 0}$ and $\{\Delta_{k,m,x}(\theta)\}_{m \geq 0}$ are Cauchy   uniformly in $\theta \in \Theta$.
\begin{lemma} \label{lemma3_dmr}
Assume Assumptions \ref{assn_a1}--\ref{assn_gbound2}. Then, there exist a constant $\rho \in (0,1)$ and random sequences $\{A_{k,m}\}_{k\geq 1, m \geq 0}$ and $\{B_k\}_{k\geq 1}$ such that, for all $1 \leq k \leq n$ and $m' \geq m \geq 0$, 
\begin{align*}
(a) & \quad \sup_{x,x' \in \mathcal{X}} \sup_{\theta \in \Theta} \left|\Delta_{k,m,x}(\theta)-\Delta_{k,m',x'}(\theta)\right| \leq A_{k,m} \rho^{\lfloor (k+m)/3p \rfloor}, \\
(b) & \quad \sup_{x \in \mathcal{X}} \sup_{\theta \in \Theta} \left|\Delta_{k,m,x}(\theta)-\Delta_{k,m}(\theta)\right| \leq A_{k,m} \rho^{\lfloor (k+m)/3p \rfloor}, \\
(c) & \quad \sup_{m \geq 0} \sup_{x \in \mathcal{X}} \sup_{\theta \in \Theta} \left|\Delta_{k,m,x}(\theta) \right|+ \sup_{m \geq 0} \sup_{\theta \in \Theta}\left|\Delta_{k,m}(\theta) \right| \leq B_k, 
\end{align*}
where $\mathbb{P}_{\theta^*}\left(A_{k,m} \geq M\ \text{i.o.}\right)=0$ for a constant $M <\infty$ and $B_k \in L^1(\mathbb{P}_{\theta^*})$.
\end{lemma}
Lemma \ref{lemma3_dmr}(a) implies that $\{\Delta_{k,m,x}(\theta)\}_{m \geq 0}$ is a uniform Cauchy sequence in $\theta \in \Theta$ with probability one and that $\lim_{m\to\infty} \Delta_{k,m,x}(\theta)$ does not depend on $x$. Let $\Delta_{k,\infty}(\theta)$ denote this limit. Because $\{\Delta_{k,m,x}(\theta)\}_{m \geq 0}$ is uniformly bounded in $L^1(\mathbb{P}_{\theta^*})$ from Lemma \ref{lemma3_dmr}(c), $\{\Delta_{k,m,x}(\theta)\}_{m \geq 0}$ converges to $\Delta_{k,\infty}(\theta)$ in $L^1(\mathbb{P}_{\theta^*})$ and $\Delta_{k,\infty}(\theta) \in L^1(\mathbb{P}_{\theta^*})$ by the dominated convergence theorem. Define $l(\theta):= \mathbb{E}_{\theta^*}[\Delta_{0,\infty}(\theta)]$. Lemma \ref{lemma3_dmr} also implies that $n^{-1}l_n(\theta)$ converge to $n^{-1}\sum_{k=1}^n\Delta_{k,\infty}(\theta)$, which converges to $l(\theta)$ by the ergodic theorem. Therefore, the consistency of $\hat \theta_{x_0}$ is proven if this convergence of $n^{-1}l_n(\theta)-l(\theta)$ is strengthened to uniform convergence in $\theta \in \Theta$ and the additional regularity conditions are confirmed.

We introduce additional assumptions on the continuity of $q_\theta$ and $g_\theta$ and identification of $\theta^*$. 
\begin{assumption} \label{assn_consis}
(a) For all $(\overline{\bf y},y',w) \in \mathcal{Y}^s \times \mathcal{Y} \times \mathcal{W}$ and uniformly in $x,x' \in \mathcal{X}$, $q_{\theta}(x,x')$ and $g_\theta(y'|\overline{\bf y}, x, w)$ are continuous in $\theta$. (b) $\mathbb{P}_{\theta^*}[p_{\theta^*}(Y_1 |\overline{\bf{Y}}_{-m}^0,{\bf W}_{-m}^1) \neq p_{\theta}(Y_1 |\overline{\bf{Y}}_{-m}^0,{\bf W}_{-m}^1) ]>0$ for all $m \geq 0$ and all $\theta \in \Theta$ such that $\theta \neq \theta^*$.
\end{assumption}
Assumption \ref{assn_consis}(b) is a high-level assumption because it is imposed on $p_{\theta}(Y_1 |\overline{\bf{Y}}_{-m}^0,{\bf W}_{-m}^1)$. When the covariate $W_k$ is absent, DMR prove consistency under a lower-level assumption (their (A5$'$)), which is stated in terms of $p_{\theta}({\bf Y}_1^n |\overline{\bf{Y}}_0)$. We use Assumption \ref{assn_consis}(b) for brevity.

The following proposition shows the strong consistency of the (conditional) MLE. 
\begin{proposition} \label{prop_consistency}
Assume Assumptions \ref{assn_a1}--\ref{assn_consis}. Then, $\sup_{x_0 \in \mathcal{X}} |\hat \theta_{x_0} - \theta^*| \to 0$ $\mathbb{P}_{\theta^*}$-a.s.
\end{proposition}
\citet[][Theorem 3]{francq98stat} prove the consistency of the MLE when the state space of $X_k$ is finite. Proposition \ref{prop_consistency} generalizes Theorem 3 of \citet{francq98stat} in the following three aspects. First, we allow $X_k$ to be continuously distributed. Second, we analyze the log-likelihood function conditional on $X_0=x_0$, whereas \citet{francq98stat} set the initial distribution of $X_1$ to any probability vector with strictly positive elements. In other words, we allow for zeros in the postulated initial distribution of $\{X_k\}$. Third, we allow for an exogenous covariate $\{W_k\}_{k=0}^n$. \citet{leroux92spa},  \citet{legrandmevel00math}, and \citet{doucmatias01bernouiil} analyze the asymptotic property of the MLE of hidden Markov models, which are the special case of the model considered here in that $g_\theta(Y_k|\overline{\bf Y}_{k-1},X_k,W_k)$ does not depend on $\overline{\bf Y}_{k-1}$.
 
Define the MLE with a probability measure $\xi$ on $\mathcal{B}(\mathcal{X})$ for $x_0$ as $\hat \theta_{\xi}:= \argmax_{\theta \in \Theta} l_n(\theta,\xi)$ with $l_n(\theta,\xi)$ defined in (\ref{l_n_xi}). Proposition \ref{prop_consistency} implies the following corollary. 
\begin{corollary} \label{cor_consistency_xi}
Assume Assumptions \ref{assn_a1}--\ref{assn_consis}. Then, for any $\xi$, $\hat \theta_{\xi} \to \theta^*$ $\mathbb{P}_{\theta^*}$-a.s.
\end{corollary}

\section{Asymptotic distribution of the MLE}

In this section, we derive the asymptotic distribution of the MLE and consistency of the asymptotic covariance matrix estimate. Because $\hat \theta_{x_0}$ is consistent, expanding the first-order condition $\nabla_\theta l_n(\hat\theta_{x_0},x_0)=0$ around $\theta^*$ gives
\begin{equation} \label{taylor}
0 = \nabla_\theta l_n(\hat\theta_{x_0},x_0) = \nabla_\theta l_n(\theta^*,x_0) + \nabla_{\theta}^2 l_n(\overline\theta,x_0) (\hat \theta_{x_0} - \theta^*),
\end{equation}
where $\overline\theta \in [\theta^*,\hat\theta_{x_0}]$ and $\overline\theta$ may take different values across different rows of $\nabla_{\theta}^2 l_n(\overline\theta,x_0)$. In the following, we approximate $\nabla_{\theta}^j l_n(\theta,x_0) =\sum_{k=1}^n \nabla_{\theta}^j \log p_\theta (Y_k| \overline{\bf Y}_{0}^{k-1},{\bf W}_{0}^{k},X_0=x_0)$ for $j=1,2$ by $\sum_{k=1}^n \nabla_{\theta}^j \log p_\theta (Y_k| \overline{\bf Y}_{-\infty}^{k-1},{\bf W}_{-\infty}^{k})$, which is a sum of a stationary process. We then apply the central limit theorem and law of large numbers to $n^{-j/2}\sum_{k=1}^n \nabla_{\theta}^j \log p_\theta (Y_k| \overline{\bf Y}_{-\infty}^{k-1},{\bf W}_{-\infty}^{k})$. A similar expansion gives the asymptotic distribution of $n^{1/2}(\hat \theta_{\xi} - \theta^*)$.

We introduce additional assumptions. Define $\mathcal{X}_\theta^+:=\{(x,x') \in \mathcal{X}^2: q_\theta(x,x')>0\}$. 
\begin{assumption}\label{assn_distn}
There exists a constant $\delta>0$ such that the following conditions hold on $G:= \{\theta\in\Theta: |\theta-\theta^*| < \delta\}$: (a) For all $(\overline{\bf{y}},y',w,x,x')\in \mathcal{Y}^s\times \mathcal{Y}\times \mathcal{W} \times \mathcal{X}\times \mathcal{X}$, the functions $g_{\theta}(y'|\overline{\bf y},w,x)$ and $q_{\theta}(x,x')$ are twice continuously differentiable in $\theta \in G$. (b) $\sup_{\theta\in G} \sup_{x,x'\in \mathcal{X}_\theta^+} | \nabla_{\theta}\log q_{\theta}(x,x')|< \infty$ and $\sup_{\theta\in G} \sup_{x,x'\in \mathcal{X}_\theta^+} | \nabla_{\theta}^2\log q_{\theta}(x,x')|< \infty$. (c) $\mathbb{E}_{\theta^*}[\sup_{\theta\in G} \sup_{x\in \mathcal{X}} | \nabla_{\theta}\log g_{\theta}(Y_1|\overline{\bf{Y}}_0,x,W_1) |^{2}]< \infty$ and $\mathbb{E}_{\theta^*}[\sup_{\theta\in G} \sup_{x\in \mathcal{X}} | \nabla_{\theta}^2\log g_{\theta}(Y_1|\overline{\bf{Y}}_0,x,W_1) |]< \infty$. (d) For almost all $(\overline{\bf{y}},y',w)\in \mathcal{Y}^s \times \mathcal{Y}\times \mathcal{W}$, there exists a function $f_{\overline{\bf{y}},y',w}: \mathcal{X}\rightarrow \mathbb{R}^+$ in $L^1(\mu)$ such that $\sup_{\theta\in G} g_{\theta}(y'|\overline{\bf y},x,w) \leq f_{\overline{\bf y},y',w}(x)$. (e) For almost all $(x,\overline{\bf{y}},w) \in \mathcal{X} \times \mathcal{Y}^s\times \mathcal{W}$ and $j=1,2$, there exist functions $f^j_{x,\overline{\bf y},w}: \mathcal{Y}\rightarrow \mathbb{R}^+$ in $L^1(\nu)$ such that $|\nabla_{\theta}^j g_{\theta}(y'|\overline{\bf y},x,w) |\leq f^j_{x,\overline{\bf y},w}(y')$ for all $\theta\in G$.
\end{assumption} 
\begin{assumption}\label{assn_nabla_moment}
$\mathbb{E}_{\theta^*}[\sup_{m \geq 0} \sup_{\theta \in G} | \nabla_{\theta}\log p_{\theta}(Y_1|\overline{\bf{Y}}^0_{-m},{\bf W}^{1}_{-m}) |^{2}]< \infty$, \\
$\mathbb{E}_{\theta^*}[\sup_{m \geq 0}\sup_{\theta \in G} | \nabla_{\theta}^2\log p_{\theta}(Y_1|\overline{\bf{Y}}^0_{-m},{\bf W}^{1}_{-m}) |]< \infty$, \\
$\mathbb{E}_{\theta^*}[\sup_{m \geq 0} \sup_{\theta \in G}\sup_{x \in \mathcal{X}} | \nabla_{\theta}\log p_{\theta}(Y_1|\overline{\bf{Y}}^0_{-m},{\bf W}^{1}_{-m}, X_{-m}=x) |^{2}]< \infty$, and \\
$\mathbb{E}_{\theta^*}[\sup_{m \geq 0}\sup_{\theta \in G}\sup_{x \in \mathcal{X}} | \nabla_{\theta}^2\log p_{\theta}(Y_1|\overline{\bf{Y}}^0_{-m},{\bf W}^{1}_{-m}, X_{-m}=x) |]< \infty$.
\end{assumption} 
Assumption \ref{assn_distn} is the same as Assumptions (A6)--(A8) of DMR except for accommodating the case $\inf_{(x,x') \in \mathcal{X}^2}q_\theta(x,x') =0$ and the covariate $W$.  In Assumption \ref{assn_distn}(b), the supremum is taken over $\mathcal{X}_\theta^+$ because $\nabla_{\theta}\log q_\theta(x,x')$ and $\nabla_{\theta}^2 \log q_\theta(x,x')$ are not well-defined when $q_\theta(x,x')=0$. Examples \ref{example_1}--\ref{example_4} satisfy Assumption \ref{assn_distn}. Assumption \ref{assn_nabla_moment} is a high-level assumption that bounds the moments of $\nabla_{\theta}^j \log p_\theta (Y_k| \overline{\bf Y}_{-m}^{k-1},{\bf W}_{-m}^{k})$ and $\nabla_{\theta}^j \log p_\theta (Y_k| \overline{\bf Y}_{-m}^{k-1},{\bf W}_{-m}^{k},X_{-m}=x)$ uniformly in $m$. When $p=1$, DMR could derive Assumption \ref{assn_nabla_moment} by using the $L^{3-j}(\mathbb{P}_{\theta^*})$ convergence of $\nabla_{\theta}^j \log p_\theta (Y_k| \overline{\bf Y}_{-m}^{k-1},{\bf W}_{-m}^{k})$ and $\nabla_{\theta}^j \log p_\theta (Y_k| \overline{\bf Y}_{-m}^{k-1},{\bf W}_{-m}^{k},X_{-m}=x)$ to $\nabla_{\theta}^j \log p_\theta (Y_k| \overline{\bf Y}_{-\infty}^{k-1},{\bf W}_{-\infty}^{k})$ as $m \to \infty$. When $p \geq 2$, we need to assume Assumption \ref{assn_nabla_moment} because our Lemma \ref{lemma_psi_bound} only shows that these sequences converge to $\nabla_\theta^j \log p_\theta({\bf Y}_k|\overline{\bf Y}^{k-1}_{-\infty},{\bf W}_{-\infty}^k)$ in probability.
 
\subsection{Asymptotic distribution of the score function}

This section derives the asymptotic distribution of $n^{-1/2}\nabla_\theta l_n(\theta^*,x_0)$ and $n^{-1/2}\nabla_\theta l_n(\theta^*,\xi)$. We introduce a result known as the Louis missing information principle \citep{louis82jrssb}, which expresses the derivatives of the log-likelihood function of a latent variable model in terms of the conditional expectation of the derivatives of the \emph{complete data} log-likelihood function. Let $(X,Y,W)$ be random variables with $p_\theta(y,x|w)$ denoting the joint density of $(Y,X)$ given $W$, and let $p_\theta(y|w)$ be the marginal density of $Y$ given $W$. Then, a straightforward differentiation that is valid under Assumption \ref{assn_distn} gives $\nabla_\theta \log p_\theta(Y|W)  = \mathbb{E}_{\theta}\left[ \nabla_\theta \log p_\theta(Y,X|W) \middle| Y,W \right]$. 
 In terms of the variables in our model, we have, for any $k \geq 1$ and $m \geq 0$,
\begin{equation} \label{louis_p}
\begin{aligned}
& \nabla_\theta \log p_\theta({\bf Y}_{-m+1}^k|\overline{\bf Y}_{-m},{\bf W}_{-m}^k,X_{-m}) \\
& = \mathbb{E}_{\theta}\left[ \nabla_\theta \log p_\theta({\bf Y}_{-m+1}^k, {\bf X}_{-m+1}^k|\overline{\bf Y}_{-m},{\bf W}_{-m}^k,X_{-m})\middle| \overline{\bf Y}_{-m}^{k},{\bf W}_{-m}^k,X_{-m} \right] \\
& = \mathbb{E}_{\theta}\left[ \sum_{t=-m+1}^k \nabla_\theta \log p_\theta(Y_t,X_t|\overline{\bf Y}_{t-1}, X_{t-1}, W_t)\middle| \overline{\bf Y}_{-m}^{k},{\bf W}_{-m}^k,X_{-m} \right],
\end{aligned}
\end{equation} 
%From (\ref{louis_p}), the term $\nabla_\theta \log p_{\theta}({\bf Y}^k_1|\overline{\bf Y}_{0}, {\bf W}^{k}_{0}, X_{0}= x_0)$ in (\ref{dl_x0}) is written as
%\begin{equation} \label{dlogp_y}
%\begin{aligned}
%& \nabla_\theta \log p_{\theta}({\bf Y}^k_1|\overline{\bf Y}_{0}, {\bf W}^{k}_{0}, X_{0}= x_0) \\
%& = \mathbb{E}_{\theta}\left[ \nabla_\theta \log p_\theta({\bf Y}_{1}^k, {\bf X}_{1}^k|\overline{\bf Y}_{0},{\bf W}_{0}^k,X_{0}=x_0 )\middle| \overline{\bf Y}_{0}^{k},{\bf W}_{0}^k,X_{0}=x_0 \right] \\
%& = \mathbb{E}_{\theta}\left[ \sum_{t=1}^k \nabla_\theta \log p_\theta(Y_t,X_t|\overline{\bf Y}_{t-1}, X_{t-1}, W_t)\middle| \overline{\bf Y}_{0}^{k},{\bf W}_{0}^k,X_{0}=x_0 \right],
%\end{aligned}
%\end{equation}
where the last equality follows from Assumption \ref{assn_a2}. %The term $\nabla_\theta \log p_{\theta}({\bf Y}^{k-1}_1|\overline{\bf Y}_{0}, {\bf W}^{k}_{0}, X_{0}= x_0)$ in (\ref{dl_x0}) admits a similar expression.

Define $\overline{\bf Z}_{k-1}^k:=(Y_k,X_{k},\overline{\bf Y}_{k-1},X_{k-1})$. For $j=1,2$, denote the derivatives of the complete data log-density of $(Y_k,X_{k})$ given $(\overline{\bf Y}_{k-1},X_{k-1},W_k)$ by
\begin{align*} 
\phi^j(\theta,\overline{\bf Z}_{k-1}^k,W_k) & :=\nabla_{\theta}^j \log p_{\theta}(Y_k,X_k|\overline {\bf Y}_{k-1}, X_{k-1},W_k) \\
& = \nabla_{\theta}^j \log q_{\theta}(X_{k-1},X_{k}) + \nabla_{\theta}^j \log g_{\theta}(Y_k,|\overline {\bf Y}_{k-1},X_{k},W_k).
\end{align*}
We use a short-handed notation $\phi^j_{\theta k}:=\phi^j(\theta,\overline{\bf Z}_{k-1}^k,W_k)$. We also suppress the superscript $1$ from $\phi^1_{\theta k}$, so that $\phi_{\theta k}=\phi^1_{\theta k}$. Let $|\phi^j_k|_{\infty}:= \sup_{\theta\in G} \sup_{x,x'\in \mathcal{X}_\theta^+}|\nabla_\theta^j \log q_\theta(x,x')| \\+ \sup_{\theta\in G} \sup_{x \in \mathcal{X}}|\nabla_{\theta}^j \log g_\theta(Y_k|\overline{\bf{Y}}_{k-1},x,W_k)|$.
Define, for $x \in \mathcal{X}$, $k \geq 1$, $m \geq 0$, and $j=1,2$,\footnote{DMR (page 2272) use the symbol $\Delta_{k,m,x}(\theta)$ to denote our $\Psi_{k,m,x}^1(\theta)$, but we use $\Psi_{k,m,x}(\theta)$ to avoid confusion with $\Delta_{k,m,x}(\theta)$ used in Lemma \ref{lemma3_dmr}.} 
\begin{align}
\Psi_{k,m,x}^j(\theta) & := \mathbb{E}_{\theta}\left[ \sum_{t=-m+1}^k \phi^j_{\theta t} \middle| \overline{\bf Y}^{k}_{-m}, {\bf W}^{k}_{-m}, X_{-m}= x \right] - \mathbb{E}_{\theta}\left[ \sum_{t=-m+1}^{k-1} \phi^j_{\theta t} \middle| \overline{\bf Y}^{k-1}_{-m}, {\bf W}^{k-1}_{-m}, X_{-m}= x \right]. \label{Psi_x_defn}
\end{align} 
It follows from  (\ref{louis_p}) and (\ref{Psi_x_defn}) that $\Psi_{k,m,x}^1(\theta) =  \nabla_\theta \log p_{\theta}({\bf Y}^k_{-m+1}|\overline{\bf Y}_{-m}, {\bf W}^{k}_{-m}, X_{-m}=x) -$ \\$ \nabla_\theta \log p_{\theta}({\bf Y}^{k-1}_{-m+1}|\overline{\bf Y}_{-m}, {\bf W}^{k-1}_{-m}, X_{-m}=x) = \nabla_\theta \log p_{\theta}(Y_k|\overline{\bf Y}^{k-1}_{-m}, {\bf W}^{k}_{-m}, X_{-m}=x)$. Therefore,   we can express $\nabla_\theta l_n(\theta,x_0)$ as
\[
\nabla_\theta l_n(\theta,x_0) = \sum_{k=1}^n \nabla_\theta \log p_{\theta}(Y_k|\overline{\bf Y}^{k-1}_{0}, {\bf W}^{k}_{0}, X_{0}= x_0) = \sum_{k=1}^n \Psi_{k,0,x_0}^1(\theta).
\]

Lemma \ref{lemma_psi_bound} below shows that $\{\Psi_{k,m,x}^j(\theta)\}_{m \geq 0}$ is a Cauchy sequence that converges to a limit at an exponential rate in probability. Note that $\Psi_{k,m,x}^j(\theta)$ is a function of $\mathbb{E}_{\theta}[ \phi^j_{\theta t} | \cdot]$ for $t=-m+1,\ldots,k$. When $t$ is large, the difference between $\mathbb{E}_{\theta}[ \phi^j_{\theta t} |\overline{\bf Y}^{k}_{-m}, {\bf W}^{k}_{-m}, X_{-m}= x]$ and $\mathbb{E}_{\theta}[ \phi^j_{\theta t} |\overline{\bf Y}^{k}_{-m'}, {\bf W}^{k}_{-m'}, X_{-m'}= x']$ with $m' >m$ is small because the chain $\{X_t\}_{t=-m'}^{k}$ conditional on $(\overline{\bf Y}_{-m'}^k, {\bf W}_{-m'}^k)$ forgets its past (i.e., $\overline{\bf Y}_{-m'}^{m}$, ${\bf W}_{-m'}^{m}$, and $X_{-m}$) at an exponential rate by virtue of Lemma \ref{corollary_1}. When $t$ is small,  the term $\mathbb{E}_{\theta}[ \phi^j_{\theta t} |\overline{\bf Y}^{k}_{-m}, {\bf W}^{k}_{-m}, X_{-m}= x] - \mathbb{E}_{\theta}[ \phi^j_{\theta t} |\overline{\bf Y}^{k-1}_{-m}, {\bf W}^{k-1}_{-m}, X_{-m}= x]$ in $\Psi_{k,m,x}^j(\theta)$ is small because Lemma \ref{lemma_dmr_39} in the appendix shows that the time-reversed process $\{X_{k-t}\}_{0 \leq t\leq k+m}$ conditional on $(\overline{\bf Y}^{k}_{-m},{\bf W}^{k}_{-m})$ forgets its initial condition (i.e., $Y_k$ and $W_{k}$) at an exponential rate. 

Define, for $k \geq 0$, $m \geq 0$, and $j=1,2$,
\begin{align*}
\Psi_{k,m}^j(\theta) &:= \mathbb{E}_{\theta}\left[ \sum_{t=-m+1}^k \phi^j_{\theta t} \middle| \overline{\bf Y}^{k}_{-m}, {\bf W}^{k}_{-m}\right] - \mathbb{E}_{\theta}\left[ \sum_{t=-m+1}^{k-1} \phi^j_{\theta t} \middle| \overline{\bf Y}^{k-1}_{-m}, {\bf W}^{k-1}_{-m}\right].
\end{align*}
Note that $\Psi_{k,m}^1(\theta) = \nabla_\theta \log p_{\theta}(Y_k|\overline{\bf Y}^{k-1}_{-m}, {\bf W}^{k}_{-m})$. From Lemma \ref{corollary_1} and Lemma \ref{lemma_dmr_39}, we obtain the following bound on $\Psi_{k,m,x}^j(\theta)-\Psi_{k,m}^j(\theta)$ and $\Psi_{k,m,x}^j(\theta)-\Psi_{k,m',x'}^j(\theta)$.
\begin{lemma} \label{lemma_psi_bound}
Assume Assumptions \ref{assn_a1}--\ref{assn_nabla_moment}. Then, for $j=1,2$, there exist a constant $\rho \in (0,1)$, random sequences $\{A_{k,m}\}_{k\geq 1, m \geq 0}$ and $\{B_{m}\}_{m\geq 0}$, and a random variable $K_j \in L^{3-j}(\mathbb{P}_{\theta^*})$ such that, for all $1 \leq k \leq n$ and $m' \geq m \geq 0$, 
\begin{align*}
(a) & \quad  \sup_{x \in \mathcal{X}} \sup_{\theta \in G}\ \left|\Psi_{k,m,x}^j(\theta)-\Psi_{k,m}^j(\theta)\right| \leq K_j (k + m)^2 	\rho^{\lfloor (k+m)/4(p+1)\rfloor} A_{k,m}, \\
(b) & \quad \sup_{x,x' \in \mathcal{X}}  \sup_{\theta \in G}\ \left|\Psi_{k,m,x}^j(\theta)-\Psi_{k,m',x'}^j(\theta)\right| \leq [ K_j (k + m)^2 + B_m] \rho^{\lfloor (k+m)/4(p+1)\rfloor} A_{k,m}, 
\end{align*}
where $\mathbb{P}_{\theta^*}\left(A_{k,m} \geq 1\ \text{i.o.}\right)=0$, $B_m < \infty$ $\pstar$-a.s., and the distribution function of $B_m$ does not depend on $m$.
\end{lemma}
Because $B_m \rho^{\lfloor (k+m)/4(p+1)\rfloor/2} \to_p 0$ as $m \to \infty$, Lemma \ref{lemma_psi_bound} implies that $\{\Psi_{k,m,x}^{1}(\theta)\}_{m\geq 0}$ converges to $\Psi_{k,\infty}^{1}(\theta) = \nabla_{\theta} \log p_{\theta}(Y_k|\overline{\bf Y}^{k-1}_{-\infty}, {\bf W}^k_{-\infty})$ in probability uniformly in $\theta \in G$ and $x \in \mathcal{X}$. Define the filtration $\mathcal{F}$ by $\mathcal{F}_k:=\sigma((\overline{\bf{Y}}_{i},W_{i+1}): -\infty < i\leq k)$. It follows from $\mathbb{E}_{\theta^*}[\Psi_{k,m}^1(\theta^*)|\overline{\bf Y}^{k-1}_{-m}, {\bf W}^{k}_{-m}]=0$, Assumption \ref{assn_nabla_moment}, and combining Exercise 2.3.7 and Theorem 5.5.9 of \citet{durrett10book} that \\$\mathbb{E}_{\theta^*}[\Psi_{k,\infty}^1(\theta^*)|\overline{\bf{Y}}_{-\infty}^{k-1}, {\bf W}^{k}_{-\infty}]=0$ and $I(\theta^*):=\mathbb{E}_{\theta^*}[\Psi_{0,\infty}^1(\theta^*)(\Psi_{0,\infty}^1(\theta^*))'] < \infty$. Therefore, \\ $\{\Psi_{k,\infty}^1(\theta^*)\}_{k=-\infty}^{\infty}$ is an $(\mathcal{F},\mathbb{P}_{\theta^*})$-adapted stationary, ergodic, and square integrable martingale difference sequence, to which a martingale central limit theorem is applicable.

 Setting $m=0$ and letting $m' \to \infty$ in Lemma \ref{lemma_psi_bound} shows that \\$n^{-1/2}\sum_{k=1}^n \Psi_{k,0,x_0}^1(\theta^*) - n^{-1/2}\sum_{k=1}^n \Psi_{k,\infty}^{1}(\theta^*)$ is bounded by $n^{-1/2}\sum_{k=1}^n k^2 \tilde \rho^k$ in probability for some $\tilde \rho \in (0,1)$. Consequently, as the following proposition shows, the score function is asymptotically normally distributed. 
\begin{proposition}\label{prop_score}
Assume Assumptions \ref{assn_a1}--\ref{assn_nabla_moment}. Then, (a) for any $x_0 \in \mathcal{X}$, $n^{-1/2}\nabla_\theta l_n(\theta^*,x_0) \to_d N(0,I(\theta^*))$; (b) for any probability measure $\xi$ on $\mathcal{B}(\mathcal{X})$ for $x_0$, $n^{-1/2}\nabla_\theta l_n(\theta^*,\xi) \to_d N(0,I(\theta^*))$.
\end{proposition}

\subsection{Convergence of the Hessian}

This section derives the probability limit of $n^{-1}\nabla_{\theta}^2 l_n(\theta,x_0)$ and $n^{-1}\nabla_{\theta}^2 l_n(\theta,\xi)$ when $\theta$ is in a neighborhood of $\theta^*$. 
 The Louis missing information principle for the second derivative is given by $\nabla_{\theta}^2 \log p_\theta(Y|W) = \mathbb{E}_{\theta}\left[ \nabla_{\theta}^2 \log p_\theta(Y,X|W) \middle| Y,W \right] + \text{var}_{\theta}\left[ \nabla_\theta \log p_\theta(Y,X|W)\middle| Y,W \right]$. In terms of the variables in our model, we have, for any $k \geq 1$ and $m \geq 0$,
\begin{equation} \label{louis2}
\begin{aligned}
&\nabla_{\theta}^2 \log p_\theta({\bf Y}_{-m+1}^k|\overline{\bf Y}_{-m},{\bf W}_{-m}^k,X_{-m}) \\
 = &\mathbb{E}_{\theta}\left[ \nabla_\theta^2 \log p_\theta({\bf Y}_{-m+1}^k, {\bf X}_{-m+1}^k|\overline{\bf Y}_{-m},{\bf W}_{-m}^k,X_{-m})\middle| \overline{\bf Y}_{-m}^{k},{\bf W}_{-m}^k,X_{-m} \right] \\
&+ \text{var}_{\theta}\left[ \nabla_\theta \log p_\theta({\bf Y}_{-m+1}^k, {\bf X}_{-m+1}^k|\overline{\bf Y}_{-m},{\bf W}_{-m}^k,X_{-m})\middle| \overline{\bf Y}_{-m}^{k},{\bf W}_{-m}^k,X_{-m}  \right].
\end{aligned}
\end{equation}

Define
\begin{align} 
\Gamma_{k,m,x}(\theta) & := \text{var}_{\theta}\left[ \sum_{t=-m+1}^k \phi_{\theta t} \middle| \overline{\bf Y}^{k}_{-m}, {\bf W}^{k}_{-m}, X_{-m}= x\right] - \text{var}_{\theta}\left[ \sum_{t=-m+1}^{k-1} \phi_{\theta t} \middle| \overline{\bf Y}^{k-1}_{-m}, {\bf W}^{k-1}_{-m},X_{-m}= x\right], \label{gamma_x_defn} \\
\Gamma_{k,m}(\theta) & := \text{var}_{\theta}\left[ \sum_{t=-m+1}^k \phi_{\theta t} \middle| \overline{\bf Y}^{k}_{-m}, {\bf W}^{k}_{-m}\right] - \text{var}_{\theta}\left[ \sum_{t=-m+1}^{k-1} \phi_{\theta t} \middle| \overline{\bf Y}^{k-1}_{-m}, {\bf W}^{k-1}_{-m}\right]. \label{gamma_defn} 
\end{align}
From (\ref{Psi_x_defn})--(\ref{gamma_defn}), we can write $\nabla_{\theta}^2 l_n(\theta,x_0)$ in terms of $\{\Psi^2_{k,m,x}(\theta)\}$ and $\{\Gamma_{k,m,x}(\theta)\}$ as
\begin{align*}
\nabla_{\theta}^2 l_n(\theta,x_0) &= \sum_{k=1}^n \nabla_{\theta}^2 \log p_\theta(Y_k|\overline{\bf Y}^{k-1}_{0}, {\bf W}^{k}_{0}, X_{0}= x_0) = \sum_{k=1}^n [\Psi_{k,0,x_0}^2(\theta) + \Gamma_{k,0,x_0}(\theta) ]. 
\end{align*}
The following lemma provides the bounds on $\Gamma_{k,m,x}(\theta)$ that are analogous to Lemma \ref{lemma_psi_bound}.
\begin{lemma} \label{lemma_gamma_cgce}
Assume Assumptions \ref{assn_a1}--\ref{assn_nabla_moment}. Then, there exist a constant $\rho \in (0,1)$, random sequences $\{C_{k,m}\}_{k\geq 1, m \geq 0}$ and $\{D_{m}\}_{m\geq 0}$, and a random variable $K \in L^{1}(\mathbb{P}_{\theta^*})$ such that, for all $1 \leq k \leq n$ and $m' \geq m \geq 0$,
\begin{align*}
(a) & \quad  \sup_{x \in \mathcal{X}} \sup_{\theta \in G} \left|\Gamma_{k,m,x}(\theta)-\Gamma_{k,m}(\theta)\right| \leq K (k + m)^3	\rho^{\lfloor (k+m)/8(p+1)\rfloor} C_{k,m}, \\
(b) & \quad  \sup_{x,x' \in \mathcal{X}}  \sup_{\theta \in G} \left|\Gamma_{k,m,x}(\theta)-\Gamma_{k,m',x'}(\theta)\right| \leq K [(k + m)^3 +D_m]\rho^{\lfloor (k+m)/16(p+1)\rfloor} C_{k,m},
\end{align*}
where $\mathbb{P}_{\theta^*}\left(C_{k,m} \geq 1 \text{ i.o.}\right)=0$, $D_m < \infty$ $\pstar$-a.s. and the distribution function of $D_m$ does not depend on $m$.
\end{lemma}
Lemma \ref{lemma_gamma_cgce} implies that $\{\Gamma_{k,m,x}(\theta)\}_{m\geq 0}$ converges to $\Gamma_{k,\infty}(\theta)$ in probability uniformly in $x \in \mathcal{X}$ and $\theta \in G$. The following proposition is a local uniform law of large numbers for the observed Hessian.
\begin{proposition}\label{prop_hessian}
Assume Assumptions \ref{assn_a1}--\ref{assn_nabla_moment}. Then, 
\begin{align*}
\sup_{x\in \mathcal{X}} \sup_{\theta \in G}  \left| n^{-1}\nabla_{\theta}^2 l_n(\theta,x) - \mathbb{E}_{\theta^*}[\Psi_{0,\infty}^{2}(\theta) +\Gamma_{0,\infty}(\theta)] \right| \to_p 0.
\end{align*}
\end{proposition}
The following proposition shows the asymptotic normality of the MLE. \begin{proposition}\label{prop_distn}
Assume Assumptions \ref{assn_a1}--\ref{assn_nabla_moment}. Then, (a) for any $x_0 \in \mathcal{X}$, $n^{-1/2}(\hat \theta_{x_0} - \theta^*) \to_d N(0,I(\theta^*)^{-1})$; (b) for any probability measure $\xi$ on $\mathcal{B}(\mathcal{X})$ for $x_0$, $n^{-1/2}(\hat \theta_{\xi} - \theta^*) \to_d N(0,I(\theta^*)^{-1})$.
\end{proposition}

\subsection{Convergence of the covariance matrix estimate}

When conducting statistical inferences with the MLE, the researcher needs to estimate the asymptotic covariance matrix of the MLE. Proposition \ref{prop_hessian} already derived the consistency of the observed Hessian. We derive the consistency of the outer-product-of-gradients (OPG) estimates:
\begin{align}
\hat I_{x_0}(\theta) &:= n^{-1}\sum_{k=1}^n \nabla_\theta \log p_{\theta}(Y_k|\overline{\bf Y}^{k-1}_{0}, {\bf W}^{k}_{0}, x_0)(\nabla_\theta \log p_{\theta}(Y_k|\overline{\bf Y}^{k-1}_{0}, {\bf W}^{k}_{0}, x_0))', \label{I-x0}\\
\hat I_{\xi}(\theta) &:= n^{-1}\sum_{k=1}^n \nabla_{\theta} \log p_{\theta \xi}(Y_k|\overline{\bf Y}^{k-1}_{0}, {\bf W}^{k}_{0})(\nabla_\theta \log p_{\theta \xi}(Y_k|\overline{\bf Y}^{k-1}_{0}, {\bf W}^{k}_{0}))', \label{I-xi}
\end{align}
where $\nabla_\theta\log p_{\theta \xi}(Y_k|\overline{\bf Y}^{k-1}_{0}, {\bf W}^{k}_{0}):= \nabla_\theta \log \int p_{\theta}(Y_k|\overline{\bf Y}^{k-1}_{0}, {\bf W}^{k}_{0}, x_0) \xi(dx_0)$. In applications, \\$\nabla_\theta \log p_{\theta}(Y_k|\overline{\bf Y}^{k-1}_{0}, {\bf W}^{k}_{0}, x_0) $ can be computed by numerically differentiating $\log p_{\theta}(Y_k|\overline{\bf Y}^{k-1}_{0}, {\bf W}^{k}_{0}, x_0)$, which in turn can be computed by using the recursive algorithm of \citet{hamilton96joe}.

The following proposition shows the consistency of the OPG estimate. Its proof is similar to that of Proposition \ref{prop_hessian} and hence omitted.
\begin{proposition}\label{prop_opg}
Assume Assumptions \ref{assn_a1}--\ref{assn_nabla_moment}. Then, $\sup_{x_0 \in \mathcal{X}}| \hat I_{x_0}(\hat\theta) -I(\theta^*) | \to_p 0$ and $\hat I_\xi(\hat\theta) \to_p I(\theta^*)$ for any $\hat\theta$ such that $\hat\theta \to_p \theta^*$ and any $\xi$.
\end{proposition}

\section{Simulation}
As an illustration, we provide a small simulation study based on Hamilton's model (\ref{hamilton}) and the MS-CD model (\ref{MSCD}) with the Weibull distribution. The simulation was conducted with an R package we developed for Markov regime switching models.\footnote{The R package is available at \url{https://github.com/chiyahn/rMSWITCH}.}

\subsection{Hamilton's model}
We generate 1000 data sets of sample sizes $n=200, 400$, and $800$ from model (\ref{hamilton}) with $p=5$, using the parameter value taken from Table I of \citet{hamilton89em} with $\theta=(\mu_1,\mu_2,\gamma_1,\gamma_2,\gamma_3,\gamma_4,\sigma,p_{11},p_{22})'=(1.522,-0.3577,0.014,-0.058,-0.247,-0.213,0.7690,0.9049,0.7550)'$.\footnote{We simulate $(800 +n)$ periods and use the last $n$ observations as our sample, so that the initial value for our data set is approximately drawn from the stationary distribution.} For each data, we estimate the parameter $\theta$ together with the  initial distribution of $X_0$, $\xi$. Panel A of Table \ref{table1} reports the frequency at which the 95 percent confidence interval constructed from  (\ref{I-xi}) contains the true parameter value. The asymptotic 95 percent confidence intervals slightly undercover the true parameter at $n=200$ but the actual coverage probability approaches 95 percent as the sample size increases from $n=200$ to $400$, and then to $800$. Panel B of Table \ref{table1} presents the coverage probabilities when we use the estimator (\ref{I-x0}) by setting $x_0=2$ rather than (\ref{I-xi}). Consistent with our theoretical derivation, the results in Panel B of Table \ref{table1} are similar to those in Panel A of Table \ref{table1}, suggesting that the choice of the initial value of $x_0$ in constructing the covariance matrix estimate does not affect the coverage probabilities.

\begin{table}[h]\caption{Coverage probability of the asymptotic 95\% confidence intervals for Hamilton's model} \label{table1} \medskip
\centering
\small{
\begin{tabular}{c|ccccccccc}
\hline \hline 
\multicolumn{10}{c}{Panel A: 95\% confidence intervals constructed from  (\ref{I-xi})  }\\ \hline
& $p_{11}$ & $p_{21}$ & $\beta_1$ & $\beta_2$ & $\beta_3$ & $\beta_4$ & $\mu_1$ & $\mu_2$ & $\sigma$\\ \hline
$n=200$& 0.916 & 0.911 & 0.938 & 0.926 & 0.944 & 0.925 & 0.916 & 0.896 & 0.875\\
$n=400$&0.938 & 0.933 & 0.930 & 0.944 & 0.943 & 0.937 & 0.946 & 0.929 & 0.922\\
$n=800$&0.942 & 0.942 & 0.945 & 0.941 & 0.950 & 0.956 & 0.939 & 0.941 & 0.930\\ \hline\hline
\multicolumn{10}{c}{Panel B: 95\% confidence intervals constructed from   the OPG estimator}\\ \hline
& $p_{11}$ & $p_{21}$ & $\beta_1$ & $\beta_2$ & $\beta_3$ & $\beta_4$ & $\mu_1$ & $\mu_2$ & $\sigma$\\ \hline
$n=200$&0.915&0.920&0.938&0.927&0.941&0.934&0.922&0.901&0.884\\
$n=400$&0.932&0.932&0.938&0.949&0.942&0.939&0.945&0.929&0.923\\
$n=800$&0.943&0.945&0.945&0.939&0.949&0.956&0.936&0.937&0.929\\
 \hline 
\end{tabular}
} \\
\begin{flushleft}
Notes: Based on 1000 replications. Each entry in Panel A reports the frequency at which the asymptotic 95\% confidence interval constructed from (\ref{I-xi})  contains the true parameter value. Panel B reports the case of the asymptotic 95\% confidence interval constructed from (\ref{I-x0}) with $x_0=2$.
\end{flushleft}
\end{table}
  
\subsection{MS-CD model} 
 
We generate 1000 data sets of sample sizes $n=200, 400$, and $800$ from the MS-CD model (\ref{MSCD}), using the parameter value $\theta=  (\mu_1,\mu_2,\beta, \gamma, p_{11},p_{22},\Pr(X_0=1))' =(0.5, 1.2, 0.05, 0.95, 0.95, 0.95, 0.5)$, and examine the coverage probabilities of the asymptotic 95 percent confidence intervals. Panel A of Table \ref{table2} presents the coverage probabilities based on (\ref{I-xi}) and Panel B presents those based on (\ref{I-x0}) by setting $x_0=2$. The coverage probability improves as the sample size increases from $n=200$ to $n=800$ in both Panel A and Panel B. The results in Panels A and B are similar, indicating that the choice of the initial value of $x_0$ does not affect the confidence intervals.

\begin{table}[h]\caption{Coverage probability of the asymptotic 95\% confidence intervals for the MS-CD model } \label{table2} \medskip
\centering
\small{
\begin{tabular}{c|cccccc}
\hline \hline 
\multicolumn{7}{c}{Panel A: 95\% confidence intervals   from   (\ref{I-xi})  }\\ \hline
&$p_{11}$&$p_{22}$ & $\mu_1$ & $\mu_2$ & $\beta$ &$\gamma$\\ \hline
$n=200$&0.713&0.759&0.757&0.876&0.928&0.910  \\
$n=400$&0.827&0.844&0.861&0.902&0.956&0.919 \\
$n=800$&0.892&0.909&0.932&0.949&0.987&0.952 \\ \hline\hline
\multicolumn{7}{c}{Panel B: 95\% confidence intervals  from  (\ref{I-x0})}\\ \hline
&$p_{11}$&$p_{22}$ & $\mu_1$ & $\mu_2$ & $\beta$ &$\gamma$\\ \hline
$n=200$&0.713&0.759&0.757&0.876&0.928&0.910 \\
$n=400$&0.835&0.846&0.862&0.903&0.957&0.919 \\
$n=800$&0.894&0.910&0.932&0.950&0.987&0.950 \\
 \hline 
\end{tabular}
} \\
\begin{flushleft}
Notes: Based on 1000 replications. Each entry in Panel A reports the frequency at which the asymptotic 95\% confidence interval constructed from (\ref{I-xi})  contains the true parameter value. Panel B reports the case of the asymptotic 95\% confidence interval constructed from (\ref{I-x0}) with $x_0=2$.
\end{flushleft}
\end{table}

\section{Empirical application: Duration between stock price changes}
We estimate the MS-CD model (\ref{MSCD}) by using duration data taken from \citet{DeLucaGallo04snde} on the FIAT stock traded on the Milan Stock Exchange between May 2, 2000 and May 15, 2000, where the duration is defined as the time between every price change. We use their ``adjusted durations,'' which remove the daily seasonal component as well as exclude overnight durations between the first price change of a day and the last price change of the previous day. See \citet{DeLucaGallo04snde} for more details on the construction of their adjusted durations. 

We estimate the MS-CD model for $M=2$ and $3$. The regimes are ordered from the smallest to the largest in terms of the estimated values of $\mu_{X_k}$. For the model with $M=3$, we restrict some elements of the transition probability matrix so that $\Pr(X_k=3| X_{k-1}=1)=\Pr(X_k=1| X_{k-1}=3)=0$.\footnote{When the model is estimated without restricting the transition probabilities, the estimated transition probabilities between the first and third regimes are close to zero.} 

Table \ref{table3} reports the parameter estimates and their standard errors constructed from (\ref{I-xi}) for models with $M=2$ and $3$. For both $M=2$ and $3$, the estimated values of $\mu_{X_k}$ are well separated across regimes given the relatively small standard errors. The estimated values of $\gamma$ are $0.987$ and $1.005$ for the models with $M=2$ and $3$, respectively, providing some evidence that the density function is unbounded for the model with $M=2$.

%For both models, the estimated values of  diagonal elements of transition probability are close to 1,   indicating that each regime is highly persistent over time. 

The upper panel of Figure \ref{fig} shows the posterior probabilities of being in each regime for the model with $M=2$ for the first 3000 observations, where the solid red line represents the ``more frequent price changes'' regime (Regime 1), while the dotted blue line represents the ``less frequent price changes'' regime (Regime 2). Reflecting the high persistence of latent regimes, the posterior probabilities of being in each regime are either close to zero or one continuously over a prolonged period; the FIAT stock is in Regime 2 from the 1200th to 1900th observations and then switches to Regime 1 until the 2700th observation. As reported in the lower panel of Figure \ref{fig}, when the number of regimes is specified as $M=3$, the FIAT stock is the least frequently traded (Regime 3) from the 1200th to 1800th observations and most frequently traded (Regime 1) from the 1900th to 2100th observations as well as from the 2200th to 2500th observations.

%\begin{table}[h]\caption{Estimates of MS-CD model  for the Fiat stock duration} \label{table3} \medskip
%\centering
%\footnotesize{
%\begin{tabular}{c|cc|cc|cc}
%\hline \hline  
%&
%\multicolumn{2}{c|}{$M=2$}&
%\multicolumn{2}{c|}{$M=3$}&
%\multicolumn{2}{c}{$M=4$}\\
%&Estimate & S.D. &Estimate & S.D.&Estimate & S.D.\\  \hline
%$\mu_1$&0.483&0.013&0.359&0.016&0.345&0.020\\
%$\mu_2$&1.138&0.021&0.717&0.026&0.629&0.059\\
%$\mu_3$&&&1.290&0.037&1.072&0.101\\
%$\mu_4$&&&&&1.795&0.075\\  \hline
%$\beta$&0.053&0.012&0.032&0.015&0.005&0.024\\
%$\gamma$&0.987&0.009&1.005&0.012&1.016&0.015\\ \hline
%$p_{11}$&0.991&0.002&0.991&0.004&0.991&0.005\\
%$p_{21}$&&&0.003&0.001&0.004&0.003\\
%$p_{22}$&0.996&0.001&0.985&0.004&0.982&0.008\\
%$p_{32}$&&&&&0.006&0.004\\
%$p_{33}$&&&0.991&0.003&0.985&0.095\\
%$p_{44}$&&&&&0.964&0.025\\ \hline
%log-likelihood&\multicolumn{2}{c|}{-7037.90}&\multicolumn{2}{c|}{-6972.75}&\multicolumn{2}{c}{-6947.72} \\
% \hline 
%\end{tabular}
%} \\
%\begin{flushleft}
%Notes:  The asymptotic standard error is constructed from (\ref{I-xi}).  For the model with $M\geq 3$, we set $p_{ij}=0$ whenever $|i-j|\geq 2$. \end{flushleft}
%\end{table}

\begin{table}[h]\caption{Estimates of the MS-CD model for the FIAT stock duration} \label{table3} \medskip
\centering
\small{
\begin{tabular}{c|cc|cc }
\hline \hline  
&
\multicolumn{2}{c|}{$M=2$}&
\multicolumn{2}{c }{$M=3$} \\
&Estimate & S.D. &Estimate & S.D. \\  \hline
$\mu_1$&0.483&0.013&0.359&0.009\\
$\mu_2$&1.138&0.021&0.717&0.043\\
$\mu_3$&&&1.290&0.086\\
$\beta$&0.053&0.012&0.032&0.017\\
$\gamma$&0.987&0.009&1.005&0.011\\
$p_{11}$&0.991&0.002&0.991&0.009\\
$p_{21}$&&&0.003&0.004\\
$p_{22}$&0.996&0.001&0.984&0.008\\
$p_{33}$&&&0.991&0.008 
 \\ \hline
Log-likelihood&\multicolumn{2}{c|}{-7037.90}&\multicolumn{2}{c }{-6972.75}  \\
 \hline 
\end{tabular}
} \\
\begin{flushleft}
Notes: The asymptotic standard error is constructed from (\ref{I-xi}). For the model with $M= 3$, we set $p_{13}=p_{31}=0$. \end{flushleft}
\end{table}

 \begin{figure}[tb] % 
	\caption{The posterior probabilities of each regime: FIAT stock duration }\bigskip
	\centering
   \begin{minipage}{.9\linewidth}
   \centering 
   \bigskip
   \textbf{M=2}\\ \smallskip
   \includegraphics[width=480pt]{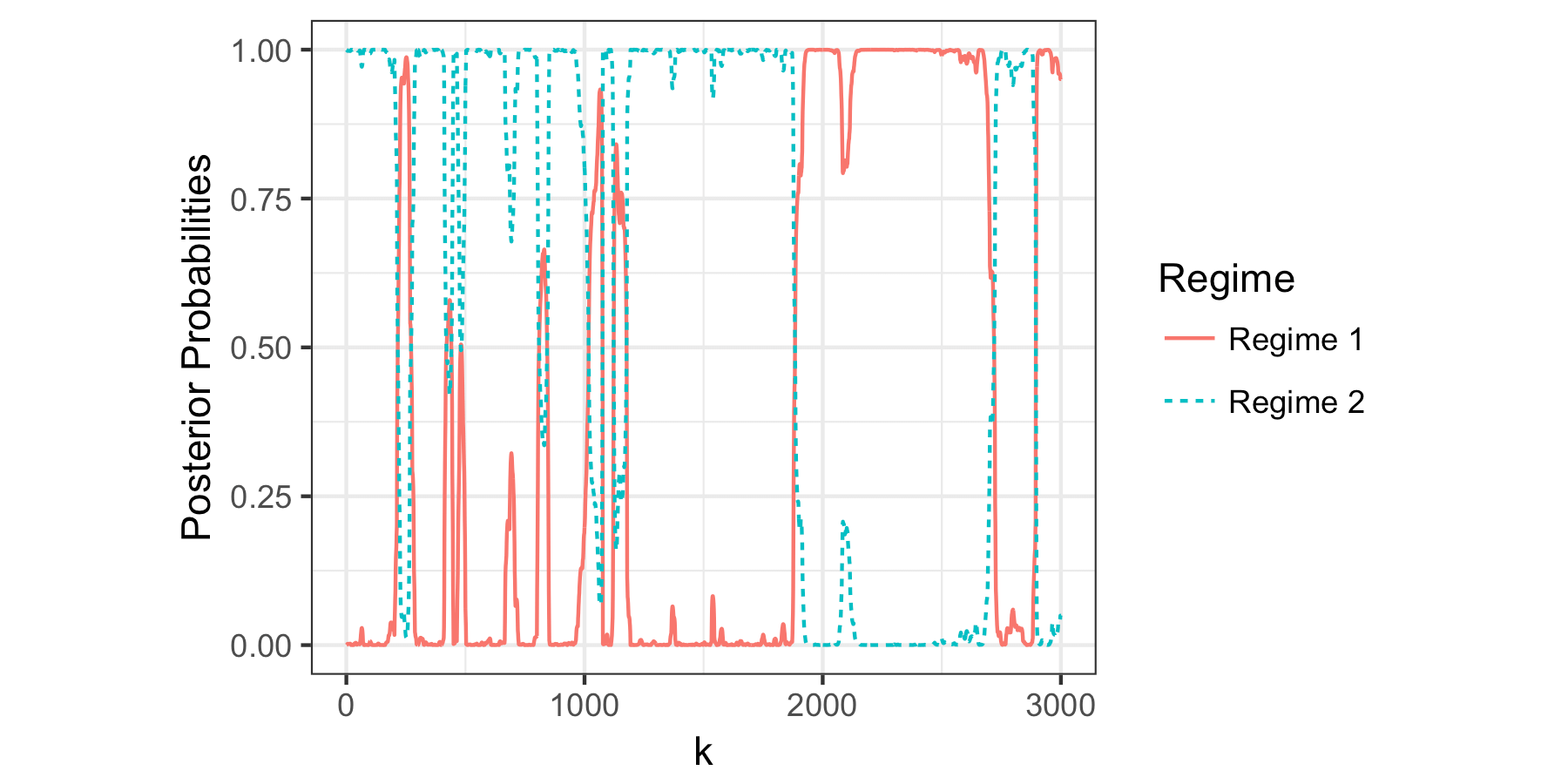}
    %\includegraphics[scale=0.5]{posterior_M2.png}
  % \label{fig:mYit}
   \end{minipage} \\ \bigskip \bigskip
   \begin{minipage}{.9\linewidth}
   \centering
   \medskip
   \textbf{M=3}\\ \smallskip
   \includegraphics[width=480pt]{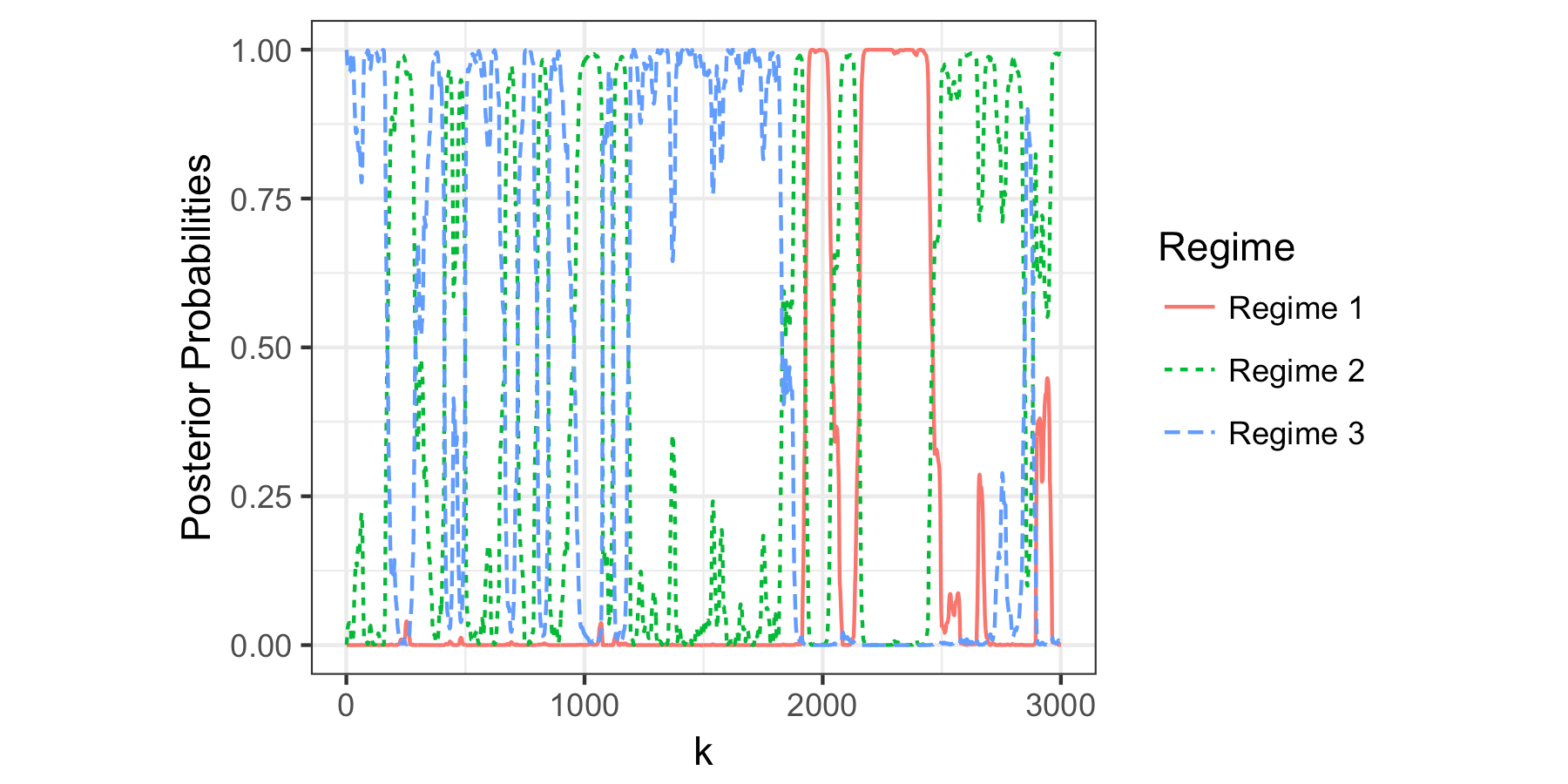}
%   \includegraphics[scale=0.5]{posterior_M3.png}
  % \label{fig:mYit}
   \end{minipage} 
\label{fig}
 \end{figure}
 
 \clearpage

\begin{appendices}

\section{Proofs}

Throughout these proofs, define $\overline{\bf V}_{b}^a:=(\overline{\bf{Y}}_{b}^a,{\bf W}_{b}^a)$.

\begin{proof}[Proof of Lemma \ref{corollary_1}] 
This lemma is an immediate consequence of Lemmas \ref{lemma_dmr_1} and \ref{lemma_coupling} when $-m+p \leq k \leq n$. When $k < - m + p$, this lemma holds because $\|\mu_1 - \mu_2\|_{TV} \leq 1$ for any probability measures $\mu_1$ and $\mu_2$.
\end{proof}

\begin{proof}[Proof of Lemma \ref{lemma_lnx0}] 
In view of (\ref{l_n_x0}), the stated result holds if there exist constants $\rho \in (0,1)$ and $M < \infty$ and a random sequence $\{b_k\}$ with $\mathbb{P}_{\theta^*}(b_{k} \geq M \text{ i.o.})=0$ such that, for $k=1,\ldots,n$, 
\begin{equation} \label{diff_bound}
 \sup_{x_0 \in \mathcal{X}} \sup_{\theta \in \Theta} \left| \log p_{\theta}(Y_k| \overline{\bf Y}_0^{k-1},{\bf W}_{0}^k,x_0) - \log {p}_{\theta}(Y_k| \overline{\bf Y}_0^{k-1},{\bf W}_{0}^k) \right| \leq \min\left\{\frac{b_+ (\overline{\bf Y}_{k-1}^k,W_k)} { b_-(\overline{\bf Y}_{k-1}^k,W_k)}, \rho^{\lfloor k/3p\rfloor } b_{k}\right\},
\end{equation}
because $b_+ (\overline{\bf Y}_{k-1}^k,W_k) / b_-(\overline{\bf Y}_{k-1}^k,W_k)< \infty$ $\mathbb{P}_{\theta^*}$-a.s.\ from Assumption  \ref{assn_gbound}.  

First, it follows from $p_{\theta}(Y_k| \overline{\bf Y}_{0}^{k-1},{\bf W}_{0}^k,x_{0})= \int g_{\theta}(Y_k| \overline{\bf Y}_{k-1},x_{k},W_k) \mathbb{P}_\theta (dx_k|x_{0}, \overline{\bf Y}^{k-1}_{0}, {\bf W}_{0}^{k})$, \\$p_{\theta}(Y_k| \overline{\bf Y}_{0}^{k-1},{\bf W}_{0}^k)= \int g_{\theta}(Y_k| \overline{\bf Y}_{k-1},x_{k},W_k) \mathbb{P}_\theta (dx_k|\overline{\bf Y}^{k-1}_{0}, {\bf W}_{0}^{k})$, and Assumption \ref{assn_gbound}(a) that \\ $p_{\theta}(Y_k| \overline{\bf Y}_0^{k-1},{\bf W}_{0}^k,x_0), p_{\theta}(Y_k| \overline{\bf Y}_0^{k-1},{\bf W}_{0}^k) \in [b_-(\overline{\bf Y}_{k-1}^k,W_k), b_+(\overline{\bf Y}_{k-1}^k,W_k)]$ uniformly in $\theta \in \Theta$ and $x_0 \in \mathcal{X}$. Hence, from the inequality $|\log x - \log y|\leq |x-y|/(x \wedge y)$, we have, for $k=1,\ldots,n$, 
\begin{equation} \label{pp_bound_1}
 \sup_{x_0 \in \mathcal{X}} \sup_{\theta \in \Theta}| \log p_{\theta}(Y_k| \overline{\bf Y}_0^{k-1},{\bf W}_{0}^k,x_0) - \log {p}_{\theta}(Y_k| \overline{\bf Y}_0^{k-1},{\bf W}_{0}^k) | \leq b_+(\overline{\bf Y}_{k-1}^k,W_k) / b_-(\overline{\bf Y}_{k-1}^k,W_k).
\end{equation} 
This gives the first bound in (\ref{diff_bound}).

We proceed to derive the second bound in (\ref{diff_bound}). Using a derivation similar to (\ref{x_a}) and noting that $X_k$ is independent of $W_k$ given $X_{k-1}$ gives, for any $-m +p \leq k \leq n$, 
\begin{equation} \label{yw_lag}
\mathbb{P}_{\theta}(X_{k} \in \cdot |X_{k-p},\overline{\bf Y}_{-m}^{k-1},{\bf W}_{-m}^{k}) = \mathbb{P}_{\theta}(X_{k} \in \cdot |X_{k-p},\overline{\bf Y}_{k-p}^{k-1},{\bf W}_{k-p}^{k-1}).
\end{equation}
Consequently, for any $-m +p \leq k \leq n$,
\begin{align}
& p_{\theta}(Y_k| \overline{\bf Y}_{-m}^{k-1},{\bf W}_{-m}^k,x_{-m}) \nonumber \\
& = \int\int g_{\theta}(Y_k| \overline{\bf Y}_{k-1},x_{k},W_k) p_\theta (x_k|x_{k-p}, \overline{\bf Y}^{k-1}_{k-p}, {\bf W}_{k-p}^{k-1})\mathbb{P}_{\theta}(d x_{k-p}|x_{-m},\overline{\bf Y}_{-m}^{k-1},{\bf W}_{-m}^{k-1})\mu(d x_{k}), \label{p_kx0} \\
& p_{\theta}(Y_k| \overline{\bf Y}_{-m}^{k-1},{\bf W}_{-m}^k) \nonumber \\
& = \int\int g_{\theta}(Y_k| \overline{\bf Y}_{k-1},x_{k},W_k) p_\theta (x_k|x_{k-p}, \overline{\bf Y}^{k-1}_{k-p}, {\bf W}_{k-p}^{k-1}) \mathbb{P}_{\theta}(d x_{k-p}|\overline{\bf Y}_{-m}^{k-1},{\bf W}_{-m}^{k-1})\mu(d x_{k}).\label{p_k}
\end{align}
Furthermore,
\begin{equation} \label{p_x_yw}
\mathbb{P}_{\theta}(X_{k-p} \in \cdot |\overline{\bf Y}_{-m}^{k-1},{\bf W}_{-m}^{k-1}) = \int \mathbb{P}_{\theta}(X_{k-p} \in \cdot|x_{-m},\overline{\bf Y}_{-m}^{k-1},{\bf W}_{-m}^{k-1}) \mathbb{P}_{\theta}(dx_{-m}|\overline{\bf Y}_{-m}^{k-1},{\bf W}_{-m}^{k-1}).
\end{equation}
Combining (\ref{p_kx0}), (\ref{p_k}), and (\ref{p_x_yw}) for $m=0$ and applying Lemma \ref{corollary_1} and the property of the total variation distance gives that, for any $p \leq k \leq n$ and uniformly in $x_0 \in \mathcal{X}$,
\begin{align}
& \left| p_{\theta}(Y_k| \overline{\bf Y}_0^{k-1},{\bf W}_{0}^k,x_0) - p_{\theta}(Y_k| \overline{\bf Y}_0^{k-1},{\bf W}_{0}^k) \right| \nonumber \\
& \leq \left| \int \int g_{\theta}(Y_k| \overline{\bf Y}_{k-1},x_{k}, W_k) p_\theta (x_k|x_{k-p}, \overline{\bf Y}^{k-1}_{k-p}, {\bf W}_{k-p}^{k-1}) \mu(d x_{k}) \right. \nonumber \\
& \quad \times \left. \left( \mathbb{P}_{\theta}(dx_{k-p}|x_0,\overline{\bf Y}_0^{k-1},{\bf W}_{0}^{k-1}) - \mathbb{P}_{\theta}(dx_{k-p}|\overline{\bf Y}_0^{k-1},{\bf W}_{0}^{k-1}) \right) \right| \nonumber \\
& \leq \prod_{i=1}^{\lfloor (k-p)/p \rfloor} \left( 1-\omega(\overline{\bf V}_{pi-p}^{pi-1}) \right) \sup_{x_{k-p}} \int g_{\theta}(Y_k| \overline{\bf Y}_{k-1},x_{k}, W_k) p_\theta (x_k|x_{k-p}, \overline{\bf Y}^{k-1}_{k-p}, {\bf W}_{k-p}^{k-1}) \mu(d x_{k}) \nonumber \\
& \leq \prod_{i=1}^{\lfloor (k-p)/p \rfloor} \left( 1-\omega(\overline{\bf V}_{pi-p}^{pi-1}) \right) \sup_{x_k',x_{k-p}\in \mathcal{X}} p_\theta (x_k'|x_{k-p}, \overline{\bf Y}^{k-1}_{k-p}, {\bf W}_{k-p}^{k-1}) \int g_{\theta}(Y_k| \overline{\bf Y}_{k-1},x_{k}, W_k) \mu(dx_{k}). \label{rho_power}
\end{align}
Furthermore, (\ref{p_kx0}) and (\ref{p_k}) imply that, for any $k \geq p$, $( p_{\theta}(Y_k| \overline{\bf Y}_0^{k-1},{\bf W}_{0}^k,x_0) \wedge p_{\theta}(Y_k| \overline{\bf Y}_0^{k-1},{\bf W}_{0}^k) ) \\ \geq \inf_{x_k',x_{k-p}\in \mathcal{X}} p_\theta (x_k'|x_{k-p}, \overline{\bf Y}^{k-1}_{k-p}, {\bf W}_{k-p}^{k-1}) \int g_{\theta}(Y_k| \overline{\bf Y}_{k-1},x_{k}, W_k) \mu(dx_{k})$. Therefore, it follows from $|\log x - \log y|\leq |x-y|/(x \wedge y)$, (\ref{rho_power}), and (\ref{p_p_bound}) and the subsequent argument that, for $p \leq k \leq n$,
\begin{equation} \label{pp_bound_2}
\sup_{x_0 \in \mathcal{X}}\sup_{\theta \in \Theta} \left| \log p_{\theta}(Y_k| \overline{\bf Y}_0^{k-1},{\bf W}_{0}^k,x_0) - \log {p}_{\theta}(Y_k| \overline{\bf Y}_0^{k-1},{\bf W}_{0}^k) \right| \leq \frac{\prod_{i=1}^{\lfloor (k-p)/p \rfloor} \left( 1-\omega(\overline{\bf V}_{pi-p}^{pi-1}) \right)}{ \omega(\overline{\bf V}_{k-p}^{k-1})}.
\end{equation}
We first bound $\prod_{i=1}^{\lfloor (k-p)/p \rfloor} ( 1-\omega(\overline{\bf V}_{pi-p}^{pi-1}) )$ on the right hand side of (\ref{pp_bound_2}). Fix $\epsilon \in (0, 1/8]$. Because $\omega(\overline{\bf V}_{t-p}^{t-1})>0 $ for all $\overline{\bf V}_{t-p}^{t-1} \in \mathcal{Y}^{p+s-1}\times \mathcal{W}^p$ from Assumption \ref{assn_a3} (note that $\omega(\overline{\bf V}_{t-p}^{t-1})=\sigma_-/ \sigma_+>0$ when $p=1$), there exists $\rho \in (0,1)$ such that $\mathbb{P}_{\theta^*}(1- \omega(\overline{\bf V}_{t-p}^{t-1}) \geq \rho) \leq \epsilon$. Define $I_{i}:=\mathbb{I}\{1-\omega(\overline{\bf V}_{pi-p}^{pi-1}) \geq \rho\}$; then, we have $\mathbb{E}_{\theta^*}[I_{i}] \leq \epsilon$ and $1-\omega(\overline{\bf V}_{pi-p}^{pi-1}) \leq \rho^{1-I_i}$. Consequently, with $a_{k}:= \rho^{- \sum_{i=1}^{\lfloor(k-p)/p\rfloor} I_{i} }$, 
\begin{equation} \label{omega_prod}
\prod_{i=1}^{\lfloor (k-p)/p \rfloor} \left( 1-\omega(\overline{\bf V}_{pi-p}^{pi-1}) \right) \leq \rho^{\lfloor(k-p)/p\rfloor - \sum_{i=1}^{\lfloor(k-p)/p\rfloor} I_{i}} = \rho^{\lfloor(k-p)/p\rfloor } a_{k}. 
\end{equation}
Because $\overline{\bf V}^{t-1}_{t-p}$ is stationary and ergodic, it follows from the strong law of large numbers that $(\lfloor (k-p)/p\rfloor)^{-1} \sum_{i=1}^{\lfloor(k-p)/p\rfloor} I_{i} \to \mathbb{E}_{\theta^*}[I_{i}] \leq \epsilon$ $\mathbb{P}_{\theta^*}$-a.s.\ as $k \to \infty$. Therefore, $a_k$ is bounded as
\begin{equation} \label{omega_high}
\mathbb{P}_{\theta^*}(a_{k} \geq \rho^{-2\epsilon\lfloor (k-p)/p\rfloor} \text{ i.o.})=0. 
\end{equation}
We then bound $1/\omega(\overline{\bf V}_{k-p}^{k-1})$ on the right hand side of (\ref{pp_bound_2}). First, we consider the case $p \geq 2$.  
Define $b_-^+(\overline{\bf Y}^i_{i-1},W_i):=b_+(\overline{\bf Y}^i_{i-1},W_i)/ b_-(\overline{\bf Y}^i_{i-1},W_i)$ and $C_3 := (\sigma_-/\sigma_+)C_1^{-2(p-1)}$ with $C_1$ defined in Assumption \ref{assn_gbound2}. It follows from the definition of $\omega(\cdot)$ that $\omega(\overline{\bf V}_{k-p}^{k-1}) \geq (\sigma_-/\sigma_+) \prod_{i=k-p+1}^{k-1} b^+_-(\overline{\bf Y}^i_{i-1},W_i)^{-2} = C_3 C_1^{2(p-1)}  \prod_{i=k-p+1}^{k-1} b^+_-(\overline{\bf Y}^i_{i-1},W_i)^{-2}$. In view of $\rho \in (0,1)$, there exists a finite and positive constant $C_4$ such that $\rho^{\epsilon} = e^{-2\alpha(p-1)C_4}$ with $\alpha$ defined in Assumption \ref{assn_gbound2}. Then, 
\begin{align}
& \mathbb{P}_{\theta^*} \left(\omega(\overline{\bf V}_{k-p}^{k-1}) \leq C_3 \rho^{\epsilon \lfloor(k-p)/p\rfloor } \right) \nonumber \\
& =  \mathbb{P}_{\theta^*} \left(\omega(\overline{\bf V}_{k-p}^{k-1}) \leq C_3 e^{-2\alpha(p-1)C_4 \lfloor(k-p)/p\rfloor} \right) \nonumber  \\
& \leq \mathbb{P}_{\theta^*} \left(  \prod_{i=k-p+1}^{k-1} b^+_-(\overline{\bf Y}^i_{i-1},W_i )^{-2} \leq C_1^{-2(p-1)} e^{-2 \alpha (p-1)C_4 \lfloor(k-p)/p\rfloor }\right) \nonumber \\
& = \mathbb{P}_{\theta^*} \left(  \prod_{i=k-p+1}^{k-1} b^+_-(\overline{\bf Y}^i_{i-1},W_i ) \geq C_1^{(p-1)} e^{\alpha (p-1)C_4 \lfloor(k-p)/p\rfloor }\right). \label{omega_lower}
\end{align}
Observe that, if $X_1,\ldots, X_\ell$ are identically distributed, we have $P(X_1\cdots X_\ell \geq A) \leq P(\{X_1 \geq A^{1/\ell}\} \cup \{X_2 \geq A^{1/\ell}\} \cup \cdots \cup \{X_\ell \geq A^{1/\ell}\} ) \leq \sum_{i=1}^\ell P(X_i \geq A^{1/\ell}) = \ell P(X_i \geq A^{1/\ell})$. Therefore, (\ref{omega_lower}) is bounded by $(p-1) \mathbb{P}_{\theta^*}(b^+_-(\overline{\bf Y}_{k-1}^k,W_k) \geq C_1 e^{\alpha C_4 \lfloor(k-p)/p\rfloor })$. From Assumption \ref{assn_gbound2}, this is no larger than $(p-1) C_2 (C_4 \lfloor(k-p)/p\rfloor )^{-\beta}$ for $k \geq 2p$, and  $\mathbb{P}_{\theta^*}(\omega(\overline{\bf V}_{k-p}^{k-1}) \leq C_3 \rho^{\epsilon \lfloor(k-p)/p\rfloor } \text{ i.o.})=0$ follows from the Borel-Cantelli lemma. When $p=1$, we have $\mathbb{P}_{\theta^*}(\omega(\overline{\bf V}_{k-p}^{k-1}) \leq C_3 \rho^{\epsilon \lfloor(k-p)/p\rfloor } \text{ i.o.})=0$ because $\omega(\overline{\bf V}_{k-p}^{k-1})=\sigma_-/ \sigma_+$.  Substituting this bound and (\ref{omega_prod}) and (\ref{omega_high}) into (\ref{pp_bound_2}) gives, for $p \leq k \leq n$,
\begin{equation} \label{omega_bound}
 \sup_{x_0 \in \mathcal{X}} \sup_{\theta \in \Theta} \left| \log p_{\theta}(Y_k| \overline{\bf Y}_0^{k-1},{\bf W}_{0}^k,x_0) - \log {p}_{\theta}(Y_k| \overline{\bf Y}_0^{k-1},{\bf W}_{0}^k) \right| \leq \rho^{(1-3\epsilon)\lfloor(k-p)/p\rfloor } b_{k},
\end{equation} 
where $\mathbb{P}_{\theta^*}(b_{k} \geq M \text{ i.o.})=0$ for a constant $M < \infty$. 

The right hand side of (\ref{omega_bound}) gives the second bound in (\ref{diff_bound})  because $(1-3\epsilon)\lfloor(k-p)/p\rfloor \geq \lfloor(k-p)/p\rfloor/2 \geq \lfloor(k-p)/2p\rfloor \geq \lfloor k/3p\rfloor$, where the last inequality holds because, for any numbers $a,b>0$ and $k\geq 0$, 
\begin{equation} \label{floor_ineq}
{\lfloor(k-a)_+/b\rfloor } \geq {\lfloor k/(a+b)\rfloor }.
\end{equation}
Therefore, (\ref{diff_bound}) holds, and the stated result is proven.

\end{proof}

\begin{proof}[Proof of Lemma \ref{lemma3_dmr}]
The proof uses a similar argument to the proof of Lemma 3 in DMR and the proof of Lemma \ref{lemma_lnx0}. We first show part (a) for $-m+p \leq k \leq n$. Using a similar argument to (\ref{p_kx0}) and (\ref{rho_power}) in conjunction with Lemma \ref{corollary_1} gives
\begin{align}
& p_\theta(Y_k|\overline{\bf Y}^{k-1}_{-m},{\bf W}_{-m}^k,X_{-m}=x) - p_\theta(Y_k|\overline{\bf Y}^{k-1}_{-m'},{\bf W}_{-m'}^k,X_{-m'}=x') \nonumber \\
& = \int \int \int g_{\theta}(Y_k|\overline{\bf Y}_{k-1},x_k,W_k) p_\theta (x_k|x_{k-p}, \overline{\bf Y}^{k-1}_{k-p}, {\bf W}_{k-p}^{k-1}) \mu (dx_k) \nonumber \\
& \quad \times \mathbb{P}_\theta(dx_{k-p}|X_{-m}=x_{-m},\overline{\bf Y}^{k-1}_{-m}, {\bf W}_{-m}^{k}) \left[  \delta(x_{-m}-x) dx_{-m}  - \mathbb{P}_\theta(dx_{-m}| X_{-m'}=x', \overline{\bf Y}^{k-1}_{-m'},{\bf W}_{-m'}^k) \right] \label{p_m_mprime} \\
& \leq \prod_{i=1}^{\lfloor (k-p+m)/p \rfloor} \left( 1-\omega(\overline{\bf V}_{-m+pi-p}^{-m+pi-1}) \right) \sup_{ x_k',x_{k-p}\in \mathcal{X}} p_\theta ( x_k'|x_{k-p}, \overline{\bf Y}^{k-1}_{k-p}, {\bf W}_{k-p}^{k-1}) \int g_{\theta}(Y_k| \overline{\bf Y}_{k-1},x_{k}, W_k) \mu(dx_{k}), \nonumber 
\end{align}
where  $\delta(\cdot)$ denotes the Dirac delta function, and  the first equality uses the fact $\mathbb{P}_\theta(X_{k-p} \in \cdot |X_{-m},\overline{\bf Y}^{k-1}_{-m'},{\bf W}_{-m'}^k) = \mathbb{P}_\theta(X_{k-p}\in \cdot |X_{-m},\overline{\bf Y}^{k-1}_{-m},{\bf W}_{-m}^{k})$, which is proven as (\ref{yw_lag}).

Furthermore, (\ref{p_kx0}) and (\ref{p_k}) imply that, for any $k \geq -m+ p$, $( p_{\theta}(Y_k| \overline{\bf Y}_{-m}^{k-1},{\bf W}_{-m}^k,x_{-m}) \wedge p_{\theta}(Y_k| \overline{\bf Y}_{-m'}^{k-1},{\bf W}_{-m'}^k,x_{-m'}) ) \geq \inf_{x_k',x_{k-p}\in \mathcal{X}} p_\theta (x_k'|x_{k-p}, \overline{\bf Y}^{k-1}_{k-p}, {\bf W}_{k-p}^{k-1}) \int g_{\theta}(Y_k| \overline{\bf Y}_{k-1},x_{k}, W_k) \mu(dx_{k})$. Therefore, it follows from the inequality $|\log x - \log y|\leq |x-y|/(x \wedge y)$ that 
\begin{equation}\label{omega_kk}
\begin{aligned}
& \left| \log p_\theta(Y_k|\overline{\bf Y}^{k-1}_{-m},{\bf W}_{-m}^k,X_{-m}=x) - \log p_\theta(Y_k|\overline{\bf Y}^{k-1}_{-m'},{\bf W}_{-m'}^k,X_{-m'}=x') \right| \\
& \leq \frac{\prod_{i=1}^{\lfloor (k-p+m)/p \rfloor} \left( 1-\omega(\overline{\bf V}_{-m+pi-p}^{-m+pi-1}) \right)}{ \omega(\overline{\bf V}_{k-p}^{k-1})}.
\end{aligned}
\end{equation}
Proceeding as in (\ref{omega_prod})--(\ref{omega_bound}) in the proof of Lemma \ref{lemma_lnx0}, we find that there exist $\rho \in (0,1)$ and $\epsilon \in (0,1/8]$ such that the right hand side of (\ref{omega_kk}) is bounded by $\rho^{(1-2\epsilon)\lfloor(k-p+m)/p\rfloor} \rho^{- \epsilon\lfloor(k-p)/p\rfloor} B_{k,m}$, where $\mathbb{P}_{\theta^*}(B_{k,m} \geq M \text{ i.o.})=0$ for a constant $M < \infty$. Therefore, part (a) is proven for $-m+p \leq k \leq n$ by noting that $\rho^{- \epsilon\lfloor(k-p)/p\rfloor} \leq \rho^{- \epsilon\lfloor(k-p+m)/p\rfloor}$ and using the argument following (\ref{omega_bound}). Part (a) holds for $1\leq k \leq -m+p-1$ because $| \log p_\theta(Y_k|\overline{\bf Y}^{k-1}_{-m},{\bf W}_{-m}^k,X_{-m}=x) - \log p_\theta(Y_k|\overline{\bf Y}^{k-1}_{-m'},{\bf W}_{-m'}^k,X_{-m'}=x') |$ is bounded by $b_-^+(\overline{\bf Y}_{k-1}^k,W_k)$, which is finite $\mathbb{P}_{\theta^*}$-a.s. Part (b) follows from replacing $\mathbb{P}_\theta(dx_{-m}| X_{-m'}=x', \overline{\bf Y}^{k-1}_{-m'},{\bf W}_{-m'}^k)$ in (\ref{p_m_mprime}) with $\mathbb{P}_\theta(dx_{-m}|\overline{\bf Y}^{k-1}_{-m}, {\bf W}_{-m}^k)$. Part (c) follows from $b_-(\overline{\bf Y}_{k-1}^k,W_k) \leq p_\theta(Y_k|\overline{\bf Y}^{k-1}_{-m},{\bf W}_{-m}^k,X_{-m}=x) \leq b_+(\overline{\bf Y}_{k-1}^k,W_k)$ and Assumption \ref{assn_gbound}.
\end{proof}

\begin{proof}[Proof of Proposition \ref{prop_consistency}]
The proof follows the argument of the proof of Proposition 2 and Theorem 1 in DMR. From Property 24.2 of \citet[][page 385]{gourieroux95book}, the stated result holds if (i) $\Theta$ is compact, (ii) $l_n(\theta,x_0)$ is continuous uniformly in $x_0 \in \mathcal{X}$, (iii) $\sup_{x_0 \in \mathcal{X}}\sup_{\theta \in \Theta} | n^{-1}l_n(\theta,x_0) - l(\theta)| \to 0$ $\mathbb{P}_{\theta^*}$-a.s., and (iv) $l(\theta)$ is uniquely maximized at $\theta^*$. 

(i) follows from Assumption \ref{assn_a1}(a). (ii) follows from Assumption \ref{assn_consis}(a). In view of Lemma \ref{lemma_lnx0} and the compactness of $\Theta$, (iii) holds if, for all $\theta \in \Theta$,
\begin{equation}\label{ell_cgce}
\limsup_{\delta\to 0}\limsup_{n\to\infty} \sup_{|\theta'-\theta| \leq \delta} | n^{-1}l_n(\theta') - l(\theta)|=0 \quad \mathbb{P}_{\theta^*}\text{-a.s.}
\end{equation}
Noting that $l_n(\theta)=\sum_{k=1}^n \Delta_{k,0}(\theta)$, the left hand side of (\ref{ell_cgce}) is bounded by $A+B+C$, where
\begin{align*}
A & := \limsup_{n\to\infty}\sup_{\theta' \in \Theta} \left| n^{-1}\sum_{k=1}^n( \Delta_{k,0}(\theta') - \Delta_{k,\infty}(\theta')) \right|, \\
B & := \limsup_{\delta\to 0}\limsup_{n\to\infty} \sup_{|\theta'-\theta| \leq \delta}\left| n^{-1}\sum_{k=1}^n( \Delta_{k,\infty}(\theta') - \Delta_{k,\infty}(\theta)) \right|, \\
C & := \limsup_{n\to\infty} \left| n^{-1}\sum_{k=1}^n( \Delta_{k,\infty}(\theta) - \mathbb{E}_{\theta^*}\Delta_{k,\infty}(\theta)) \right|. 
\end{align*}
Fix $x \in \mathcal{X}$. Setting $m=0$ and letting $m' \to \infty$ in Lemma \ref{lemma3_dmr}(a)(b) show that $\sup_{\theta \in \Theta}|\Delta_{k,0}(\theta) - \Delta_{k,\infty}(\theta)|\leq \sup_{\theta \in \Theta} |\Delta_{k,0}(\theta)-\Delta_{k,0,x}(\theta) | + \sup_{\theta \in \Theta} |\Delta_{k,0,x}(\theta) - \Delta_{k,\infty}(\theta)| \leq 2 A_{k,0} \rho^{\lfloor k/3p \rfloor}$ while \\ $\sup_{\theta \in \Theta}|\Delta_{k,0}(\theta)-\Delta_{k,0,x}(\theta) | + \sup_{\theta \in \Theta}|\Delta_{k,0,x}(\theta) - \Delta_{k,\infty}(\theta)| \leq 4B_k$ follows from Lemma \ref{lemma3_dmr}(c). Consequently, $A=0$ $\mathbb{P}_{\theta^*}$-a.s. $B$ is bounded by, from the ergodic theorem and Lemma \ref{lemma_4_dmr}, 
\begin{align*}
& \lim_{\delta \to 0}\limsup_{n \to \infty} n^{-1} \sum_{k=1}^n \sup_{|\theta'-\theta| \leq \delta} | \Delta_{k,\infty}(\theta') - \Delta_{k,\infty}(\theta)| \\
& = \lim_{\delta \to 0} \mathbb{E}_{\theta^*} \left[ \sup_{|\theta'-\theta| \leq \delta} | \Delta_{0,\infty}(\theta') - \Delta_{0,\infty}(\theta)| \right] = 0\quad \mathbb{P}_{\theta^*}\text{-a.s.}
\end{align*}
$C=0$ $\mathbb{P}_{\theta^*}$-a.s.\ by the ergodic theorem, and hence (iii) holds. For (iv), observe that \\ $\mathbb{E}_{\theta^*}|\log p_{\theta}(Y_1 |\overline{\bf{Y}}_{-m}^0,{\bf W}_{-m}^1)|< \infty$ from Lemma \ref{lemma3_dmr}(c). Therefore, for any $m$, $\mathbb{E}_{\theta^*}[ \log p_{\theta}(Y_1 |\overline{\bf{Y}}_{-m}^0,{\bf W}_{-m}^1)]$ is uniquely maximized at $\theta^*$ from Lemma 2.2 of \citet{neweymcfadden94hdbk} and Assumption \ref{assn_consis}(b). Then, (iv) follows because $\mathbb{E}_{\theta^*}[ \log p_{\theta}(Y_1 |\overline{\bf{Y}}_{-m}^0,{\bf W}_{-m}^1)]$ converges to $l(\theta)$ uniformly in $\theta$ as $m \to \infty$ from Lemma \ref{lemma3_dmr} and the dominated convergence theorem. Therefore, (iv) holds, and the stated result is proven. 
\end{proof}

\begin{proof}[Proof of Corollary \ref{cor_consistency_xi}]
Observe that $|n^{-1}l_n(\theta,\xi) - l(\theta)| \leq \sup_{x_0 \in \mathcal{X}}|n^{-1}l_n(\theta,x_0) - l(\theta)|$ because \\ $\inf_{x_0 \in \mathcal{X}} l_n(\theta,x_0) \leq l_n(\theta,\xi) \leq \sup_{x_0 \in \mathcal{X}} l_n(\theta,x_0)$. Furthermore, $l_n(\theta,\xi)$ is continuous in $\theta$ from the continuity of $l_n(\theta,x_0)$. Therefore, the stated result follows from the proof of Proposition \ref{prop_consistency}.
\end{proof}

\begin{proof}[Proof of Lemma \ref{lemma_psi_bound}]
The proof follows the argument of the proof of Lemma 13 in DMR. When $(k,m)=(1,0)$, the stated result follows from $\Psi_{ 1,0,x}^j(\theta) = \mathbb{E}_{\theta^*}[\phi^j_{\theta 1}|\overline{\bf V}^{}_{0},X_{0}= x]$, $\Psi_{1,0}^j(\theta) = \mathbb{E}_{\theta^*}[\phi^j_{\theta 1}|\overline{\bf V}_{0}]$, $\sup_{\theta \in G}|\phi^j_{\theta k}| \leq |\phi^j_k|_{\infty}$, and Assumption \ref{assn_distn}. Henceforth, assume $(k,m) \neq (1,0)$ so that $k+m \geq 2$.

For part (a), it follows from Lemma \ref{lemma_ijl}(a)--(e) that
\begin{align}
 \left|\Psi_{k,m,x}^j(\theta)-\Psi_{k,m}^j(\theta)\right| &\leq 4\sum_{t=-m+1}^k |\phi_t^j|_{\infty} \left( \Omega_{t-1,-m} \wedge \tilde \Omega_{t,k-1} \right) \nonumber \\
& \leq 4 \max_{-m \leq t' \leq k} |\phi_{t'}^j|_{\infty} \sum_{t=-m+1}^{k} \left( \Omega_{t-1,-m} \wedge \tilde \Omega_{t,k-1} \right), \label{psi_diff0}
\end{align}
where $\Omega_{t-1,-m}:=\prod_{i=1}^{\lfloor (t-1+m)/p \rfloor} ( 1-\omega(\overline{\bf V}_{-m+pi-p}^{-m+pi-1}))$ and $\tilde \Omega_{t,k-1}:=\prod_{i=1}^{\lfloor (k-1-t)/p \rfloor} ( 1-\omega(\overline{\bf V}_{k- 2 -pi+1}^{k- 2 -pi+p}))$ as defined in the paragraph preceding Lemma \ref{lemma_ijl}. As shown on page 2294 of DMR, we have \\$\max_{-m \leq t' \leq k} |\phi_{t'}^j|_{\infty} \leq \sum_{t=-m}^k(|t|\vee 1)^2 |\phi_t^j|_{\infty}/(|t|\vee 1)^{2} \leq 2 (k \vee m)^2 [\sum_{t=-\infty}^\infty |\phi_t^j|_{\infty} / (|t|\vee 1)^{2}] \leq (k+m)^2 K_j$ with $K_j \in L^{3-j}(\mathbb{P}_{\theta^*})$. 

We proceed to bound $\sum_{t=-m+1}^{k} ( \Omega_{t-1,-m} \wedge \tilde \Omega_{t,k-1} )$ on the right hand side of (\ref{psi_diff0}). Similar to the proof of Lemma \ref{lemma_lnx0}, fix $\epsilon \in (0,1/8 p (p+1)]$; then, there exists $\rho \in (0,1)$ such that $\mathbb{P}_{\theta^*}(1- \omega(\overline{\bf V}_{k-p}^{k-1}) \geq \rho ) \leq \epsilon$. Define $I_{p,i}:=\sum_{t=0}^{(p-2)_+}\mathbb{I}\{1-\omega(\overline{\bf V}_{t+i}^{t+i+p-1}) \geq \rho\}$ and $\nu_b^a := \sum_{i=b}^{a}I_{p,i}$. Observe that (recall we define $\prod_{i=c}^d x_i=1$ when $c>d$)
\begin{equation}\label{omega_prod_bound1}
\begin{aligned}
\prod_{i=\lfloor (b-s)/p \rfloor+1}^{\lfloor (a-s)/p \rfloor} ( 1-\omega(\overline{\bf V}_{s+pi-p}^{s+pi-1})) 
& \leq \rho^{(\lfloor (a-s)/p \rfloor - \lfloor (b-s)/p \rfloor)_+ - \sum_{i=\lfloor (b-s)/p \rfloor+1}^{\lfloor (a-s)/p \rfloor} \mathbb{I}\{1-\omega(\overline{\bf V}_{s+pi-p}^{s+pi-1}) \geq \rho\}} \\
& \leq \rho^{\lfloor (a-b)_+/p \rfloor - \nu_{b-p}^{a-p}},
\end{aligned}
\end{equation}
where the second inequality follows from $\lfloor x \rfloor - \lfloor y \rfloor \geq \lfloor x-y \rfloor$, $(\lfloor x/p \rfloor)_+=\lfloor x_+/p \rfloor$, $s+p(\lfloor (b-s)/p \rfloor+1)-p \geq b-p$, and $s+p\lfloor (a-s)/p \rfloor-1 \leq a-1$. Similarly, we obtain
\begin{equation}\label{omega_prod_bound2}
\prod_{i=\lfloor (k-1-a)/p \rfloor+1}^{\lfloor (k-1-b)/p \rfloor} ( 1-\omega(\overline{\bf V}_{k-2-pi+1}^{k-2-pi+p})) \leq \rho^{\lfloor (a-b)_+/p \rfloor - \nu_{b}^{a}},
\end{equation}
because $k-2-p\lfloor (k-1-b)/p \rfloor+1 \geq b$ and $k-2-p(\lfloor (k-1-a)/p \rfloor+1)+p \leq a+p-1$. By applying (\ref{omega_prod_bound1}) to $\Omega_{t-1,-m}$ with $a=t-1,b=s=-m$, applying (\ref{omega_prod_bound2}) to $\tilde \Omega_{t,k-1}$ with $a=k-1$ and $b=t$, and using (\ref{floor_ineq}) and $-m+1 \leq t \leq k$, we obtain
\begin{equation} \label{omega_two}
\begin{aligned}
\Omega_{t-1,-m} &\leq \rho^{\lfloor (t-1+m)_+/p \rfloor - \nu_{-m-p}^{t-1-p}} \leq \rho^{\lfloor (t+m)/(p+1) \rfloor - \nu_{-m-p}^{k}}, \\
 \tilde \Omega_{t,k-1} & \leq \rho^{\lfloor (k-1-t)_+/p \rfloor - \nu_{t}^{k-1}} \leq \rho^{\lfloor (k-t)/(p+1) \rfloor - \nu_{-m-p}^{k}}.
\end{aligned}
\end{equation} 
Observe that, for any $\rho \in (0,1)$, $c>0$ and any integers $a<b$,
\begin{equation} \label{rho_sum_bound}
\begin{aligned}
\sum_{t=-\infty}^\infty \left( \rho^{\lfloor (t+a)/c \rfloor} \wedge \rho^{\lfloor(b-t)/c\rfloor} \right) & \leq \sum_{t=-\infty}^{\lfloor(b-a)/2\rfloor} \rho^{\lfloor (b-t)/c \rfloor} + \sum_{t=\lfloor(b-a)/2\rfloor+1}^{\infty} \rho^{\lfloor(t+a)/c \rfloor} \\
& \leq \frac{c}{1-\rho} \left( \rho^{\lfloor(b-\lfloor(b-a)/2\rfloor)/c\rfloor} + \rho^{\lfloor(\lfloor(b-a)/2\rfloor+1+a)/c\rfloor} \right) \\
& \leq \frac{2c}{1-\rho} \rho^{\lfloor(a+b)/2c\rfloor}.
\end{aligned}
\end{equation}
From (\ref{rho_sum_bound}), $\sum_{t=-m+1}^{k} ( \Omega_{t-1,-m} \wedge \tilde \Omega_{t,k-1} )$ is bounded by $2(p+1) \rho^{\lfloor(k+m)/2(p+1)\rfloor-\nu_{-m-p}^k} /(1-\rho)$. Because $\overline{\bf V}_{i}^{i+p-1}$ is stationary and ergodic, it follows from the strong law of large numbers that $(\lfloor (k+m)/2(p+1)\rfloor)^{-1} \nu_{-m-p}^k \to 2(p+1) \estar [I_{p,i}] \leq 2 p (p+1)\epsilon$ $\pstar$-a.s.\ as $k+m \to \infty$. In view of $\epsilon < 1/8 p (p+1)$, we have $\pstar(\rho^{\lfloor(k+m)/2(p+1)\rfloor-\nu_{-m-p}^k} \geq \rho^{ \lfloor (k+m)/2(p+1)\rfloor/2}\text{ i.o.})=0$. Henceforth, let $\{b_{k,m}\}_{k\geq 1, m \leq 0}$ denote a generic nonnegative random sequence such that $\mathbb{P}_{\theta^*}(b_{k,m} \geq M \text{ i.o.})=0$ for a finite constant $M$. With this notation and the fact that $\lfloor(k+m)/2(p+1)\rfloor/2 \geq \lfloor(k+m)/4(p+1)\rfloor$, $\sum_{t=-m+1}^{k} ( \Omega_{t-1,-m} \wedge \tilde \Omega_{t,k-1} )$ is bounded by 
\begin{equation} \label{omega_sum_4}
\rho^{\lfloor(k+m)/4(p+1)\rfloor}b_{k,m},
\end{equation}
and part (a) is proven.

For part (b), it follows from (\ref{Psi_x_defn}) and Lemma \ref{lemma_ijl}(a)--(e) that
\begin{align*}
 \left| \Psi_{k,m,x}^j(\theta)-\Psi_{k,m',x'}^j(\theta) \right| & \leq 4 \sum_{t=-m+1}^{k} \left( \Omega_{t-1,-m} \wedge \tilde \Omega_{t,k-1} \right)|\phi_t^j|_{\infty} + 2 \sum_{t=-m'+1}^{-m} \tilde \Omega_{t,k-1} |\phi_t^j|_{\infty}.
\end{align*}
The first term on the right hand side is bounded by $(k + m)^2 K_j\rho^{\lfloor(k+m)/4(p+1)\rfloor}b_{k,m}$ with $K_j \in L^{3-j}(\mathbb{P}_{\theta^*})$ from the same argument as the proof of part (a). For the second term on the right hand side, write $\tilde \Omega_{t,k-1}$ as $\tilde \Omega_{t,k-1} = \tilde \Omega_{-m,k-1}\tilde \Omega_{t,k-1}^{-m}$, where $\tilde \Omega_{t,k-1}^{-m}:=\prod_{i=\lfloor (k-1+m)/p \rfloor+1}^{\lfloor (k-1-t)/p \rfloor} ( 1-\omega(\overline{\bf V}_{k-2-pi+1}^{k-2-pi+p}))$. By applying (\ref{omega_prod_bound2}) to $\tilde \Omega_{t,k-1}^{-m}$ with $a=-m$ and $b=t$, we obtain $\tilde \Omega_{t,k-1}^{-m} \leq \rho^{\lfloor (-m-t)/p \rfloor - \nu_{t}^{-m}}$. In conjunction with $\tilde \Omega_{t,k-1}^{-m} \leq 1$, the second term on the right hand side is bounded by $2\tilde \Omega_{-m,k-1} R_{m,m'}$, where 
\begin{equation} \label{Rmm}
R_{m,m'}:=\sum_{t=-m'+1}^{-m} d_{t,m} |\phi_t^j|_{\infty}, \quad d_{t,m}:=\min\{1,\rho^{\lfloor (-m-t)/p \rfloor - \nu_{t}^{-m}}\}.
\end{equation}
From a similar argument to (\ref{omega_two})--(\ref{omega_sum_4}), we can bound $\tilde \Omega_{-m,k-1}$ as $\tilde \Omega_{-m,k-1} \leq \rho^{\lfloor(k+m)/4(p+1)\rfloor}b_{k,m}$. It follows from $(-m-t)^{-1}\nu_{t}^{-m} \to \estar[I_{p,i}] \leq p\epsilon$ $\pstar$-a.s.\ as $t+m \to - \infty$ that $\pstar( d_{t,m} \geq \rho^{\lfloor (-m-t)/p \rfloor/2} \text{ i.o.})=0$. Furthermore, $|\phi_t^j|_{\infty}$ satisfies $\pstar(|\phi_t^j|_{\infty} \geq \rho^{-\lfloor (-m-t)/p \rfloor/4} \text{ i.o.})=0$ from Markov's inequality and the Borel-Cantelli lemma. Therefore, $\pstar(d_{t,m}|\phi_t^j|_{\infty} \geq \rho^{\lfloor (-m-t)/p \rfloor/4} \text{ i.o.})=0$. In conjunction with $0 \leq d_{t,m}|\phi_t^j|_{\infty}<\infty$ $\pstar$-a.s., we obtain $\overline R_{m}:=\sup_{m'\geq m} R_{m,m'}<\infty$ $\pstar$-a.s., and the distribution of $\overline R_{m}$ does not depend on $m$ because $\overline{\bf V}_{t}$ is stationary. Therefore, part (b) is proven by setting $B_m=\overline R_{m}$.
\end{proof}

\begin{proof}[Proof of Proposition \ref{prop_score}]
By setting $m=0$ and letting $m' \to \infty$ in Lemma \ref{lemma_psi_bound}, we obtain\\ $\sup_{\theta \in G} \sup_{x \in \mathcal{X}} |\Psi_{k,0,x}^1(\theta)-\Psi_{k,\infty}^1(\theta)| \leq (K_1+B_0) k^2 \rho^{\lfloor k/4(p+1)\rfloor} A_{k,0}$. Furthermore, the sum over finitely many $\sup_{\theta \in G} \sup_{x \in \mathcal{X}} |\Psi_{k,0,x}^1(\theta)-\Psi_{k,\infty}^1(\theta)|$ is $o(n^{1/2})$ $\pstar$-a.s.\ because \\ $\estar[\sup_{\theta \in G}\sup_{x \in \mathcal{X}} |\Psi_{k,0,x}^1(\theta)|] < \infty$ and $\estar[\sup_{\theta \in G} |\Psi_{k,\infty}^1(\theta)|] < \infty$ from Assumption \ref{assn_nabla_moment}. Therefore, we have $n^{-1/2}\nabla_\theta l_n(\theta^*,x_0) = n^{-1/2}\sum_{k=1}^n \Psi_{k,0,x_0}^1(\theta^*) = n^{-1/2}\sum_{k=1}^n \Psi_{k,\infty}^1(\theta^*) + o_p(1)$.

Because $\{\Psi_{k,\infty}^1(\theta^*)\}_{k=-\infty}^{\infty}$ is a stationary, ergodic, and square integrable martingale difference sequence, it follows from a martingale difference central limit theorem \citep[][Theorem 2.3]{mcleish74ap} that $n^{-1/2}\sum_{k=1}^n \Psi_{k,\infty}^1(\theta^*)\to_d N(0,I(\theta^*))$, and part (a) follows. For part (b), let $p_{n\theta}(x_0)$ denote $p_\theta({\bf Y}_1^n|\overline{\bf Y}_{0}, {\bf W}_{0}^n, x_0)$, and observe that
\begin{align*}
\nabla_\theta l_n(\theta,\xi) & = \frac{\nabla_\theta \int p_{n\theta}(x_0)\xi(dx_0)}{\int p_{n\theta}(x_0)\xi(dx_0)} = \frac{ \int \nabla_\theta \log p_{n\theta}(x_0) p_{n\theta}(x_0) \xi(dx_0)}{\int p_{n\theta}(x_0)\xi(dx_0)}.
\end{align*}
Therefore, $\min_{x_0}\nabla_\theta l_n(\theta^*,x_0) \leq \nabla_\theta l_n(\theta^*,\xi) \leq \max_{x_0}\nabla_\theta l_n(\theta^*,x_0)$ holds, and part (b) follows.
\end{proof}

\begin{proof}[Proof of Lemma \ref{lemma_gamma_cgce}]
The proof follows the argument of the proof of Lemma 17 in DMR and the proof of Lemma \ref{lemma_psi_bound}. Fix $\epsilon\in(0,1/32 p(p+1)]$ and choose $\rho\in(0,1)$ as in the proof of Lemma \ref{lemma_psi_bound}. When $(k,m)=(1,0)$, the stated result follows from $\sup_{\theta \in G}|\phi_{\theta k}| \leq |\phi_k|_{\infty}$. Henceforth, assume $(k,m) \neq (1,0)$ so that $k+m \geq 2$. For $a \leq b$, define $S_a^b:= \sum_{t=a}^{b} \phi_{\theta t}$. Let $\{b_{k,m}\}_{k\geq 1, m \leq 0}$ denote a generic nonnegative random sequence such that $\mathbb{P}_{\theta^*}(b_{k,m} \geq M \text{ i.o.})=0$ for a finite constant $M$. We prove part (a) first. Write $\Gamma_{k,m,x}(\theta)-\Gamma_{k,m}(\theta) = A + 2B + C$, where
\begin{align*}
A& :=\text{var}_{\theta}[ S_{-m+1}^{k-1} | \overline{\bf V}^{k}_{-m},X_{-m}= x] - \text{var}_{\theta}[ S_{-m+1}^{k-1} | \overline{\bf V}^{k-1}_{-m},X_{-m}= x] \\
& \quad - \text{var}_{\theta}[ S_{-m+1}^{k-1} | \overline{\bf V}^{k}_{-m}] + \text{var}_{\theta}[ S_{-m+1}^{k-1} | \overline{\bf V}^{k-1}_{-m}], \\
B& := \text{cov}_{\theta}[ \phi_{\theta k}, S_{-m+1}^{k-1} | \overline{\bf V}^{k}_{-m},X_{-m}= x] - \text{cov}_{\theta}[ \phi_{\theta k},S_{-m+1}^{k-1} | \overline{\bf V}^{k}_{-m}], \\
C &:= \text{var}_{\theta}[ \phi_{\theta k}| \overline{\bf V}^{k}_{-m},X_{-m}= x] - \text{var}_{\theta}[ \phi_{\theta k} | \overline{\bf V}^{k}_{-m}].
\end{align*}
From Lemma \ref{lemma_ijl}(f)--(k), $A$ is bounded as 
\begin{align*}
|A| & \leq 24 \sum_{-m+1 \leq s \leq t \leq k-1} \left( \Omega_{s-1,-m} \wedge \Omega_{t-1,s} \wedge \tilde \Omega_{t,k-1} \right) \max_{-m+1 \leq s \leq t \leq k-1} |\phi_t|_{\infty}|\phi_s|_{\infty}.
\end{align*}
From equation (46) of DMR on page 2299, we have $\max_{-m+1 \leq s \leq t \leq k-1} |\phi_t|_{\infty}|\phi_s|_{\infty} \leq \\ (m^3+k^3)\sum_{t=-\infty}^\infty |\phi_t|_{\infty}^2/(|t|\vee 1)^{2} \leq (k+m)^3 K$ for $K \in L^1(\pstar)$.

We proceed to bound $\Omega_{s-1,-m} \wedge \Omega_{t-1,s} \wedge \tilde \Omega_{t,k-1}$. By using the argument in (\ref{omega_prod_bound1})-(\ref{omega_two}), we obtain
\begin{equation*}
\Omega_{s-1,-m} \wedge \Omega_{t-1,s} \wedge \tilde \Omega_{t,k-1} \leq \left( \rho^{\lfloor(s+m)/(p+1)\rfloor } \wedge\rho^{\lfloor(t-s)/(p+1)\rfloor } \wedge \rho^{\lfloor (k-t)/(p+1) \rfloor } \right) \rho^{-\nu_{-m-p}^k}.
\end{equation*}
Furthermore, a derivation similar to DMR (page 2299) gives, for $n \geq 2$,
\begin{equation*} 
\sum_{0 \leq s \leq t \leq n} \left( \rho^{ \lfloor s/(p+1) \rfloor} \wedge \rho^{ \lfloor(t-s)/(p+1) \rfloor} \wedge \rho^{ \lfloor(n-t)/(p+1) \rfloor} \right) 
\leq 2 \sum_{s=0}^{n/2}\sum_{t = s}^{n-s} \left(\rho^{\lfloor(t-s)/(p+1) \rfloor} \wedge \rho^{\lfloor(n-t)/(p+1) \rfloor}\right).
\end{equation*}
From (\ref{rho_sum_bound}), the right hand side is bounded by, for a generic positive constant $\mathcal{C}$ that may take different values at different places,
\begin{equation}\label{st_bound}
\mathcal{C} \sum_{s=0}^{n/2} \rho^{\lfloor (n-s)/2(p+1) \rfloor} \leq \mathcal{C} \rho^{\lfloor n/4(p+1) \rfloor},
\end{equation}
where the inequality holds because $\sum_{t=a}^\infty \rho^{\lfloor t/b \rfloor} \leq b\rho^{\lfloor a/b \rfloor} / (1-\rho)$ for any integers $a \geq 0$ and $b>0$. Hence, $A$ is bounded by $K (k+m)^3 \rho^{\lfloor (k+m)/4(p+1) \rfloor} b_{k,m}$ by setting $n=k+m$ in (\ref{st_bound}) and noting that $(\lfloor (k+m)/4(p+1)\rfloor)^{-1} \nu_{-m-p}^k \to 4(p+1) \estar[I_{p,i}] \leq 4p(p+1)\epsilon < 1/2$ $\pstar$-a.s.\ as $k+m \to \infty$. For $B$, from Lemma \ref{lemma_ijl}(f)--(i), (\ref{omega_prod_bound1}), (\ref{omega_two}), $t \geq -m$, and (\ref{rho_sum_bound}), $B$ is bounded as, with $M_k:=\max_{-m+1 \leq t \leq k-1}|\phi_k|_{\infty} |\phi_t|_{\infty}$,
\begin{align*}
|B| &\leq 12 \sum_{-m+1 \leq t \leq k-1} \left( \Omega_{t-1,-m} \wedge \Omega_{k-1,t} \right) M_k \\ 
& \leq 12 \sum_{-m+1 \leq t \leq k-1} \left( \rho^{(t+m)/(p+1)} \wedge \rho^{(k-t)/(p+1)} \right) \rho^{-\nu_{-m-p}^k} M_k \\
& \leq \mathcal{C} \rho^{\lfloor (k+m)/2(p+1) \rfloor - \nu_{-m-p}^k} M_k,
\end{align*}
which is written as $K (k+m)^3 \rho^{\lfloor (k+m)/4(p+1) \rfloor} b_{k,m}$ for $K \in L^1(\pstar)$.
$C$ is bounded by $6 \Omega_{k-1,-m} |\phi_k|_{\infty}^2 $ from Lemma \ref{lemma_ijl}(h), and part (a) is proven.

We proceed to prove part (b). Write $\Gamma_{k,m',x'}(\theta) = A + 2B + 2C+D$, where
\begin{align*}
A& :=\text{var}_{\theta}[ S_{-m+1}^{k} | \overline{\bf V}^{k}_{-m'},X_{-m'}= x'] - \text{var}_{\theta}[ S_{-m+1}^{k-1} | \overline{\bf V}^{k-1}_{-m'},X_{-m'}= x'], \\
B& := \text{cov}_{\theta}[\phi_{\theta k}, S_{-m'+1}^{-m} | \overline{\bf V}^{k}_{-m'},X_{-m'}= x'], \\
C& := \text{cov}_{\theta}[S_{-m+1}^{k-1}, S_{-m'+1}^{-m} | \overline{\bf V}^{k}_{-m'},X_{-m'}= x'] - \text{cov}_{\theta}[ S_{-m+1}^{k-1}, S_{-m'+1}^{-m} | \overline{\bf V}^{k-1}_{-m'}, X_{-m'}= x'], \\
D &:= \text{var}_{\theta}[S_{-m'+1}^{-m}| \overline{\bf V}^{k}_{-m'},X_{-m'}= x'] - \text{var}_{\theta}[S_{-m'+1}^{-m} | \overline{\bf V}^{k-1}_{-m'},X_{-m'}= x'].
\end{align*}
$|\Gamma_{k,m,x}(\theta) - A|$ is bounded similarly to $|\Gamma_{k,m,x}(\theta)-\Gamma_{k,m}(\theta)|$ in part (a) by using Lemma \ref{lemma_ijl}. From Lemma \ref{lemma_ijl}(g), $B$ is bounded by $2 \sum_{t=-m'+1}^{-m} \Omega_{k-1,t}|\phi_k|_{\infty} |\phi_t|_{\infty} = B_1 \times B_2$, where 
\begin{align*}
& B_1: = 2 |\phi_k|_{\infty} \prod_{i= \lfloor (-m-t)/p\rfloor +1}^{\lfloor (k-1-t)/p\rfloor} ( 1-\omega(\overline{\bf V}_{t+pi-p}^{t+pi-1})), \quad B_2:= \sum_{t=-m'+1}^{-m} \prod_{i= 1}^{\lfloor (-m-t)/p\rfloor} ( 1-\omega(\overline{\bf V}_{t+pi-p}^{t+pi-1}))|\phi_t|_{\infty}.
\end{align*}
$B_1$ is bounded by $|\phi_k|_{\infty}\rho^{\lfloor (k+m)/2(p+1) \rfloor} b_{k,m}$ from the same argument as part (a). Because $\pstar(|\phi_k|_{\infty} \geq \rho^{-\lfloor (k+m)/2(p+1) \rfloor/2} \text{ i.o.})=0$, $B_1$ is bounded by $\rho^{\lfloor (k+m)/4(p+1) \rfloor} b_{k,m}$. For $B_2$, because $\prod_{i= 1}^{\lfloor (-m-t)/p\rfloor} ( 1-\omega(\overline{\bf V}_{t+pi-p}^{t+pi-1}))$ is bounded by $\rho^{\lfloor (-m-t)/p\rfloor - \nu_{t-p}^{-m}}$ from (\ref{omega_prod_bound1}), we can use the same argument as the one for $R_{m,m'}$ defined in (\ref{Rmm}) to show that $\overline B_{2m}:=\sup_{m'\geq m} B_2 <\infty$ $\pstar$-a.s.\ and $\overline B_{2m}$ is stationary. Therefore, $B$ is bounded by $\rho^{\lfloor (k+m)/4(p+1) \rfloor} b_{k,m}\overline B_{2m}$.

$|C|+|D|$ is bounded by, with $\Delta_{t,s}:=|\text{cov}_{\theta}[\phi_{\theta t},\phi_{\theta s} | \overline{\bf V}^{k}_{-m'},X_{-m'}= x'] - \text{cov}_{\theta}[\phi_{\theta t},\phi_{\theta s} | \overline{\bf V}^{k-1}_{-m'}, X_{-m'}= x']|$,
\begin{equation}\label{cd_bound}
\sum_{t=-m'+1}^{k-1}\sum_{s=-m'+1}^{-m}\Delta_{t,s} \leq 2\sum_{s=-m'+1}^{-m} \sum_{t=s}^{k-1} \Delta_{t,s} \leq 2\sum_{s=-m'+1}^{-m} \sum_{t=s}^{k-1} \left( \Omega_{t-1,s} \wedge \tilde \Omega_{t,k-1} \right) |\phi_t|_{\infty} |\phi_s|_{\infty}.
\end{equation}
Similar to (\ref{omega_two}), we obtain
\begin{align*} 
\Omega_{t-1,s} \wedge \tilde \Omega_{t,k-1} 
& \leq \left( \rho^{\lfloor(t-s)/(p+1)\rfloor - \nu_{s-p}^{t-1-p}} \wedge \rho^{\lfloor (k-t)/(p+1) \rfloor - \nu_t^{k-1}} \right) \leq \left( \rho^{\lfloor(t-s)/(p+1)\rfloor} \wedge \rho^{\lfloor (k-t)/(p+1) \rfloor } \right) \rho^{ - \nu_{s-p}^{k-1}}.
\end{align*}
Therefore, the right hand side of (\ref{cd_bound}) is bounded by
\begin{equation} \label{cd_bound2}
2\sum_{s=-m'+1}^{-m} \sum_{t=s}^{k-1} \left( \rho^{\lfloor (t-s)/(p+1) \rfloor} \wedge \rho^{\lfloor(k-t)/(p+1)\rfloor} \right) \rho^{-\nu_{-m}^{k-1}} \rho^{-\nu_{s-p}^{-m}} |\phi_t|_{\infty} |\phi_s|_{\infty}. 
\end{equation}
DMR (page 2300) show that the following holds for $k \geq 1$, $m \geq 0$ and $t,s \leq 0$:
\begin{align*}
\text{if } t \leq (k + s - 1)/2, &\text{ then }(|t| - 1)/2 \leq (3k + s - 3)/4 - t, \\
\text{if } (k + s - 1)/2 \leq t \leq k - 1, &\text{ then } (|t | - 1)/4 \leq t + (-k - 3s + 1)/4. 
\end{align*}
Consequently, (\ref{cd_bound2}) is bounded by
\begin{align*}
& 2 \rho^{\lfloor (k+m-2)/8(p+1) \rfloor -\nu_{-m}^{k-1}} \sum_{s=-m'+1}^{-m} \rho^{\lfloor (k-2s-m)/8(p+1) \rfloor} \rho^{-\nu_{s-p}^{-m}} |\phi_s|_{\infty} \\
& \quad \times \left(\sum_{t=s}^{(k+s)/2} \rho^{\lfloor ((3k+s-3)/4-t)/(p+1) \rfloor}|\phi_t|_{\infty} + \sum_{t=(k+s)/2}^{k-1} \rho^{\lfloor (t+(-k-3s+1)/4)/(p+1)\rfloor}|\phi_t|_{\infty} \right) \\
& \leq \mathcal{C} \rho^{\lfloor (k+m-2)/8(p+1) \rfloor - \nu_{-m}^{k-1} } \sum_{s=-m'+1}^{-m} \rho^{\lfloor (-m-s)/8(p+1) \rfloor -\nu_{s-p}^{-m} }|\phi_s|_{\infty} \sum_{t= s}^{k-1} \rho^{\lfloor (|t|-1)/4(p+1)\rfloor}|\phi_t|_{\infty} \\
& \leq \rho^{\lfloor (k+m)/16(p+1) \rfloor}b_{k,m} \times E \times F_{m,m'}, 
\end{align*}
where $E:=\sum_{t= -\infty}^{\infty} \rho^{\lfloor (|t|-1)/4(p+1)\rfloor}|\phi_t|_{\infty}$, and $F_{m,m'}:= \sum_{s=-m'+1}^{-m} \rho^{\lfloor (-m-s)/8(p+1) \rfloor -\nu_{s-p}^{-m} }|\phi_s|_{\infty}$. Because $E \in L^1(\pstar)$, $\overline F_{m}:=\sup_{m'\geq m} F_{m,m'} <\infty$ $\pstar$-a.s., and $\overline F_{m}$ is stationary, (\ref{cd_bound2}) is bounded by $\rho^{\lfloor (k+m)/16(p+1) \rfloor}b_{k,m} E \overline F_{m}$, and part (b) is proven.
\end{proof}

\begin{proof}[Proof of Proposition \ref{prop_hessian}] 

Define $\Upsilon_{k,m,x}(\theta) :=\Psi_{k,m,x}^2(\theta) +\Gamma_{k,m,x}(\theta)$ and $\Upsilon_{k,\infty}(\theta) :=\Psi_{k,\infty}^2(\theta) +\Gamma_{k,\infty}(\theta)$, so that $\nabla_\theta^2 l_n(\theta,x) = \sum_{k=1}^n \Upsilon_{k,0,x}(\theta)$. By setting $m=0$ and letting $m' \to \infty$ in Lemmas \ref{lemma_psi_bound} and \ref{lemma_gamma_cgce}, we obtain $\sup_{\theta \in G} \sup_{x \in \mathcal{X}} |\Upsilon_{k,0,x}(\theta) - \Upsilon_{k,\infty}(\theta)| \leq (K_2+B_0) k^2 \rho^{\lfloor k/4(p+1)\rfloor} A_{k,0} + K(k^{3}+D_0) \rho^{\lfloor k/16(p+1)\rfloor} C_{k,0}$. Furthermore, the sum over finitely many $\sup_{\theta \in G} \sup_{x \in \mathcal{X}} |\Upsilon_{k,0,x}(\theta) - \Upsilon_{k,\infty}(\theta)| $ is $o(n)$ $\pstar$-a.s.\ because $\estar \sup_{\theta \in G}\sup_{x \in \mathcal{X}} |\Upsilon_{k,0,x}(\theta)| < \infty$ and $\estar \sup_{\theta \in G} |\Upsilon_{k,\infty}(\theta)| < \infty$ from Assumption \ref{assn_nabla_moment}. Therefore, we have $\sup_{\theta \in G}\sup_{x \in \mathcal{X}} |n^{-1}\nabla_\theta^2 l_n(\theta,x) - n^{-1}\sum_{k=1}^n \Upsilon_{k,\infty}(\theta) | = o_p(1)$.

Consequently, it suffices to show that
\begin{equation} \label{psi_lln}
\sup_{\theta \in G} \left| n^{-1}\sum_{k=1}^n \Upsilon_{k,\infty}(\theta) - \mathbb{E}_{\theta^*}[\Upsilon_{0,\infty}(\theta)] \right| \to_p 0.
\end{equation}
Because $G$ is compact, (\ref{psi_lln}) holds if, for all $\theta \in G$,
\begin{align}
& n^{-1}\sum_{k=1}^n \Upsilon_{k,\infty}(\theta)- \mathbb{E}_{\theta^*}[\Upsilon_{0,\infty}(\theta)] \to_p 0, \label{psi_psi_2} \\
& \lim_{\delta \to 0} \lim_{n \to \infty} \sup_{|\theta'-\theta| \leq \delta} \left| n^{-1}\sum_{k=1}^n \Upsilon_{k,\infty}(\theta')- n^{-1}\sum_{k=1}^n \Upsilon_{k,\infty}(\theta) \right| =0 \quad \mathbb{P}_{\theta^*}{\text -a.s.} \label{psi_psi_1} 
\end{align}
(\ref{psi_psi_2}) holds by ergodic theorem. Note that the left hand side of (\ref{psi_psi_1}) is bounded by \\$\lim_{\delta \to 0}\lim_{n \to \infty}n^{-1}\sum_{k=1}^n \sup_{|\theta'-\theta| \leq \delta} |\Upsilon_{k,\infty}(\theta') -\Upsilon_{k,\infty}(\theta) |$, which equals \\ $\lim_{\delta \to 0}\mathbb{E}_{\theta^*} \sup_{|\theta'-\theta| \leq \delta} |\Upsilon_{0,\infty}(\theta') -\Upsilon_{0,\infty}(\theta) |$ $\mathbb{P}_{\theta^*}$-a.s.\ from ergodic theorem. Therefore, (\ref{psi_psi_1}) holds if
\begin{equation} \label{psi_psi_3}
\lim_{\delta \to 0} \estar \sup_{|\theta'-\theta| \leq \delta} \left|\Upsilon_{0,\infty}(\theta') - \Upsilon_{0,\infty}(\theta) \right| =0.
\end{equation}
Fix a point $x_0 \in \mathcal{X}$. The left hand side of (\ref{psi_psi_3}) is bounded by $2A_m+C_m$, where
\[
A_m:= \estar \sup_{\theta \in G} \left|\Upsilon_{0,m,x_0}(\theta) - \Upsilon_{0,\infty}(\theta) \right|, \quad
C_m := \lim_{\delta \to 0} \estar \sup_{|\theta'-\theta| \leq \delta} \left|\Upsilon_{0,m,x_0}(\theta') - \Upsilon_{0,m,x_0}(\theta) \right|.
\]
From Lemmas \ref{lemma_psi_bound} and \ref{lemma_gamma_cgce}, $\sup_{\theta \in G} |\Upsilon_{0,m,x_0}(\theta) - \Upsilon_{0,\infty}(\theta) | \to_p 0$ as $m \to \infty$. Furthermore, we have $\estar \sup_{m \geq 1} \sup_{\theta \in G} |\Upsilon_{0,m,x_0}(\theta) |< \infty$ and $\estar \sup_{\theta \in G} |\Upsilon_{0,\infty}(\theta)|<\infty$ from Assumption \ref{assn_nabla_moment}. Therefore, $A_{m} \to 0$ as $m \to \infty$ by the dominated convergence theorem \citep[][Exercise 2.3.7]{durrett10book}. $C_{m}=0$ from Lemma \ref{lemma_14_dmr} if $m \geq p$. Therefore, (\ref{psi_psi_3}) holds, and the stated result is proven.
\end{proof}

\begin{proof}[Proof of Proposition \ref{prop_distn}]
In view of (\ref{taylor}) and Propositions \ref{prop_consistency}, \ref{prop_score}, and \ref{prop_hessian}, part (a) holds if (i) $\mathbb{E}_{\theta^*}[\Psi_{0,\infty}^{2}(\theta) +\Gamma_{0,\infty}(\theta)]$ is continuous in $\theta \in G$ and (ii) $\mathbb{E}_{\theta^*}[\Psi_{0,\infty}^{2}(\theta^*) +\Gamma_{0,\infty}(\theta^*)] = -I(\theta^*)$. (i) follows from (\ref{psi_psi_3}). For (ii), it follows from the Louis information principle and information matrix equality that, for all $m \geq 1$, $\mathbb{E}_{\theta^*}[\Psi_{0,m}^{1}(\theta^*)(\Psi_{0,m}^{1}(\theta^*))'] = - \mathbb{E}_{\theta^*}[\Psi_{0,m}^{2}(\theta^*) +\Gamma_{0,m}(\theta^*)]$. From Lemmas \ref{lemma_psi_bound} and \ref{lemma_gamma_cgce}, Assumption \ref{assn_nabla_moment}, and the dominated convergence theorem, the left hand side converges to $\mathbb{E}_{\theta^*}[\Psi_{0,\infty}^{1}(\theta^*)(\Psi_{0,\infty}^{1}(\theta^*))'] = I(\theta^*)$, and the right hand side converges to $- \mathbb{E}_{\theta^*}[\Psi_{0,\infty}^{2}(\theta^*) +\Gamma_{0,\infty}(\theta^*)]$. Therefore, (ii) holds, and part (a) is proven.

For part (b), an elementary calculation gives, with $p_{n\theta}(x)$ denoting $p_\theta({\bf Y}_1^n|\overline{\bf Y}_{0},{\bf W}_{0}^n, x)$,
\begin{align*}
n^{-1}\nabla_\theta^2 l_n(\theta,\xi) & = \frac{ \int n^{-1} \nabla_\theta^2 \log p_{n\theta}(x_0) p_{n\theta}(x_0) \xi(dx_0)}{\int p_{n\theta}(x_0)\xi(dx_0)} + \frac{ \int (n^{-1/2} \nabla_\theta \log p_{n\theta}(x_0))^2 p_{n\theta}(x_0) \xi(dx_0)}{\int p_{n\theta}(x_0)\xi(dx_0)} \\
& - \left(\frac{ \int n^{-1/2} \nabla_\theta \log p_{n\theta}(x_0) p_{n\theta}(x_0) \xi(dx_0)}{\int p_{n\theta}(x_0)\xi(dx_0)}\right)^2.
\end{align*}
The sum of the last two terms is $o_p(1)$ because $\sup_{x \in \mathcal{X}}\sup_{\theta \in G} | n^{-1/2}\nabla_\theta \log p_{n\theta}(x) - n^{-1/2}\sum_{k=1}^n \Psi_{k,\infty}^1(\theta)|=o_p(1)$. Therefore, for any $\xi$ on $\mathcal{B}(\mathcal{X})$, we have  $\sup_{x_0 \in \mathcal{X}}\sup_{\theta \in G}| n^{-1}\nabla_\theta^2 l_n(\theta,\xi) -n^{-1}\nabla_\theta^2 l_n(\theta,x_0) |=o_p(1)$ holds, and part (b) follows.
\end{proof}
\section{Auxiliary results}

Lemma 1 of DMR derives the minorization condition \citep{rosenthal95jasa} on the conditional hidden Markov chain  when $p=1$ and the covariate $W_k$ is absent.  This lemma generalizes Lemma 1 of DMR to accommodate $p \geq 2$ and covariate $W_k$.\footnote{We replace the conditioning variable $\overline{\bf Y}_m^n$ in DMR with $\overline{\bf Y}_{-m}^n$, because the subsequent analysis uses $\overline{\bf Y}_{-m}^n$.} When $p \geq 2$, the minorization coefficient $\omega(\cdot)$ depends on $(\overline{\bf Y}_{k-p}^{k-1},{\bf W}_{k-p}^{k-1})$ because $\overline{\bf Y}_{k-p}^{k-1}$ provide information on $X_k$ in addition to the information provided by $X_{k-p}$.
\begin{lemma} \label{lemma_dmr_1}
Assume Assumptions \ref{assn_a1}--\ref{assn_a3}. Let $m, n \in \mathbb{Z}$ with $-m \leq n$. Then,  the following holds for all $\theta \in \Theta$;  (a) under $\mathbb{P}_\theta$, conditionally on $(\overline{\bf Y}_{-m}^n, {\bf W}_{-m}^n)$, $\{X_k\}_{k=-m}^n$ is an inhomogeneous Markov chain, and (b) for all $-m+p \leq k \leq n$, there exists a function $\mu_{k,\theta}(\overline{\bf y}_{k-1}^n,{\bf w}_{k}^n, A)$ such that
\begin{enumerate}
\item[(i)] For any $A \in \mathcal{B}(\mathcal{X})$, $\mu_{k,\theta}(\cdot,\cdot,A)$ is Borel measurable function defined on $\mathcal{Y}^{n-k+s+1} \times \mathcal{W}^{n-k+1}$;
\item[(ii)] For any $(\overline{\bf y}_{k-1}^n,{\bf w}_{k}^n)$, $\mu_{k,\theta}(\overline{\bf y}_{k-1}^n,{\bf w}_{k}^n, \cdot)$ is a probability measure on $\mathcal{B}(\mathcal{X})$. Furthermore, \\$\mu_{k,\theta}(\overline{\bf y}_{k-1}^n,{\bf w}_{k}^n, \cdot)$  is absolutely continuous with respect to  $\mu$ for all $(\overline{\bf y}_{k-1}^n,{\bf w}_{k}^n)$, and, for all $(\overline{\bf y}_{-m}^n,{\bf w}_{-m}^n)$,
\[
\inf_{x \in \mathcal{X}} \mathbb{P}_\theta \left(X_k \in A \middle| X_{k-p}= x, \overline{\bf y}_{-m}^n, {\bf w}_{-m}^n \right) \geq \omega(\overline{\bf y}_{k-p}^{k-1}, {\bf w}_{k-p}^{k-1}) \mu_{k,\theta}(\overline{\bf y}_{k-1}^n,{\bf w}_{k}^n, A),
\] 
with $\omega(\overline{\bf y}_{k-p}^{k-1},{\bf w}_{k-p}^{k-1})$ defined in (\ref{omega_defn}).
\end{enumerate}
\end{lemma}

\begin{proof}
The proof uses a similar argument to the proof of Lemma 1 in DMR. Because $\{Z_k\}_{k=-m}^{n}$ is a Markov chain given $\{W_k\}_{k=-m}^{n}$, we have, for $-m < k \leq n$, 
\[
\mathbb{P}_\theta(X_k \in A| {\bf X}^{k-1}_{-m}, \overline{\bf Y}^n_{-m}, {\bf W}_{-m}^n) = \mathbb{P}_\theta(X_k \in A| X_{k-1}, \overline{\bf Y}_{k-1}^n, {\bf W}_{k}^n).
\]
Therefore, $\{X_k\}_{k =-m}^n$ conditional on $(\overline{\bf Y}^n_{-m}, {\bf W}_{-m}^n)$ is an inhomogeneous Markov chain, and part (a) follows.

We proceed to prove part (b). Observe that if $-m+p \leq k \leq n$, 
\begin{equation} \label{x_a}
\mathbb{P}_\theta (X_k \in A|x_{k-p}, \overline{\bf y}^n_{-m}, {\bf w}_{-m}^n) = \mathbb{P}_\theta (X_k \in A|x_{k-p}, \overline{\bf y}_{k-p}^n, {\bf w}_{k-p}^n),
\end{equation}
because the left hand side of (\ref{x_a}) can be written as
\begin{align*}
\frac{\mathbb{P}_\theta (X_k \in A, {\bf y}^n_{k-p+1}|x_{k-p}, \overline{\bf y}^{k-p}_{-m}, {\bf w}_{-m}^n)}{\mathbb{P}_\theta ({\bf y}^n_{k-p+1}|x_{k-p}, \overline{\bf y}^{k-p}_{-m}, {\bf w}_{-m}^n)} & = \frac{\mathbb{P}_\theta (X_k \in A, {\bf y}^n_{k-p+1}|x_{k-p}, \overline{\bf y}_{k-p}, {\bf w}_{k-p}^n)}{\mathbb{P}_\theta ({\bf y}^n_{k-p+1}|x_{k-p}, \overline{\bf y}_{k-p}, {\bf w}_{k-p}^n)}.
\end{align*}
The equality (\ref{x_a}) holds even when the conditioning variable ${\bf w}_{k-p}^n$ on the right hand side is replaced with ${\bf w}_{k-p+1}^n$, but we use ${\bf w}_{k-p}^n$ for notational simplicity. Write the right hand side of (\ref{x_a}) as
\begin{align*}
& \mathbb{P}_\theta (X_k \in A|x_{k-p}, \overline{\bf y}_{k-p}^n, {\bf w}_{k-p}^n) \\
& = \frac{\int_A {p}_\theta (X_k= x, {\bf y}_k^n|x_{k-p}, \overline{\bf y}^{k-1}_{k-p}, {\bf w}_{k-p}^n) \mu(dx) }{p_\theta ({\bf y}_k^n|x_{k-p}, \overline{\bf y}^{k-1}_{k-p}, {\bf w}_{k-p}^n)} \\
& = \int_A p_\theta(X_{k}= x|x_{k-p},\overline{\bf y}^{k-1}_{k-p}, {\bf w}_{k-p}^{k-1}) p_\theta ({\bf y}_k^n|X_{k}= x,\overline{\bf y}_{k-1}, {\bf w}_{k}^n) \mu(dx) \\
& \times \left(\int_{\mathcal{X}} p_\theta(X_{k}= x|x_{k-p},\overline{\bf y}^{k-1}_{k-p}, {\bf w}_{k-p}^{k-1}) p_\theta ({\bf y}_k^n|X_{k}= x,\overline{\bf y}_{k-1}, {\bf w}_{k}^n) \mu(dx) \right)^{-1}.
\end{align*} 

When $p=1$, we have $p_\theta(x_{k}|x_{k-p},\overline{\bf y}^{k-1}_{k-p}, {\bf w}_{k-p}^{k-1}) = q_\theta(x_{k-1}, x_{k}) \in [\sigma_-,\sigma_+]$. Therefore, the stated result follows with $\mu_{k,\theta}(\overline{\bf y}_{k-1}^n,{\bf w}_k^n, A)$ defined as
\begin{equation}\label{mu_k}
\mu_{k,\theta}(\overline{\bf y}_{k-1}^n,{\bf w}_k^n, A) := \int_A p_\theta({\bf y}_k^n| X_{k}=x, \overline{\bf y}_{k-1}, {\bf w}_k^n) \mu(dx) \Big/\int_{\mathcal{X}} p_\theta({\bf y}_k^n| X_{k}=x, \overline{\bf y}_{k-1}, {\bf w}_k^n) \mu(dx).
\end{equation}
Note that $\int_{\mathcal{X}} p_\theta({\bf y}_k^n| X_k=x, \overline{\bf y}_{k-1},{\bf w}_k^n) \mu(dx)>0$ from Assumption \ref{assn_a3}. 

When $p \geq 2$, a lower bound on $p_\theta(x_{k}|x_{k-p},\overline{\bf y}^{k-1}_{k-p}, {\bf w}_{k-p}^{k-1})$ is obtained as
\begin{align}
& p_\theta(x_k| x_{k-p}, \overline{\bf y}_{k-p}^{k-1}, {\bf w}_{k-p}^{k-1}) \nonumber \\
& = \frac{p_\theta(x_{k},{\bf y}_{k-p+1}^{k-1}| x_{k-p}, \overline{\bf y}_{k-p}, {\bf w}_{k-p}^{k-1})}{ p_\theta({\bf y}_{k-p+1}^{k-1}| x_{k-p}, \overline{\bf y}_{k-p}, {\bf w}_{k-p}^{k-1})} \nonumber \\
& = \frac{\int \prod_{i=k-p+1}^{k} q_\theta(x_{i-1},x_i) \prod_{i=k-p+1}^{k-1} g_\theta(y_i|\overline{\bf y}_{i-1},x_i,w_i) \mu^{\otimes (p-1)} (d{\bf x}^{k-1}_{k-p+1}) }{\int \prod_{i=k-p+1}^{k} q_\theta(x_{i-1},x_i) \prod_{i=k-p+1}^{k-1} g_\theta(y_i|\overline{\bf y}_{i-1},x_i,w_i) \mu^{\otimes p} (d{\bf x}^{k}_{k-p+1}) } \nonumber \\
& \geq \frac{ \inf_\theta\inf_{x_{k-p},x_{k}} q_{\theta}^p(x_{k-p},x_k) \inf_\theta\inf_{{\bf x}^{k-1}_{k-p+1}} \prod_{i=k-p+1}^{k-1} g_\theta(y_i|\overline{\bf y}_{i-1},x_i,w_i) }{  \sup_\theta\sup_{{\bf x}^{k-1}_{k-p+1}} \prod_{i=k-p+1}^{k-1} g_\theta(y_i|\overline{\bf y}_{i-1},x_i,w_i) } \label{p_p_bound}.
\end{align} 
Similarly, an upper bound on $p_\theta(x_k| x_{k-p}, \overline{\bf y}_{k-p}^{k-1}, {\bf w}_{k-p}^{k-1})$ is given by  
\[
\frac{ \sup_\theta\sup_{x_{k-p},x_{k}} q_{\theta}^p(x_{k-p},x_k) \sup_\theta\sup_{{\bf x}^{k-1}_{k-p+1}} \prod_{i=k-p+1}^{k-1} g_\theta(y_i|\overline{\bf y}_{i-1},x_i,w_i) }{ \inf_\theta\inf_{{\bf x}^{k-1}_{k-p+1}} \prod_{i=k-p+1}^{k-1} g_\theta(y_i|\overline{\bf y}_{i-1},x_i,w_i) }.
\]
Therefore, the stated result holds with $\mu_{k,\theta}(\overline{\bf y}_{k-1}^n,{\bf w}_k^n, A)$ defined in (\ref{mu_k}).
\end{proof}

The following lemma provides the convergence rate of a Markov chain $X_t$. When $X_t$ is time-homogeneous, this result has been proven by Theorem 1 of \citet{rosenthal95jasa}. This lemma extends \citet{rosenthal95jasa} to time-inhomogeneous $X_t$.

\begin{lemma} \label{lemma_coupling}
Let $\{X_t\}_{t \geq 1}$ be a Markov process that lies in $\mathcal{X}$, and let $P_t(x,A) :=\mathbb{P}( X_{t} \in A| X_{t-1} = x)$. Suppose there is a probability measure $Q_t(\cdot)$ on $\mathcal{X}$, a positive integer $p$, and $\varepsilon_t \geq 0$ such that
\[
P_t^{p}(x,A):= \mathbb{P}(X_{t} \in A|X_{t-p}=x) \geq \varepsilon_t Q_t(A),
\]
for all $x \in \mathcal{X}$ and all measurable subsets $A \subset \mathcal{X}$. Let $X_0$ and $Y_0$ be chosen from the initial distributions $\pi_1$ and $\pi_2$, respectively, and update them according to $P_t(x,A)$. Then,
\[
\|\mathbb{P}(X_k \in \cdot) - \mathbb{P}(Y_k\in \cdot)\|_{TV} \leq \prod_{i=1}^{\lfloor k/p \rfloor} (1-\varepsilon_{ip}).
\]
\end{lemma}
\begin{proof}
The proof follows the line of argument in the proof of Theorem 1 of \citet{rosenthal95jasa}. Starting from $(X_0,Y_0)$, we let $X_t$ and $Y_t$ for $t \geq 1$ progress as follows. Given the value of $X_{t}$ and $Y_{t}$, flip a coin with the probability of heads equal to $\varepsilon_{t+p}$. If the coin comes up heads, then choose a point $x \in \mathcal{X}$ according to $Q_{t+p}(\cdot)$ and set $X_{t+p}=Y_{t+p}=x$, choose $(X_{t+1},\ldots,X_{t+p-1})$ and $(Y_{t+1},\ldots,Y_{t+p-1})$ independently according to the transition kernel $P_{t+1}(x_{t+1}|x_{t}),\ldots,P_{t+p-1}(x_{t+p-1}|x_{t+p-2})$ conditional on $X_{t+p}=x$ and $Y_{t+p}=x$, and update the processes after $t+p$ so that they remain equal for all future time. If the coin comes up tails, then choose $X_{t+p}$ and $Y_{t+p}$ independently according to the distributions $(P_{t+p}^{p}(X_{t},\cdot)-\varepsilon_{t+p} Q_{t+p}(\cdot))/(1-\varepsilon_{t+p})$ and $(P_{t+p}^{p}(Y_{t},\cdot)-\varepsilon_{t+p} Q_{t+p}(\cdot))/(1-\varepsilon_{t+p})$, respectively, and choose $(X_{t+1},\ldots,X_{t+p-1})$ and $(Y_{t+1},\ldots,Y_{t+p-1})$ independently according to the transition kernel $P_{t+1}(x_{t+1}|x_{t}),\ldots,P_{t+p-1}(x_{t+p-1}|x_{t+p-2})$ conditional on the value of $X_{t+p}$ and $Y_{t+p}$. It is easily checked that $X_t$ and $Y_t$ are each marginally updated according to the transition kernel $P_{t}(x,A)$.

Furthermore, $X_{t}$ and $Y_{t}$ are coupled the first time (call it $T$) when we choose $X_{t+p}$ and $Y_{t+p}$ both from $Q_{t+p}(\cdot)$ as earlier. It now follows from the coupling inequality that
\[
\|\mathbb{P}(X_k \in \cdot) - \mathbb{P}(Y_k \in \cdot)\|_{TV} \leq \mathbb{P}(X_k \neq Y_k) \leq \mathbb{P}(T >k).
\]
By construction, when $t$ is a multiple of $p$, $X_t$ and $Y_t$ will couple with probability $\varepsilon_{t}$. Hence,
\[
\mathbb{P}(T > k) \leq (1-\varepsilon_{p}) \cdots (1-\varepsilon_{ \lfloor k/p \rfloor \cdot p}),
\]
and the stated result follows.
\end{proof}

The following lemma corresponds to Lemma 4 of DMR and implies that $\mathbb{E}_{\theta^*}\left[\Delta_{0,\infty}(\theta) \right]$ is continuous in $\theta$. This lemma is used in the proof of the consistency of the MLE. 
\begin{lemma}\label{lemma_4_dmr}
Assume Assumptions \ref{assn_a1}--\ref{assn_consis}. Then, for all $\theta \in \Theta$, 
\[
\lim_{\delta \to 0} \mathbb{E}_{\theta^*}\left[ \sup_{|\theta -\theta'| \leq \delta}|\Delta_{0,\infty}(\theta')-\Delta_{0,\infty}(\theta)| \right] =0.
\] 
\end{lemma}
\begin{proof}
The proof is similar to the proof of Lemma 4 in DMR but requires a small adjustment when $p \geq 2$. We first show that $\Delta_{0,m,x}(\theta)$ is continuous in $\theta$ for any fixed $x \in \mathcal{X}$ and any $m \geq p+1$. Recall that $\Delta_{0,m,x}(\theta) = \log p_\theta(Y_0|\overline{\bf Y}^{-1}_{-m}, {\bf W}^0_{-m},X_{-m}= x)$ and
\[
p_\theta(Y_0|\overline{\bf Y}^{-1}_{-m}, {\bf W}^0_{-m},X_{-m}= x) = \frac{p_\theta({\bf Y}^0_{-m+1}|\overline{\bf Y}_{-m}, {\bf W}^0_{-m},X_{-m}= x)}{p_\theta({\bf Y}^{-1}_{-m+1}|\overline{\bf Y}_{-m}, {\bf W}^{-1}_{-m},X_{-m}= x)}.
\]
For $j \in \{-1,0\}$, we have
\begin{equation} \label{p_y_jm}
\begin{aligned}
& p_\theta({\bf Y}^{j}_{-m+1}|\overline{\bf Y}_{-m}, {\bf W}^{j}_{-m},X_{-m}= x) \\
& = \int q_{\theta}(x,x_{-m+1}) \prod_{i=-m+2}^j q_\theta(x_{i-1},x_i) \prod_{i=-m+1}^j g_\theta(Y_i|\overline{\bf Y}_{i-1},x_{i},W_i)\mu^{\otimes(m+j)} (d{\bf x}^j_{-m+1}).
\end{aligned}
\end{equation}
Because the integrand is bounded by $(\sigma_+^0  )^{m+j} \prod_{i=-m+1}^j b_+(\overline{\bf Y}_{i-1}^i,W_i)$, $p_\theta({\bf Y}^{j}_{-m+1}|\overline{\bf Y}_{-m},X_{-m}= x, {\bf W}^{j}_{-m})$ is continuous in $\theta$ $\mathbb{P}_{\theta^*}$-a.s.  by the continuity of $q_\theta$ and $g_\theta$ and the bounded convergence theorem. Furthermore, when $j \geq -m+p$, the infimum of the right hand side of (\ref{p_y_jm}) in $\theta$ is strictly positive $\mathbb{P}_{\theta^*}$-a.s.\ from Assumptions \ref{assn_a1}(d) and \ref{assn_a3}. Therefore, $\Delta_{0,m,x}(\theta)$ is continuous in $\theta$ $\mathbb{P}_{\theta^*}$-a.s. Because $\{\Delta_{0,m,x}(\theta)\}$ is continuous in $\theta$ and converges uniformly in $\theta \in \Theta$ $\mathbb{P}_{\theta^*}$-a.s., $\Delta_{0,\infty}(\theta)$ is continuous in $\theta \in \Theta$ $\mathbb{P}_{\theta^*}$-a.s. The stated result then follows from $\estar \sup_{\theta \in \Theta}|\Delta_{0,\infty}(\theta)|<\infty$ by Lemma \ref{lemma3_dmr}(c) and the dominated convergence theorem.
\end{proof}

This lemma corresponds to Lemma 9 of DMR and derives the minorization constant for the time-reversed process $\{X_{n-k}\}_{0\leq k \leq n+m}$ conditional on $(\overline{\bf Y}^{n}_{-m},{\bf W}^{n}_{-m})$. 
\begin{lemma} \label{lemma_dmr_9}
Assume Assumptions \ref{assn_a1} and \ref{assn_a2}. Let $m, n \in \mathbb{Z}$ with $-m \leq n$. Then,  the following holds for all $\theta \in \Theta$;  (a) under $\mathbb{P}_\theta$, conditionally on $(\overline{\bf Y}_{-m}^n, {\bf W}_{-m}^n)$, the time-reversed process $\{X_{n-k}\}_{0 \leq k \leq n+m}$ is an inhomogeneous Markov chain, and (b) for all $p \leq k \leq n+m$, there exists a function $\tilde \mu_{k,\theta}(\overline{\bf y}_{-m}^{n-k+p-1},{\bf w}_{-m}^{n-k+p-1}, A)$ such that
\begin{enumerate}
\item[(i)] For any $A \in \mathcal{B}(\mathcal{X})$,  $\tilde \mu_{k,\theta}(\cdot,\cdot,A)$ is Borel measurable function defined on $\mathcal{Y}^{n-k+p+m+s-1} \times \mathcal{W}^{n-k+p+m}$;
\item[(ii)] For any $(\overline{\bf y}_{-m}^{n-k+p-1},{\bf w}_{-m}^{n-k+p-1}, A)$, $\tilde\mu_{k,\theta}(\overline{\bf y}_{-m}^{n-k+p-1},{\bf w}_{-m}^{n-k+p-1}, \cdot)$ is a probability measure on $\mathcal{B}(\mathcal{X})$. Furthermore, $\tilde\mu_{k,\theta}(\overline{\bf y}_{-m}^{n-k+p-1},{\bf w}_{-m}^{n-k+p-1}, \cdot)$  is absolutely continuous with respect to $\mu$ for all $(\overline{\bf y}_{-m}^{n-k+p-1},{\bf w}_{-m}^{n-k+p-1})$, and, for all $(\overline{\bf y}_{-m}^{n-k+p-1},{\bf w}_{-m}^{n-k+p-1})$,
\begin{align*}
& \mathbb{P}_\theta \left(X_{n-k} \in A \middle| X_{n-k+p}, \overline{\bf y}_{-m}^{n}, {\bf w}_{-m}^{n} \right) \\
& = \mathbb{P}_\theta \left(X_{n-k} \in A \middle| X_{n-k+p}, \overline{\bf y}_{-m}^{n-k+p-1}, {\bf w}_{-m}^{n-k+p-1} \right)\\
& \geq \omega(\overline{\bf y}_{n-k}^{n-k+p-1}, {\bf w}_{n-k}^{n-k+p-1}) \tilde \mu_{k,\theta}(\overline{\bf y}_{-m}^{n-k+p-1},{\bf w}_{-m}^{n-k+p-1},A),
\end{align*}
where $\omega(\overline{\bf y}_{n-k}^{n-k+p-1}, {\bf w}_{n-k}^{n-k+p-1}) := \sigma_-/\sigma_+$ when $p=1$, and, when $p \geq 2$, $\omega(\overline{\bf y}_{n-k}^{n-k+p-1}, {\bf w}_{n-k}^{n-k+p-1})$ is defined as in (\ref{omega_defn}) but replacing $k-1$ and $k-p$ in (\ref{omega_defn}) with $n-k+p-1$ and $n-k$.
\end{enumerate}
\end{lemma}

\begin{proof}
The proof is similar to the proof of Lemma \ref{lemma_dmr_1}. Because the time-reversed process $\{Z_{n-k}\}_{0 \leq k\leq n+m}$ is Markov conditional on ${\bf W}_{-m}^{n}$  and ${\bf Z}_{-m}^{n-k+1}$ is independent of ${\bf W}_{n-k+2}^n$ given ${\bf W}_{-m}^{n-k+1}$, we have, for $1 \leq k \leq n+m$, 
\[
\mathbb{P}_\theta(X_{n-k} \in A| {\bf X}^{n}_{n-k+1}, \overline{\bf Y}^n_{-m}, {\bf W}_{-m}^n) = \mathbb{P}_\theta(X_{n-k} \in A| X_{n-k+1}, \overline{\bf Y}_{-m}^{n-k+1}, {\bf W}_{-m}^{n-k+1}).
\]
Therefore, $\{X_{n-k}\}_{0 \leq k \leq n+m}$ is an inhomogeneous Markov chain given $(\overline{\bf Y}_{-m}^n, {\bf W}_{-m}^n)$, and part (a) follows.

For part (b), because (i) the time-reversed process $\{Z_{n-k}\}_{0 \leq k\leq n+m}$ is Markov conditional on ${\bf W}_{-m}^{n}$, (ii) $Y_{n-k+p}$ is independent of ${\bf X}_{-m}^{n-k+p-1}$ given $(X_{n-k+p}, \overline{\bf Y}_{-m}^{n-k+p-1}, {\bf W}_{-m}^{n})$, (iii) $X_{n-k+p}$ is independent of the other random variables given $X_{n-k+p-1}$, and (iv) $W_{n-k+p}$ is independent of ${\bf Z}_{-m}^{n-k+p-1}$ given ${\bf W}_{-m}^{n-k+p-1}$, we have, for $1 \leq k \leq n+m$, 
\begin{equation} \label{x_b}
\mathbb{P}_\theta \left(X_{n-k} \in A\middle| X_{n-k+p}, \overline{\bf Y}_{-m}^{n}, {\bf W}_{-m}^{n} \right) = \mathbb{P}_\theta \left(X_{n-k} \in A\middle| X_{n-k+p}, \overline{\bf Y}_{-m}^{n-k+p-1}, {\bf W}_{-m}^{n-k+p-1} \right).
\end{equation}
Observe that in view of $n-k \geq -m$,
\begin{align*}
& \mathbb{P}_\theta \left(X_{n-k}\in A, X_{n-k+p}, \overline{\bf y}_{-m}^{n-k+p-1}, {\bf w}_{-m}^{n-k+p-1} \right) \\
& = \mathbb{P}_\theta \left(X_{n-k+p} \middle| X_{n-k}\in A, \overline{\bf y}_{n-k}^{n-k+p-1}, {\bf w}_{n-k}^{n-k+p-1} \right) \\
& \quad \times \mathbb{P}_\theta \left(X_{n-k}\in A, {\bf y}_{-m+1}^{n-k+p-1}, {\bf w}_{-m+1}^{n-k+p-1} \middle|\overline{\bf y}_{-m}, w_{-m} \right) \mathbb{P}_\theta \left(\overline{\bf y}_{-m}, w_{-m} \right).
\end{align*}
It follows that
\begin{align*}
& \mathbb{P}_\theta \left(X_{n-k} \in A \middle| X_{n-k+p}, \overline{\bf y}_{-m}^{n-k+p-1}, {\bf w}_{-m}^{n-k+p-1} \right) = \frac{\int_A G_\theta(x, X_{n-k+p}, \overline{\bf y}_{-m}^{n-k+p-1}, {\bf w}_{-m}^{n-k+p-1} ) \mu(dx) }{\int_{\mathcal{X}} G_\theta(x, X_{n-k+p}, \overline{\bf y}_{-m}^{n-k+p-1}, {\bf w}_{-m}^{n-k+p-1} ) \mu(dx)},
\end{align*}
where $G_\theta(x, X_{n-k+p}, \overline{\bf y}_{-m}^{n-k+p-1}, {\bf w}_{-m}^{n-k+p-1} ) := p_\theta(X_{n-k+p}|X_{n-k}=x, \overline{\bf y}_{n-k}^{n-k+p-1}, {\bf w}_{n-k}^{n-k+p-1}) \times \\ p_{\theta}(X_{n-k}=x, {\bf y}_{-m+1}^{n-k+p-1}, {\bf w}_{-m+1}^{n-k+p-1} |\overline{\bf y}_{-m}, w_{-m})$.

When $p=1$, we have $p_\theta(X_{n-k+p}|X_{n-k}=x, \overline{\bf y}_{n-k}^{n-k+p-1}, {\bf w}_{n-k}^{n-k+p-1}) =p_\theta(X_{n-k+1}|X_{n-k}= x) \in [\sigma_-,\sigma_+]$. Therefore, the stated result follows with $\tilde \mu_{k, \theta}(\overline{\bf y}_{-m}^{n-k+p-1},{\bf w}_{-m}^{n-k+p-1},A)$ defined as
\begin{equation} \label{tilde_mu_k}
\tilde \mu_{k, \theta} (\overline{\bf y}_{-m}^{n-k+p-1},{\bf w}_{-m}^{n-k+p-1},A) := \frac{\int_A p_{\theta}(X_{n-k}=x, {\bf y}_{-m+1}^{n-k+p-1}, {\bf w}_{-m+1}^{n-k+p-1} |\overline{\bf y}_{-m}, w_{-m}) \mu(dx)}{ \int_{\mathcal{X}} p_{\theta}(X_{n-k}=x, {\bf y}_{-m+1}^{n-k+p-1}, {\bf w}_{-m+1}^{n-k+p-1} |\overline{\bf y}_{-m}, w_{-m}) \mu(dx) }.
\end{equation}
Note that $\int_{\mathcal{X}} p_{\theta}(X_{n-k}=x, {\bf y}_{-m+1}^{n-k+p-1}, {\bf w}_{-m+1}^{n-k+p-1} |\overline{\bf y}_{-m}, w_{-m}) \mu(dx)>0$ from Assumption \ref{assn_a3}. 

When $p \geq 2$, it follows from a derivation similar to (\ref{p_p_bound}) that $p_\theta(x_{n-k+p}|x_{n-k}, \overline{\bf y}_{n-k}^{n-k+p-1}, {\bf w}_{n-k}^{n-k+p-1})$ is bounded from below by
\begin{align*}
\inf_\theta\inf_{x_{n-k},x_{n-k+p}} q_{\theta}^p(x_{n-k},x_{n-k+p}) H 
\end{align*}
where $H:=\inf_\theta\inf_{{\bf x}^{n-k+p-1}_{n-k+1}} \prod_{i=n-k+1}^{n-k+p-1} g_\theta(y_i|\overline{\bf y}_{i-1},x_i,w_i)/\sup_\theta\sup_{{\bf x}^{n-k+p-1}_{n-k+1}} \prod_{i=n-k+1}^{n-k+p-1} g_\theta(y_i|\overline{\bf y}_{i-1},x_i,w_i)$, and an upper bound on $p_\theta(x_{n-k+p}|x_{n-k}, \overline{\bf y}_{n-k}^{n-k+p-1}, {\bf w}_{n-k}^{n-k+p-1})$ is given by \\$\sup_\theta\sup_{x_{n-k},x_{n-k+p}} q_{\theta}^p(x_{n-k}, x_{n-k+p})/H$.   Therefore, the stated result holds with $\tilde \mu_k$ defined in (\ref{tilde_mu_k}).
\end{proof} 

 This lemma bounds the distance between the distributions of $X_k$ given $(\overline{\bf Y}_{-m}^n, {\bf W}_{-m}^n)$ and $(\overline{\bf Y}_{-m}^{n-1}, {\bf W}_{-m}^{n-1})$. This  lemma shows that the time-reversed process $\{X_{n-k}\}_{0 \leq k\leq n+m}$ conditional on $(\overline{\bf Y}^{n}_{-m},{\bf W}^{n}_{-m})$ forgets its initial conditioning variable (i.e., $Y_n$ and $W_{n}$) exponentially fast. Part (b) corresponds to equation (39) on page 2294 of DMR.
\begin{lemma}\label{lemma_dmr_39}
Assume Assumptions \ref{assn_a1} and \ref{assn_a2}. Let $m, n \in \mathbb{Z}$ with $m,n \geq 0$ and $\theta \in \Theta$. Then, 
\\(a) for all $-m \leq k \leq n$ and all $(\overline{\bf y}_{-m}^n, {\bf w}_{-m}^n)$,
\begin{align*}
& \left\| \mathbb{P}_\theta \left(X_k \in \cdot \middle|\overline{\bf y}_{-m}^n, {\bf w}_{-m}^n \right) - \mathbb{P}_\theta \left(X_k \in \cdot \middle| \overline{\bf y}_{-m}^{n-1}, {\bf w}_{-m}^{n-1} \right) \right\|_{TV} \\
& \leq \prod_{i=1}^{\lfloor (n-1-k)/p \rfloor} \left( 1-\omega(\overline{\bf y}_{n-2-pi+1}^{n-2-pi+p}, {\bf w}_{n-2-pi+1}^{n-2-pi+p}) \right).
\end{align*}
(b)
for all $-m+1 \leq k \leq n$ and all $(\overline{\bf y}_{-m}^n, {\bf w}_{-m}^n,x)$,
\begin{align*}
& \left\| \mathbb{P}_\theta \left(X_k \in \cdot \middle| \overline{\bf y}_{-m}^n, {\bf w}_{-m}^n, X_{-m}=x \right) - \mathbb{P}_\theta \left(X_k \in \cdot \middle| \overline{\bf y}_{-m}^{n-1}, {\bf w}_{-m}^{n-1}, X_{-m}=x \right) \right\|_{TV} \\
& \leq \prod_{i=1}^{\lfloor (n-1-k)/p \rfloor} \left( 1-\omega(\overline{\bf y}_{n-2-pi+1}^{n-2-pi+p}, {\bf w}_{n-2-pi+1}^{n-2-pi+p}) \right).
\end{align*}
\end{lemma}
\begin{proof}
When $k \geq n-1$, the stated result holds trivially because $\prod_{i=1}^{j}a_i=1$ when $j <i$. We first show part (a) for $k \leq n- 2$. Because the time-reversed process $\{Z_{n-k}\}_{0 \leq k\leq n+m}$ is Markov conditional on ${\bf W}_{-m}^{n}$ and $W_n$ is independent of $Z_{n-1}$ given $W_{n-1}$, we have $\mathbb{P}_\theta (X_k \in \cdot | \overline{\bf y}_{-m}^n, {\bf w}_{-m}^n ) = \int \mathbb{P}_\theta (X_k \in \cdot | x_{n-1}, \overline{\bf y}_{-m}^{n-1}, {\bf w}_{-m}^{n-1} ) \mathbb{P}_\theta (dx_{n-1} | \overline{\bf y}_{-m}^n, {\bf w}_{-m}^n )$.
Similarly, we obtain $\mathbb{P}_\theta (X_k \in \cdot | \overline{\bf y}_{-m}^{n-1}, {\bf w}_{-m}^{n-1} ) = \int \mathbb{P}_\theta (X_k \in \cdot | x_{n-1}, \overline{\bf y}_{-m}^{n-1}, {\bf w}_{-m}^{n-1} ) \mathbb{P}_\theta (dx_{n-1} | \overline{\bf y}_{-m}^{n-1}, {\bf w}_{-m}^{n-1} )$. It follows that
\begin{align*}
&\left| \mathbb{P}_\theta \left(X_k \in \cdot \middle| \overline{\bf y}_{-m}^n, {\bf w}_{-m}^n \right) - \mathbb{P}_\theta \left(X_k \in \cdot \middle| \overline{\bf y}_{-m}^{n-1}, {\bf w}_{-m}^{n-1} \right) \right| \\
& \leq \int \mathbb{P}_\theta \left(X_k \in \cdot \middle| x_{n-1}, \overline{\bf y}_{-m}^{n-1}, {\bf w}_{-m}^{n-1} \right) \left| \mathbb{P}_\theta \left(dx_{n-1} \middle| \overline{\bf y}_{-m}^n, {\bf w}_{-m}^n \right) - \mathbb{P}_\theta \left(dx_{n-1} \middle| \overline{\bf y}_{-m}^{n-1}, {\bf w}_{-m}^{n-1} \right) \right|.
\end{align*}
Therefore, the stated result follows from applying Lemmas \ref{lemma_dmr_9} and \ref{lemma_coupling} to the time-reversed process $\{X_{n-i}\}_{i=1}^{n-k}$ conditional on $(\overline{\bf Y}_{-m}^{n-1}, {\bf W}_{-m}^{n-1})$. 

For part (b) for $k \leq n- 2$, by using a similar argument to the proof of Lemma \ref{lemma_dmr_9}, we can show that (i) conditionally on $(\overline{\bf Y}_{-m}^n, {\bf W}_{-m}^n,X_{-m})$, the time-reversed process $\{X_{n-k}\}_{0 \leq k \leq n+m-1}$ is an inhomogeneous Markov chain, and (ii) for all $p \leq k \leq n+m-1$, there exists a probability measure $\breve \mu_{k,\theta}(\overline{\bf y}_{-m}^{n-k+p-1},{\bf w}_{-m}^{n-k+p-1}, x,A)$ such that, for all $(\overline{\bf y}_{-m}^{n-k+p-1},{\bf w}_{-m}^{n-k+p-1}, x)$,
\begin{align*}
& \mathbb{P}_\theta \left(X_{n-k} \in A \middle| X_{n-k+p}, \overline{\bf y}_{-m}^{n}, {\bf w}_{-m}^{n}, X_{-m}=x \right) \\
& = \mathbb{P}_\theta \left(X_{n-k} \in A \middle| X_{n-k+p}, \overline{\bf y}_{-m}^{n-k+p-1}, {\bf w}_{-m}^{n-k+p-1}, X_{-m}=x \right)\\
& \geq \omega(\overline{\bf y}_{n-k}^{n-k+p-1}, {\bf w}_{n-k}^{n-k+p-1}) \breve \mu_{k,\theta}(\overline{\bf y}_{-m}^{n-k+p-1},{\bf w}_{-m}^{n-k+p-1}, X_{-m}=x,A),
\end{align*}
with the same $\omega(\overline{\bf y}_{n-k}^{n-k+p-1}, {\bf w}_{n-k}^{n-k+p-1})$ as in Lemma \ref{lemma_dmr_9}. Therefore, the stated result follows from a similar argument to the proof of part (a).
\end{proof}

The following lemma is used in the proof of Lemmas \ref{lemma_psi_bound} and \ref{lemma_gamma_cgce}. This lemma provides the bounds on the difference in the conditional expectations of $\phi_{\theta t}^j=\phi^j(\theta,\overline{\bf Z}_{t-1}^t,W_t)$ when the conditioning sets are different. Define $\Omega_{\ell,k}:=\prod_{i=1}^{\lfloor (\ell-k)/p \rfloor} ( 1-\omega(\overline{\bf V}_{k+pi-p}^{k+pi-1}))$ and $\tilde \Omega_{\ell,k}:=\prod_{i=1}^{\lfloor (k-\ell)/p \rfloor} ( 1-\omega(\overline{\bf V}_{k -1 -pi+1}^{k -1 -pi+p}))$ with defining $\prod_{i=a}^b x_i:=1$ if $b<a$, where $\omega(\cdot)$ is defined in Lemma \ref{lemma_dmr_1} and $\overline{\bf V}_{a}^b:=(\overline{\bf{Y}}_{a}^b,{\bf W}_{a}^b)$.

\begin{lemma}\label{lemma_ijl}
Assume Assumptions \ref{assn_a1}--\ref{assn_distn}. Then, for all $m'\geq m\geq 0$, all $-m<s \leq t\leq n$, all $\theta\in G$, and all $x, x'\in \mathcal{X}$ and $j=1,2$,
\begin{align*}
(a) &\quad |{\mathbb{E}}_\theta[\phi_{\theta t}^j|\overline{\bf V}_{-m}^n,X_{-m}=x]-\mathbb{E}_\theta[\phi_{\theta t}^j|\overline{\bf V}_{-m}^n]| \leq 2 \Omega_{t-1,-m} |\phi_t^j|_{\infty},\\
(b) &\quad |\mathbb{E}_\theta[\phi_{\theta t}^j|\overline{\bf V}_{-m}^n]-\mathbb{E}_\theta[\phi_{\theta t}^j|\overline{\bf V}_{-m'}^n]| \leq 2 \Omega_{t-1,-m} |\phi_t^j|_{\infty},\\
(c) &\quad |\mathbb{E}_\theta[\phi_{\theta t}^j|\overline{\bf V}_{-m}^n, X_{-m}=x]-\mathbb{E}_\theta[\phi_{\theta t}^j|\overline{\bf V}_{-m'}^n, X_{-m'}=x']| \leq 2 \Omega_{t-1,-m} |\phi_t^j|_{\infty},\\
(d) &\quad |\mathbb{E}_\theta[\phi_{\theta t}^j|\overline{\bf V}_{-m}^n, X_{-m}=x]-\mathbb{E}_\theta[\phi_{\theta t}^j|\overline{\bf V}_{-m}^{n-1},X_{-m}=x]| \leq 2 \tilde\Omega_{t,n-1} |\phi_t^j|_{\infty}, \\
(e) &\quad |\mathbb{E}_\theta[\phi_{\theta t}^j|\overline{\bf V}_{-m}^n]-\mathbb{E}_\theta[\phi_{\theta t}^j|\overline{\bf V}_{-m}^{n-1}]| \leq 2 \tilde\Omega_{t,n-1} |\phi_t^j|_{\infty}, 
\end{align*}
and
\begin{align*}
(f) & \quad |\overline{\mathrm{cov}}_\theta[\phi_{\theta t},\phi_{\theta s}|\overline{\bf V}_{-m}^n]|\leq 2 \Omega_{t-1,s} |\phi_t|_{\infty}|\phi_s|_{\infty}, \\
(g) & \quad |\overline{\mathrm{cov}}_\theta[\phi_{\theta t},\phi_{\theta s}|\overline{\bf V}_{-m}^n,X_{-m}=x]|\leq 2 \Omega_{t-1,s} |\phi_t|_{\infty}|\phi_s|_{\infty}\\
(h) & \quad |\overline{\mathrm{cov}}_\theta[\phi_{\theta t},\phi_{\theta s}|\overline{\bf V}_{-m}^n,X_{-m}=x]-\overline{\mathrm{cov}}_\theta[\phi_{\theta t},\phi_{\theta s}|\overline{\bf V}_{-m}^n]| \leq 6 \Omega_{s-1,-m} |\phi_t|_{\infty}|\phi_s|_{\infty},\\
(i) & \quad |\overline{\mathrm{cov}}_\theta[\phi_{\theta t},\phi_{\theta s}|\overline{\bf V}_{-m}^n,X_{-m}=x]-\overline{\mathrm{cov}}_\theta[\phi_{\theta t},\phi_{\theta s}|\overline{\bf V}_{-m'}^n,X_{-m'}=x']| \leq 6 \Omega_{s-1,-m} |\phi_t|_{\infty}|\phi_s|_{\infty},\\
(j) & \quad |\overline{\mathrm{cov}}_\theta[\phi_{\theta t},\phi_{\theta s}|\overline{\bf V}_{-m}^n]-\overline{\mathrm{cov}}_\theta[\phi_{\theta t},\phi_{\theta s}|\overline{\bf V}_{-m}^{n-1}]| \leq 6 \tilde\Omega_{t,n-1} |\phi_t|_{\infty}|\phi_s|_{\infty},\\
(k) & \quad |\overline{\mathrm{cov}}_\theta[\phi_{\theta t},\phi_{\theta s}|\overline{\bf V}_{-m}^n,X_{-m}=x]-\overline{\mathrm{cov}}_\theta[\varphi_{\theta t},\phi_{\theta s}|\overline{\bf V}_{-m}^{n-1},X_{-m}=x]| \leq 6 \tilde\Omega_{t,n-1} |\phi_t|_{\infty}|\phi_s|_{\infty}.
\end{align*}
\end{lemma}
\begin{proof}[Proof of Lemma \ref{lemma_ijl}]
To prove parts (a)--(c), we first show that, for all $-m \leq k \leq t-1$, all probability measures $\mu_1$ and $\mu_2$ on $\mathcal{B}(\mathcal{X})$, and all $\overline{\bf V}_{-m}^n$,
\begin{equation} \label{xi_bound}
\begin{aligned}
&\sup_{A }\left| \int \mathbb{P}_{\theta}({\bf X}_{t-1}^{t} \in A |X_{k} =x,\overline{\bf V}_{-m}^n) \mu_1(dx)- \int \mathbb{P}_{\theta}({\bf X}_{t-1}^{t} \in A |X_{k} =x,\overline{\bf V}_{-m}^{n}) \mu_2(dx) \right| \\
& \leq \prod_{i=1}^{\lfloor(t-1-k)/p\rfloor} \left( 1-\omega(\overline{\bf V}_{k+pi-p}^{k+pi-1}) \right). 
\end{aligned}
\end{equation}
When $k=t-1$, (\ref{xi_bound}) holds trivially. When $-m \leq k<t-1$, equation (\ref{x_a}) and the Markov property of $Z_t$ imply that $\mathbb{P}_{\theta}({\bf X}_{t-1}^{t} \in A|X_{k}, \overline{\bf V}_{-m}^n) = \mathbb{P}_{\theta}({\bf X}_{t-1}^{t} \in A|X_{k}, \overline{\bf V}_{k}^n) = \int \mathbb{P}_{\theta}({\bf X}_{t-1}^{t} \in A|X_{t-1}=x_{t-1}, \overline{\bf V}_{k}^n) p_\theta(x_{t-1}|X_{k}, \overline{\bf V}_{k}^n)\mu(dx_{t-1})$. Consequently, from the property of the total variation distance, the left hand side of (\ref{xi_bound}) is bounded by
\begin{equation*}
\begin{aligned}
 \left\| \int \mathbb{P}_\theta(X_{t-1} \in \cdot |X_{k}=x, \overline{\bf V}_{k}^n)\mu_1(dx) - \int \mathbb{P}_\theta(X_{t-1} \in \cdot |X_{k}=x, \overline{\bf V}_{k}^n)\mu_2(dx) \right\|_{TV}.
\end{aligned}
\end{equation*}
This is bounded by $\prod_{i=1}^{\lfloor(t-1-k)/p\rfloor} ( 1-\omega(\overline{\bf V}_{k+pi-p}^{k+pi-1}))$ from Corollary \ref{corollary_1}, and (\ref{xi_bound}) is proven.

We proceed to show parts (a)--(c). For part (a), observe that
\begin{align}
{\mathbb{E}}_\theta[\phi_{\theta t}^j|\overline{\bf V}_{-m}^n,X_{-m} =x_{-m}] &= \int \nabla_\theta^j \log p_{\theta}(Y_t,x_t|\overline{\bf Y}_{t-1}, x_{t-1}, W_t) p_{\theta} ({\bf x}^{t}_{t-1}|\overline{\bf V}_{-m}^{n},x_{-m}) \mu^{\otimes 2}(d{\bf x}^{t}_{t-1}), \label{phi_x0} \\
\mathbb{E}_\theta[\phi_{\theta t}^j|\overline{\bf V}_{-m}^n] &= \int \nabla_\theta^j \log p_{\theta}(Y_t,x_t|\overline{\bf Y}_{t-1}, x_{t-1}, W_t) p_{\theta} ({\bf x}^{t}_{t-1}|\overline{\bf V}_{-m}^{n}) \mu^{\otimes 2}(d{\bf x}^{t}_{t-1}), \label{phi_x}
\end{align}
and $p_{\theta}({\bf x}^{t}_{t-1}|\overline{\bf V}_{-m}^{n}) = \int p_{\theta}({\bf x}^{t}_{t-1} |\overline{\bf V}_{-m}^{n},x_{-m}) p_{\theta}(x_{-m}|\overline{\bf V}_{-m}^{n}) \mu(dx_{-m})$. 
Note that, for any conditioning set $\mathcal{G}$, we have $\mathbb{P}_\theta({\bf X}_{t-1}^t|\mathcal{G})=0$ if $q_\theta(X_{t-1},X_t)=0$. Therefore, the right hand side of (\ref{phi_x0}) and (\ref{phi_x}) are written as
\[
\int \nabla_\theta^j \log g_{\theta}(Y_t|\overline{\bf Y}_{t-1}, x_{t}, W_t) p_{\theta} ({\bf x}^{t}_{t-1}|\mathcal{F}) \mu^{\otimes 2}(d{\bf x}^{t}_{t-1})+ 
\int_{\mathcal{X}_\theta^+} \nabla_\theta^j \log q_{\theta}(x_{t-1},x_t) p_{\theta} ({\bf x}^{t}_{t-1}|\mathcal{F}) \mu^{\otimes 2}(d{\bf x}^{t}_{t-1}),
\]
with $\mathcal{F}=\{\overline{\bf V}_{-m}^{n},x_{-m}\},\{\overline{\bf V}_{-m}^{n}\}$. Therefore, part (a) follows from the property of the total variation distance and setting $k=-m$ in (\ref{xi_bound}). Parts (b) and (c) are proven similarly.

Part (d) holds if we show that, for all $-m+1 \leq t \leq n$ and $\overline{\bf V}_{-m}^n$,
\begin{equation} \label{xi_bound_2}
\begin{aligned}
&\sup_{A}\left| \mathbb{P}_{\theta}({\bf X}_{t-1}^{t} \in A |X_{-m}=x, \overline{\bf V}_{-m}^n) - \mathbb{P}_{\theta}({\bf X}_{t-1}^{t} \in A |X_{-m}=x, \overline{\bf V}_{-m}^{n-1}) \right| \\
& \leq \prod_{i=1}^{\lfloor (n-1-t)/p \rfloor} \left( 1-\omega(\overline{\bf V}_{n-2 -pi+1}^{n- 2 -pi+p}) \right). 
\end{aligned}
\end{equation} 
When $t \geq n-1$, (\ref{xi_bound_2}) holds trivially. When $t \leq n-2$, observe
that the time-reversed process $\{Z_{n-k}\}_{0 \leq k\leq n+m}$ is Markov. Hence, for any $-m+1 \leq t \leq k$, we have $\mathbb{P}_{\theta}({\bf X}_{t-1}^{t} \in A|X_{-m}, \overline{\bf V}_{-m}^k) = \int \mathbb{P}_{\theta}({\bf X}_{t-1}^{t} \in A|X_{t}=x_{t}, \overline{\bf V}_{-m}^t) p_\theta(x_{t}|X_{-m}, \overline{\bf V}_{-m}^k)\mu(dx_{t})$. Therefore, (\ref{xi_bound_2}) is proven similarly to (\ref{xi_bound}) by using Lemma \ref{lemma_dmr_39}(b). Part (e) is proven similarly by using Lemma \ref{lemma_dmr_39}(a).

We proceed to show parts (f)--(k). In view of (\ref{phi_x}), part (f) holds if we show that, for all $-m < s \leq t \leq n$,
\begin{equation} \label{xcov_bound}
\begin{aligned}
& \sup_{A,B \in \mathcal{B}(\mathcal{X}^2)}\left| \mathbb{P}_{\theta}({\bf X}_{t-1}^{t} \in A, {\bf X}_{s-1}^{s} \in B |\overline{\bf V}_{-m}^n) - \mathbb{P}_{\theta}({\bf X}_{t-1}^{t} \in A|\overline{\bf V}_{-m}^n) \mathbb{P}_{\theta}({\bf X}_{s-1}^{s} \in B|\overline{\bf V}_{-m}^n) \right| \\
& \quad \leq \prod_{i=1}^{\lfloor(t-1-s)/p\rfloor} \left( 1-\omega(\overline{\bf V}_{s+pi-p}^{s+pi-1}) \right).
\end{aligned}
\end{equation}
When $s \geq t-1$, (\ref{xcov_bound}) holds trivially because $\prod_{i=1}^{j}a_i=1$ when $j <i$. When $s \leq t-2$, observe that $\mathbb{P}_{\theta}({\bf X}_{t-1}^{t} \in A, {\bf X}_{s-1}^{s} \in B |\overline{\bf V}_{-m}^n) = \int_B \mathbb{P}_{\theta}({\bf X}_{t-1}^{t} \in A | {\bf X}_{s-1}^{s} = {\bf x}_{s-1}^{s},\overline{\bf V}_{-m}^n) p_\theta ( {\bf x}_{s-1}^{s} |\overline{\bf V}_{-m}^n) \mu^{\otimes 2} (d{\bf x}_{s-1}^{s})$ and $\mathbb{P}_{\theta}({\bf X}_{s-1}^{s} \in B |\overline{\bf V}_{-m}^n) = \int_B p_\theta ( {\bf x}_{s-1}^{s} |\overline{\bf V}_{-m}^n) \mu^{\otimes 2} (d{\bf x}_{s-1}^{s})$. Hence, in view of the Markov property of $\{X_k\}$ given $\overline{\bf V}_{-m}^n$, the left hand side of (\ref{xcov_bound}) is bounded by $\sup_A \sup_{x_s \in \mathcal{X}} | \mathbb{P}_{\theta}({\bf X}_{t-1}^{t} \in A|X_s = x_s,\overline{\bf V}_{-m}^n) - \mathbb{P}_{\theta}({\bf X}_{t-1}^{t} \in A|\overline{\bf V}_{-m}^n) |$. From (\ref{xi_bound}), this is bounded by $\prod_{i=1}^{\lfloor(t-1-s)/p\rfloor} ( 1-\omega(\overline{\bf V}_{s+pi-p}^{s+pi-1}) )$, and (\ref{xcov_bound}) follows. Part (g) is proven similarly by replacing the conditioning variable $\overline{\bf V}_{-m}^n$ with $(X_{-m}, \overline{\bf V}_{-m}^n)$. Parts (h)--(k) follow from (\ref{xi_bound}), (\ref{xi_bound_2}), and the relation $|\mathrm{cov}(X,Y|\mathcal{F}_1)-\mathrm{cov}(X,Y|\mathcal{F}_2)| \leq |E(XY|\mathcal{F}_1)-E(XY|\mathcal{F}_2)|+ |E(X|\mathcal{F}_1)-E(X|\mathcal{F}_2)|E(Y|\mathcal{F}_1) +E(X|\mathcal{F}_2)|E(Y|\mathcal{F}_1)-E(Y|F_2)|$. 
\end{proof}

The following lemma corresponds to Lemma 14 of DMR and shows that $\mathbb{E}_{\theta^*}[\Psi_{0,m,x}^1(\theta)]$, $\mathbb{E}_{\theta^*}[\Psi_{0,m,x}^2(\theta)]$, and $\mathbb{E}_{\theta^*}[\Gamma_{0,m,x}(\theta)]$ are continuous in $\theta$.
\begin{lemma}\label{lemma_14_dmr}
Assume Assumptions \ref{assn_a1}--\ref{assn_nabla_moment}. Then, for $j=1,2$, all $x \in \mathcal{X}$ and $m \geq p$, the functions $\Psi_{0,m,x}^j(\theta )$ and $\Gamma_{0,m,x}(\theta )$ are continuous in $\theta \in G$ $\mathbb{P}_{\theta^*}$-a.s. In addition,
\begin{align*}
(a) & \quad \lim_{\delta \to 0} \mathbb{E}_{\theta^*} \left[ \sup_{|\theta'-\theta| \leq \delta} | \Psi_{0,m,x}^j(\theta) - \Psi_{0,m,x}^j(\theta') |^{3-j} \right] =0, \\
(b) & \quad \lim_{\delta \to 0} \mathbb{E}_{\theta^*} \left[ \sup_{|\theta'-\theta| \leq \delta} | \Gamma_{0,m,x}(\theta) - \Gamma_{0,m,x}(\theta') | \right] =0.
\end{align*}
\end{lemma}
\begin{proof}
The proof is similar to the proof of Lemma 14 in DMR. For brevity, we suppress $W_t$ and ${\bf W}^{0}_{-m}$ from $\phi^j(\theta,\overline{\bf Z}^t_{t-1},W_t)$ and the conditioning set. We prove part (a) first. Note that $\sup_{\theta \in G}\sup_{x \in \mathcal{X}}|\Psi_{0,m,x}^j(\theta)|^{3-j} \leq ( 2 \sum_{t=-m+1}^0|\phi_t^j|_{\infty})^{3-j} \in L^{1}(\mathbb{P}_{\theta^*})$. Hence, the stated result holds if, for $m \geq p$ and $-m+1 \leq t \leq 0$,
\[
\lim_{\delta \to 0} \sup_{|\theta'-\theta| \leq \delta} \left| \mathbb{E}_{\theta'}[\phi^j(\theta',\overline{\bf Z}^t_{t-1})| \overline{\bf Y}^{0}_{-m}, X_{-m}= x] - \mathbb{E}_{\theta}[\phi^j(\theta,\overline{\bf Z}^t_{t-1})| \overline{\bf Y}^{0}_{-m}, X_{-m}= x] \right| =0 \quad \mathbb{P}_{\theta^*}\text{-a.s.}
\]
Write
\begin{equation}\label{dmr_42}
\mathbb{E}_{\theta}[\phi^j(\theta,\overline{\bf Z}^t_{t-1})| \overline{\bf Y}^{0}_{-m}, X_{-m}= x] = \int \phi^j(\theta,\overline{\bf Z}^t_{t-1}) p_\theta({\bf X}^t_{t-1} = {\bf x}^t_{t-1}| \overline{\bf Y}^{0}_{-m}, X_{-m}= x) \mu^{\otimes 2}(d{\bf x}^t_{t-1}).
\end{equation}
For all ${\bf x}^t_{t-1}$ such that $p_\theta({\bf x}^t_{t-1}| \overline{\bf Y}^{0}_{-m}, X_{-m}= x) >0$, $\phi^j(\theta,{\bf x}^t_{t-1},\overline{\bf Y}^t_{t-1})$ is continuous in $\theta$ and bounded by $|\phi_t^j|_{\infty}< \infty$. Furthermore, 
\[
p_\theta({\bf X}^t_{t-1} = {\bf x}^t_{t-1}| \overline{\bf Y}^{0}_{-m}, X_{-m}= x) = \frac{p_\theta({\bf X}^t_{t-1} = {\bf x}^t_{t-1}, {\bf Y}^{0}_{-m+1}|\overline{\bf Y}_{-m}, X_{-m}= x)}{p_\theta({\bf Y}^{0}_{-m+1}|\overline{\bf Y}_{-m}, X_{-m}= x)}.
\]
Here, $p_\theta({\bf X}^t_{t-1} = {\bf x}^t_{t-1}, {\bf Y}^{0}_{-m+1}|\overline{\bf Y}_{-m}, X_{-m}= x)$ is continuous in $\theta$  (see (\ref{p_y_jm})) and bounded from above by $(\sigma_+^0)^m \prod_{i=-m+1}^0 b_+(\overline{\bf Y}_{i-1}^i)$, and $p_\theta({\bf Y}^{0}_{-m+1}|\overline{\bf Y}_{-m}, X_{-m}= x)$ is continuous in $\theta$ and bounded from below by $\sigma_-^{\lfloor m/p\rfloor} \prod_{t=-m+1}^0 \int \inf_{\theta \in G} g_\theta(Y_t|\overline{\bf Y}_{t-1},x_t ) \mu(dx_t) > 0$. Consequently,  $p_\theta({\bf X}^t_{t-1} = {\bf x}^t_{t-1}| \overline{\bf Y}^{0}_{-m}, X_{-m}= x)$ is continuous in $\theta$ and bounded from above uniformly in $\theta \in G$ $\mathbb{P}_{\theta^*}$-a.s., and the integrand on the right hand side of (\ref{dmr_42}) is continuous in $\theta$ and bounded from above uniformly in $\theta \in G$  $\mathbb{P}_{\theta^*}$-a.s. From the dominated convergence theorem, the left hand side of (\ref{dmr_42}) is continuous in $\theta$  $\mathbb{P}_{\theta^*}$-a.s, and part (a) is proven.

Part (b) holds if, for $-m+1 \leq s \leq t \leq 0$,
\begin{align*}
\lim_{\delta \to 0} \sup_{|\theta'-\theta| \leq \delta} & \left| \text{cov}_{\theta'}[\phi(\theta',\overline{\bf Z}^s_{s-1}),\phi(\theta',\overline{\bf Z}^t_{t-1})| \overline{\bf Y}^{0}_{-m}, X_{-m}= x] \right| \\
& \quad \left. - \text{cov}_{\theta}[\phi(\theta',\overline{\bf Z}^s_{s-1}),\phi(\theta,\overline{\bf Z}^t_{t-1})| \overline{\bf Y}^{0}_{-m}, X_{-m}= x] \right| =0 \quad \mathbb{P}_{\theta^*}\text{-a.s.}
\end{align*}
This holds by a similar argument to part (a), and part (b) follows.
\end{proof}

\end{appendices}

\bibliography{markov}

\end{document}